\documentclass[12pt]{article}
\usepackage[utf8]{inputenc}
\usepackage{amssymb,amsmath,amsthm,bm}
\usepackage[english]{babel}
\usepackage[normalem]{ulem}
\usepackage[scale=0.8]{geometry}
\usepackage{multicol}
\usepackage{algpseudocode}
\usepackage{mathrsfs,verbatim}
\usepackage{hyperref}
\usepackage{caption,cleveref}
\usepackage{graphicx, float}
\usepackage{color, mathtools, tikz, xcolor}
\usepackage{enumerate,enumitem}
\usepackage{tikz}
\usetikzlibrary{shapes.geometric}
\usepackage{bbm}

\usetikzlibrary{calc}

\usepackage[mathscr]{euscript}
\usepackage{appendix}
\usepackage{siunitx}
\usepackage{comment}
\usepackage{centernot}
\usepackage{graphicx}
\graphicspath{ {images/}}
\usepackage{subcaption}

\usepackage{mathdots}
\usepackage[
backend=biber,
natbib=true,
style=numeric,
isbn=false,
maxbibnames=9
]{biblatex}
\usepackage{csquotes}
\def\1{\mathbbm{1}}

\DeclareFieldFormat{title}{\mkbibquote{#1}}
\DeclareFieldFormat[misc]{date}{Preprint (#1)}
\newcommand{\s}[1]{\left\lvert #1 \right\rvert}

\newcommand{\J}{\widetilde{\mathcal{J}}}
\newcommand{\eps}{\varepsilon}
\newcommand{\vphi}{\varphi}

\newcommand{\N}{\mathbb{N}}

\newcommand{\R}{{\mathcal{R}}}

\newcommand{\Mod}[1]{\ (\mathrm{mod}\ #1)}

\renewbibmacro*{doi+eprint+url}{
	\iftoggle{bbx:url}
	{\iffieldundef{doi}{\usebibmacro{url+urldate}}{}}{}
	\newunit\newblock
	\iftoggle{bbx:eprint}
	{\usebibmacro{eprint}}
	{}
	\newunit\newblock
	\iftoggle{bbx:doi}
	{\printfield{doi}}
	{}
}

\renewbibmacro{in:}{}
\DeclareFieldFormat{pages}{#1}

\usepackage{hyperref}
\hypersetup{
  colorlinks = true,
  citecolor = green,
  urlcolor = blue,
  linkcolor=blue,
  pdfauthor = {Your name},
  pdftitle = {Title},
  linktocpage
}

\newtheorem{theorem}{Theorem}[section]
\newtheorem{lemma}[theorem]{Lemma}
\newtheorem{corollary}[theorem]{Corollary}
\newtheorem{conjecture}[theorem]{Conjecture}
\newtheorem{definition}[theorem]{Definition}
\newtheorem{claim}[theorem]{Claim}
\newtheorem{proposition}[theorem]{Proposition}
\newtheorem{fact}[theorem]{Fact}
\newenvironment{proofclaim}[1][Proof of claim]{\begin{proof}[#1]}{\end{proof}}

\numberwithin{equation}{section}
\newcommand{\EMAIL}[1]{\textit{{E-mail}}: \texttt{\href{mailto:#1}{#1}}}

\newcommand{\Thmstop}{\hglue-6pt. \kern6pt}
\title{Polynomial bounds for monochromatic tight cycle partition in $r$-edge-coloured~$K_n^{(k)}$}

\author{
Debmalya Bandyopadhyay 
\thanks{School of Mathematics, University of Birmingham, \EMAIL{dxb209@student.bham.ac.uk}} 
\and Allan Lo \thanks{School of Mathematics, University of Birmingham, \EMAIL{s.a.lo@bham.ac.uk}. }
}
 
\date{\today}

\def\db#1{}
	\renewcommand{\db}[1]{\footnote{\textbf{DB:}#1}}
\def\COMMENT#1{}
	\renewcommand{\COMMENT}[1]{\footnote{\textbf{DB:}#1}}

\begin{document}

\maketitle

\begin{abstract}
    Let $K_n^{(k)}$ be the complete $k$-graph on $n$ vertices. A $k$-uniform tight cycle is a $k$-graph with its vertices cyclically ordered so that every~$k$ consecutive vertices form an edge and any two consecutive edges share exactly~$k-1$ vertices. A result of Bustamante, Corsten, Frankl, Pokrovskiy and Skokan shows that all $r$-edge coloured $K_{n}^{(k)}$ can be partitioned into~$c_{r,k}$ vertex disjoint monochromatic tight cycles. However, the constant~$c_{r,k}$ is of tower-type. In this work, we show that~$c_{r, k}$ is a polynomial in~$r$. 
\end{abstract}

\section{Introduction} \label{Introduction}
An \emph{$r$-edge-colouring} of a graph or a $k$-uniform hypergraph is a colouring of its edges with $r$ colours. The set of colours is usually identified with the set $\{1, 2, \dots, r\}$. A \emph{monochromatic subgraph} of an edge-coloured graph is a subgraph where all the edges are assigned the same colour. On the other hand, a \emph{rainbow subgraph} is a subgraph all of whose edges have a different colour.

Lehel conjectured that for any $2$-edge-colouring of a complete graph, there exist two vertex-disjoint monochromatic cycles (one of each colour) covering all vertices. Isolated vertices and single-edges are considered to be degenerate cycles. For large~$n$, this conjecture was proved by \L uczak, R\"{o}dl and Szemer\'{e}di~\cite{MR1680072} using Szemer\'{e}di's Regularity Lemma. The bound on $n$ was improved later by Allen~\cite{MR2433934}. In 2010, Bessy and Thomass\'{e}~\cite{MR2595702} finally resolved this conjecture for all $n \in \mathbb{N}$.

When $r \ge 3$, Erd\H{o}s, Gy\'{a}rf\'{a}s and Pyber~\cite{MR1088628} proved that any~$r$-edge-coloured complete graph can be partitioned into $O(r^2\log r)$ monochromatic cycles and conjectured that~$r$ monochromatic cycles would be enough. This was one of the first instances of using the absorbing method. Gy\'{a}rf\'{a}s, Ruszink\'{o}, S\'{a}rk\"{o}zy and Szemer\'{e}di~\cite{MR2274080} improved their result and proved that $O(r\log r)$ monochromatic cycles suffice. However, Pokrovskiy~\cite{MR3194196} found a counter-example disproving the conjecture. A weaker conjecture was proposed stating that any~$r$-edge-coloured~$K_n$ contains~$r$ vertex-disjoint monochromatic cycles covering all but $c(r)$ of the vertices, where~$c(r)$ is a constant depending only on~$r$. Pokrovskiy~\cite{MR4489843} also proved that $c(3) \le 43000$. Kor\'andi, Lang, Letzter and Pokrovskiy~\cite{korandi2021minimum} determined a tight minimum degree threshold for monochromatic cycle partition of edge-coloured graphs, namely they showed that there exists a constant $c> 0$ such that any~$r$-edge-coloured graph~$G$ on~$n$ vertices with $\delta(G) \ge n/2 + cr\log n$ has a partition into~$O(r^2)$ monochromatic cycles.

A local $r$-colouring of a graph is an edge-colouring such that every vertex is incident with at most~$r$ edges of distinct  colours. S\'ark\"ozy~\cite{MR4131926} showed that any large locally $r$-coloured $K_n$ can be partitioned into $O(r\log r)$ monochromatic cycles. 

A $k$-uniform hypergraph (or $k$-graph) is a pair $H = (V(H), E(H))$ where $E(H) \subseteq \binom{V(H)}{k}$. \footnote{For a set $V$ and $k \in \mathbb{N}, \binom{V}{k}$ denotes the set of all subsets of $V$ of size $k$.}  Let $K_n^{(k)}$ denote the complete graph on $n$ vertices, where all $\binom{n}{k}$ edges are present. For positive integers $1 \le \ell < k \le n$, a \emph{$k$-uniform $\ell$-cycle} is a $k$-graph with its vertices cyclically ordered so that every edge contains~$k$ consecutive vertices and any two consecutive edges share exactly $\ell$ vertices. Note that $1$-cycles are called \emph{loose cycles} and $(k-1)$-cycles are called \emph{tight cycles}. 

Lehel's problem has been generalised for hypergraphs and studied for both tight and loose cycles. As in the case of graphs, any set of at most~$k$ vertices is considered as a degenerate cycle. Gy\'{a}rf\'{a}s and S\'{a}rk\"{o}zy~\cite{MR3035028} showed that for loose cycles, every $r$-edge-coloured $K_n^{(k)}$ can be partitioned into $C(k, r)$ vertex-disjoint monochromatic loose cycles. For sufficiently large~$n$, S\'{a}rk\"{o}zy~\cite{MR3240466} proved that $50kr\log(kr)$ loose cycles suffice. For an overview of results related to monochromatic partitions of (hyper)graphs we refer the reader to the surveys~\cite{fujita2015monochromatic} and~\cite{GyarfasSurvey}. 

In this paper, we focus on monochromatic tight cycle partition. 
For $k=3$, Bustamante, Han and Stein~\cite{Stein-Bustamante-Han} proved that any $2$-edge-coloured $K_n^{(3)}$ contains two vertex-disjoint monochromatic cycles of distinct colours covering all but at most $o(n)$ vertices. Lo and Pfenninger~\cite{MR4533706} proved the corresponding result
when $k=4$. Recently, Pfenninger~\cite{Vincent} generalised the result to all $k$ improving a previous result~\cite{MR1111111}.

For any $r, k \ge 3$, Bustamante, Corsten, Frankl, Pokrovskiy and Skokan~\cite{MR4117300} proved that every finite $r$-edge-coloured $K_n^{(k)}$ can be partitioned into at most $C'(k, r)$ monochromatic tight cycles. In fact, they proved that the result still holds if $K_n^{(k)}$ is replaced by large $k$-graphs with bounded independence number. It is used to solve a problem of
Elekes, Soukup, Soukup and Szentmikl\'ossy~\cite{Elekes}, that every $r$-edge-coloured infinite complete graph~$K_{\N}$ can be partitioned into a finite number of $k$-th powered paths\footnote{A $k$-th powered path is a sequence $v_1, v_2, \dots, v_{\ell}$ of distinct vertices such that for every distinct $i, j \in [\ell]$ such that $|i-j|\le k$, $v_iv_j$ is an edge.}. The bound $C'(k, r)$ of Bustamante, Corsten, Frankl, Pokrovskiy and Skokan~\cite{MR4117300} is achieved using the regularity method for hypergraphs, as a result of which the constant $C'(k, r)$ is a tower bound. 

In this paper we show that $C'(k, r)$ is a polynomial of~$r$ (for fixed $k$).

\begin{theorem} \label{Thm: ourthm}
  For all $r \in \mathbb{N} \text{ and } k\ge 3$, there exists an integer $n_0=n_0(r, k)$ such that for all $r$-edge-coloured $K_n^{(k)}$ with $n \ge n_0$, there exists a monochromatic tight cycle partition of $V\left(K_n^{(k)}\right)$ into at most $(2r)^{2^{k+4}} + 2^{k+8}r\log(2r)$ tight cycles.
\end{theorem}

We make no attempt to improve the coefficients. 

To prove Theorem~\ref{Thm: ourthm}, we use the absorption method motivated by the works of Erd\H{o}s, Gy\'{a}rf\'{a}s and Pyber~\cite{MR1088628} and Gy\'{a}rf\'{a}s, Ruszink\'{o}, S\'{a}rk\"{o}zy and Szemer\'{e}di~\cite{MR2274080}, as well as the connected matching method that is often credited to {\L}uczak~\cite{luczak1999r}. We first reserve a special vertex subset. Then we use a result by Allen, B\"ottcher, Cooley and Mycroft~\cite{MR3620730} (Theorem~\ref{Thm: almostcoverref}) to greedily remove monochromatic tight cycles from the rest of the hypergraph, until there is a small leftover set of vertices. We then use properties of the reserved structures to `absorb' the leftover set with a few monochromatic tight cycles. 

We now outline the layout of the paper. We prove Theorem~\ref{Thm: ourthm} through a series of reductions. In Section~\ref{Sec: proof of our theorem}, we reduce Theorem~\ref{Thm: ourthm} to the absorbing lemma, see Lemma~\ref{lma: reserved set absorbs everything}. Roughly speaking, Lemma~\ref{lma: reserved set absorbs everything} finds a vertex set~$R$ such that, given any small vertex subset~$B$, $R \cup B$ can be partitioned into few monochromatic tight cycles. The rest of this paper is focused on proving Lemma~\ref{lma: reserved set absorbs everything}. In Sections~\ref{Notation} and~\ref{section:regularity}, we introduce some basic notations and the hypergraph regularity lemma. In Section~\ref{sec: triangle cycles, larger absorption, tight cycle}, we reduce the size of the small vertex subset~$B$ needed in Lemma~\ref{lma: reserved set absorbs everything}. In Section~\ref{section: Multigraph properties}, we translate the problem into finding rainbow cycle partitions in edge-coloured multigraphs, see Lemma~\ref{lma: Lemma Rainbow Cycle System}. We prove this lemma in Sections~\ref{sec: rainbow path systems} and~\ref{Bowties}. In Section~\ref{Sec: Conclusion}, we end the paper with some concluding remarks and further directions. There will be further motivation and discussion in each section.

\section{Proof of Theorem~\ref{Thm: ourthm}} \label{Sec: proof of our theorem} 
We typically assume $n$ to be a large integer. Let $[n] = \{ 1, \dots, n\}$ and for integers $a \le b$, let $[a, b] = \{a, a+1, \dots, b\}$. We will use hierarchies in our statements. The phrase ``$a \ll b$'' means ``for every~$b > 0$, there exists $a_0 > 0$, such that for all $0 < a \leq a_0$ the following statements hold''. We implicitly assume all constants in such hierarchies are positive and if $1/m$ appears we assume~$m$ is an integer. For the rest of this paper,~$r$ and~$k$ will denote the number of colours and uniformity of the hypergraph, respectively.

We first show that one can cover most vertices of an $r$-edge-coloured~$K_n^{(k)}$ with few vertex-disjoint monochromatic tight cycles. We need the following theorem on the Tur\'an density of tight cycles.

\begin{theorem}{\rm{(Allen, B\"ottcher, Cooley and Mycroft}~\cite{MR3620730})} \label{Thm: almostcoverref}
Let $1/n \ll \delta, 1/k \le 1/3$ and $\alpha \in [0, 1]$. Let~$G$ be a $k$-graph on~$n$ vertices with $e(G)\ge (\alpha+\delta)\binom{n}{k}$. Then~$G$ contains a tight cycle of length~$\ell$ for every $\ell \le \alpha n$ that is divisible by~$k$. 
\end{theorem}

We use this theorem in the following proposition to greedily remove monochromatic tight cycles until we have a small leftover.

\begin{proposition} \label{Prp: leftover}
    Let $ 1/n \ll \eps, 1/k, 1/r$. Let $K_{n}^{(k)}$ be $r$-edge-coloured. Then all but at most~$\varepsilon n$ vertices of $K_n^{(k)}$ can be covered by $2r\log(1/\varepsilon)$ vertex-disjoint monochromatic tight cycles.
\end{proposition}
\begin{proof}
    Since $K_n^{(k)}$ is $r$-edge-coloured, there exists a monochromatic subgraph $G$ of~$K_n^{(k)}$ such that $e(G) \ge \binom{n}{k}/{r}$ edges. Theorem~\ref{Thm: almostcoverref} with $ 1/3r, 2/3r$ playing the roles of $\delta, \alpha,$ respectively, implies that there is a monochromatic tight cycle of length at least $2n/3r - k \ge n/2r$. Remove this cycle and repeat the argument. After removing $i$ tight cycles, there are at most $\left(1-\frac{1}{2r}\right)^i n \le ne^{-i/2r}$ many vertices left. By setting $i=2r\log(1/\varepsilon)$, we have at most $\varepsilon n$ vertices uncovered.   
\end{proof}

We need the following absorbing lemma which will be proved later.

\begin{lemma} \label{lma: reserved set absorbs everything}
Let $k \ge 3$ and $1/N \ll  1/r, 1/k$. Let $H$ be an $r$-edge-coloured~$K_N^{(k)}$. Then there exists a vertex set $R \subseteq V(H)$ such that for any $B \subseteq V(H) \setminus R$ with $|B| \le \left(2r\right)^{-2^{k+7}}k^{-1}N$, $H[R \cup B]$ can be partitioned into at most $2^{19}r (2r)^{5\cdot 2^{k+1}} + 3$ monochromatic tight cycles. 
\end{lemma}

We now prove Theorem~\ref{Thm: ourthm} assuming Lemma~\ref{lma: reserved set absorbs everything}.

\begin{proof}[Proof of Theorem~\ref{Thm: ourthm}]
   Let $\eps^* = (2r)^{-2^{k+7}}k^{-1}$ and $m = 2^{19}r (2r)^{5\cdot 2^{k+1}} + 3$. Choose constants $t_0, \eps_k, \psi$ such that $1/n \ll 1/t_0 \ll \eps_k \ll \psi \ll 1/r, 1/k$. By Lemma~\ref{lma: reserved set absorbs everything} with~$n$ playing the role of~$N$, there exists a vertex set $R \subseteq V(H)$ such that, for any $B \subseteq V(H)\setminus R$ with $|B| \le \eps^*n$, $H[R \cup B]$ can be partitioned into at most $m$ monochromatic tight cycles. Let $H' = H \setminus R$ and~$n' = |V(H')|$. By Proposition~\ref{Prp: leftover} with $\eps^*, n'$ playing the roles of $\eps, n$, respectively, all but at most $\eps^*n' \le \eps^* n$ vertices of~$H'$ can be covered by at most \begin{align*}
       2r\log(1/\eps^*) = 2^{k+8}r\log(2r)+2r\log(k)
   \end{align*} vertex-disjoint monochromatic tight cycles. Let~$B$ be the set of uncovered vertices, so~$|B|\le \eps^*n$.
   Thus $H[R \cup B]$ can be partitioned into at most $m$ monochromatic tight cycles. This covers up~$H$ with at most \begin{align*}
      2^{k+8}r\log(2r)+2r\log(k) + m &\le 2^{20+5\cdot 2^{k+1}}r^{5\cdot 2^{k+1}+1}+ 2^{k+8}r\log(2r)\\ &\le (2r)^{2^{k+4}} + 2^{k+8}r\log(2r)
   \end{align*} vertex-disjoint monochromatic tight cycles. This concludes the proof of the theorem.
\end{proof}

\section{Notation} \label{Notation}
We omit floors and ceilings if they do not affect the calculations. For two sets~$A$ and~$B$,~$A \Delta B$ denotes their set difference. We often write $v_1v_2\dots v_k$ for $\{v_1, \dots, v_k\}$.

Let~$G$ be a graph. The \emph{neighbourhood} of a vertex~$u$, denoted $N_G(u)$, is the set of vertices $\{v \in V(G)\setminus \{u\}: uv \in E(G)\}$. The \emph{closed neighbourhood} of a vertex~$u$ is defined as $\{u\} \cup N_G(u)$ and denoted as~$N_G[u]$. For a path $P = v_1\cdots v_{\ell}$, the \emph{internal vertices of~$P$}, denoted by $\mbox{int}(P)$, are $v_2, \dots , v_{\ell-1}$. For a vertex set $U \subseteq V(G)$, we denote the compliment of~$U$ as~$\overline{U}$. When~$G$ is a digraph, its minimum out-degree is denoted as~$\delta^+(G)$.  

Let $H$ be a $k$-graph.  We write $e(H) = |E(H)|$. For a set of vertices $U \subseteq V(H)$, $H[U]$ denotes the subgraph induced on~$U$. For $k$-graphs~$G$ and~$H$, $H \setminus G$ denotes the subgraph obtained by deleting $V(G)$ from~$H$ and $H-G$ denotes the subgraph after deleting $E(G)$ from~$E(H)$. For a vertex set $U \subseteq V(H)$, $H\setminus U = H[V(H)\setminus U]$.  The \emph{link graph} of a vertex $z$, denoted by $H(z)$, is the $(k-1)$-graph on $V(H)$ so that $S \in E(H(z))$ if and only if $S \cup z \in E(H)$. For $S \subseteq V(H)$, $N_H(S) = \{T \subseteq V(H)\setminus S: S \cup T \in E(H)\}$ and $d_H(S) = |N_H(S)|$. Denote $$\delta_{\ell}(H) = \min_{\substack{S \in \binom{V(H)}{\ell}}} d(S) \text{ and } \Delta_{\ell}(H) = \max_{S \in \binom{V(H)}{\ell}} d(S).$$ We write~$\delta$ and~$\Delta$ respectively for the above when $\ell=1$.  For vertex sets~$V$ and~$W$, we define $N_H(V, W)$ to be the set $\{e \in \binom{W}{k-|V|}\colon e\cup V \in E(H)\}$. In particular when $V = \{v\}$ and $W = V(H)$, we have $|N_H(V, W)|=d(v)$.

A set of edges $\mathscr{E}$ in a $k$-graph $H$ is called \emph{tightly connected} if for any pair of distinct edges $e, f \in \mathscr{E}$, there exists a sequence of edges $e_1,\dots, e_t \in E(H)$ such that $e_1 = e$, $e_t = f$ and for $i \in [t-1]$, $|e_i \cap e_{i+1}| = k-1$. 
A \emph{tight component} of a $k$-graph is a set of edges that is maximal with respect to this property.  Note that we treat tight components as $k$-graphs. A \emph{$k$-partite $k$-graph} $H$ has a partition of its vertex set into distinct and disjoint vertex classes $V_1,\dots, V_k$ such that for any edge $e \in E(H)$, $|e \cap V_i| = 1$ for each $i \in [k]$.  For vertex sets $X_1, \dots, X_k \subseteq V(H)$, an $X_1X_2\dots X_k$-edge is an edge $x_1\dots x_k$ where for each $i \in [k]$, we have $x_i \in X_i$. 

We work with edge-coloured multigraphs and use the following notation frequently. We denote the edge-colouring by $\phi$. Let $G$ be an edge-coloured (multi-)$k$-graph. The \emph{colour set}~$\phi(G)$ of~$G$ is the set of colours that appear in~$G$. When~$G$ is a multi-$k$-graph, for an edge~$e$,~$\phi(e)$ is known from the context. For a colour subset $C \subseteq \phi(G)$,~$G_C$ is the induced subgraph of~$G$ with edges of colours in~$C$ after removing any isolated vertices. Let $c \in \phi(G)$ and $S \subseteq V(G)$. We write $N_{c, G} (S)$ for $N_{G_c}(S)$ and $d_{c, G} (S) = |N_{c, G} (S)|$.  We say~$S$ \emph{sees a colour~$c$} in a graph~$G$, if $d_{c, G}(S) > 0$. The set of colours that are seen by~$S$ is denoted by $\phi_G (S)$.  We write $\delta_{\text{mon}}(G) = \min_{c \in \phi(G)} \delta(G_c)$.\footnote{Note that this is different from the monochromatic colour degree in literature, as~$G_c$ has no isolated vertices.} Let $V^*(G) = \{v \in V(G): |\phi_G(v)| \ge 2\}$.  We also define $$\delta_{\text{mon}}^* (G) = \min_{v \in V^*(G)}\ \min_{c \in \phi_G(v)} d_c(v).$$

An edge-coloured $k$-graph $H$ is \emph{locally $r$-edge-coloured} if any set of $k-1$ vertices see at most~$r$ colours in $H$. It is easy to check that an $r$-edge-coloured $k$-graph is also locally $r$-edge-coloured. A \emph{monochromatic tight component} is a tight component that is monochromatic, i.e. between any two edges in the monochromatic tight component there is a monochromatic tight path in a colour fixed for that component. Note that a monochromatic tight component~$T$ has colour~$\phi_H(T)$. 

We drop the subscript when the underlying $k$-graph is clear from context. Additional notation will be introduced within a section if only needed there.

We use the following three standard concentration inequalities in this work.

\begin{lemma}[Chernoff Bound c.f. {\cite[Remark 2.5]{MR1782847}}] \label{Lma: Chernoff}
    Let $0< \delta\le 3/2$ and $X\sim \rm{Bin}(n, p)$. Then $\mathbb{P}(|X - \mathbb{E}(X)| \ge \delta \mathbb{E}(X)) \le 2e^{-\frac{\delta^2 \mathbb{E}(X)}{3}}.$
\end{lemma}
Let $N, n, m$ be positive integers such that $\max\{n, m\} \le N$. The \emph{hypergeometric distribution}~$\rm{Hyp}(N, n, m)$ is the distribution of the random variable $X$ obtained by choosing a set~$\sigma$ of~$m$ elements from a set~$\tau$ of~$N$ elements, with $X = |\sigma \cap \tau_n|$ where $\tau_n$ is a random subset of~$\tau$ with~$n$ elements. We need the following concentration inequalities for hypergeometric random variables and martingales. 

\begin{lemma}[{Hoeffding's inequality c.f.~\cite[Theorem 2.10]{MR1782847}}]\label{Lma: Chernoff hyp}
    Let $0< \eps \le 3/2$ and $X\sim \rm{Hyp}(N, n, m)$. Then $\mathbb{P}[|X-\mathbb{E}[X]| \ge \varepsilon \mathbb{E}[X]] \le 2e^{-\frac{\eps^2}{3}\mathbb{E}[X]}$.
\end{lemma}

\begin{lemma}[{Azuma's inequality c.f.~\cite[Theorem 2.25]{MR1782847}}] \label{Lma: azuma}
    Let $Z_0, ..., Z_n$ be a martingale with $|Z_i - Z_{i-1}| \le c_i$ for all $i \in [n]$. Then for all $a>0$, $$\mathbb{P}[|Z_n - Z_0| \ge a] \le 2e^{-{\frac{a^2}{2\sum_{i=1}^n c_i^2}}}.$$
\end{lemma}

\section{Hypergraph regularity} \label{section:regularity}

In this section, we formulate the notion of hypergraph regularity that we use, closely following the formulation from Allen, B\"ottcher, Cooley and Mycroft~\cite{MR3620730}.
Recall that a \emph{hypergraph} $\mathcal{H}$ is an ordered pair $(V(\mathcal{H}), E(\mathcal{H}))$, where $E(\mathcal{H}) \subseteq 2^{V(\mathcal{H})}$. 
We identify the hypergraph $\mathcal{H}$ with its edge set $E(\mathcal{H})$. 
A subgraph $\mathcal{H}'$ of $\mathcal{H}$ is a hypergraph with $V(\mathcal{H}') \subseteq V(\mathcal{H})$ and $E(\mathcal{H}') \subseteq E(\mathcal{H})$. 
It is \emph{spanning} if $V(\mathcal{H}') = V(\mathcal{H})$.
For $U \subseteq V(\mathcal{H})$, we define $\mathcal{H}[U]$ to be the subgraph of $\mathcal{H}$ with $V(\mathcal{H}[U]) = U$ and $E(\mathcal{H}[U]) = \{ e \in E(\mathcal{H}) \colon e \subseteq U\}$.

A hypergraph~$\mathcal{H}$ is called a \emph{complex} if~$\mathcal{H}$ is down-closed, that is if for an edge $e \in \mathcal{H}$ and $f \subseteq e$, then $f \in \mathcal{H}$.
A \emph{$k$-complex} is a complex having only edges of size at most~$k$. 
We denote by~$\mathcal{H}^{(i)}$ the spanning subgraph of~$\mathcal{H}$ containing only the edges of size~$i$.
Let~$\mathcal{P}$ be a partition of~$V(\mathcal{H})$ into vertex classes $V_1, \dots, V_s$. Then we say that a set $S \subseteq V(\mathcal{H})$ is \emph{$\mathcal{P}$-partite} if $\s{S \cap V_i} \leq 1$ for all $i \in [s]$. 
For $\mathcal{P}' = \{V_{i_1}, \dots, V_{i_r}\} \subseteq \mathcal{P}$, we define the subgraph of~$\mathcal{H}$ induced by~$\mathcal{P}'$, denoted by~$\mathcal{H}[\mathcal{P'}]$ or $\mathcal{H}[V_{i_1}, \dots, V_{i_r}]$, to be the subgraph of $\mathcal{H}[\bigcup \mathcal{P}']$ containing only the edges that are~$\mathcal{P}'$-partite. 
The hypergraph~$\mathcal{H}$ is said to be~$\mathcal{P}$-partite if all of its edges are~$\mathcal{P}$-partite. We say that~$\mathcal{H}$ is \emph{$s$-partite} if it is~$\mathcal{P}$-partite for some partition~$\mathcal{P}$ of~$V(\mathcal{H})$ into~$s$ parts.
Let~$\mathcal{H}$ be a~$\mathcal{P}$-partite hypergraph. 
If~$X$ is a~$k$-set of vertex classes of~$\mathcal{H}$, then we write~$\mathcal{H}_X$ for the~$k$-partite subgraph of~$\mathcal{H}^{(k)}$ induced by~$\bigcup X$, whose vertex classes are the elements of~$X$. 
Moreover, we denote by~$\mathcal{H}_{X^<}$ the~$k$-partite hypergraph with $V(\mathcal{H}_{X^<}) = \bigcup X$ and $E(\mathcal{H}_{X^<}) = \bigcup_{X'\subsetneq X} \mathcal{H}_{X'}$.
In particular, if~$\mathcal{H}$ is a complex, then $\mathcal{H}_{X^<}$ is a $(k-1)$-complex because~$X$ is a set of size~$k$.

Let $i \geq 2$ and let~$\mathcal{P}_i$ be a partition of a vertex set~$V$ into~$i$ parts. Let~$H_i$ and~$H_{i-1}$ be a~$\mathcal{P}_i$-partite $i$-graph and a~$\mathcal{P}_i$-partite $(i-1)$-graph on a common vertex set~$V$, respectively. We say that a $\mathcal{P}_i$-partite $i$-set in~$V$ is \emph{supported on}~$H_{i-1}$ if it induces a copy of the complete $(i-1)$-graph~$K_i^{(i-1)}$ on~$i$ vertices in~$H_{i-1}$. We denote by~$K_i(H_{i-1})$ the~$\mathcal{P}_i$-partite $i$-graph on~$V$ whose edges are all~$\mathcal{P}_i$-partite $i$-sets contained in~$V$ which are supported on~$H_{i-1}$. Now we define the \emph{density of~$H_i$ with respect to~$H_{i-1}$} to be
\[
d(H_i \mid H_{i-1}) = \frac{\s{K_i(H_{i-1}) \cap H_i}}{\s{K_i(H_{i-1})}}
\]
if $\s{K_i(H_{i-1})} > 0$ and $d(H_i \mid H_{i-1}) = 0$ if $\s{K_i(H_{i-1})} = 0$.
So $d(H_i \mid H_{i-1})$ is the proportion of $\mathcal{P}_i$-partite copies of~$K_i^{i-1}$ in~$H_{i-1}$ which are also edges of~$H_i$. More generally, if $\mathbf{Q} = (Q_1, Q_2, \dots, Q_r)$ is a collection of~$r$ (not necessarily disjoint) subgraphs of~$H_{i-1}$, we define $K_i(\mathbf{Q}) = \bigcup_{j=1}^r K_i(Q_j)$ and 
\[
d(H_i \mid \mathbf{Q}) = \frac{\s{K_i(\mathbf{Q}) \cap H_i}}{\s{K_i(\mathbf{Q})}}
\]
if $\s{K_i(\mathbf{Q})} > 0$ and $d(H_i \mid \mathbf{Q}) = 0$ if $\s{K_i(\mathbf{Q})} = 0$.
We say that~$H_i$ is \emph{$(d_i, \eps, r)$-regular with respect to~$H_{i-1}$}, if we have $d(H_i \mid \mathbf{Q}) = d_i \pm \eps$ for every~$r$-set~$\mathbf{Q}$ of subgraphs of~$H_{i-1}$ with $\s{K_i(\mathbf{Q})} > \eps \s{K_i(H_{i-1})}$.
We say that~$H_i$ is \emph{$(\eps, r)$-regular with respect to~$H_{i-1}$} if there exists some~$d_i$ for which~$H_i$ is $(d_i, \eps, r)$-regular with respect to~$H_{i-1}$. 
Finally, given an~$i$-graph~$G$ whose vertex set contains that of~$H_{i-1}$, we say that~$G$ is \emph{$(d_i, \eps, r)$-regular with respect to~$H_{i-1}$} if the~$i$-partite subgraph of~$G$ induced by the vertex classes of~$H_{i-1}$ is $(d_i, \eps, r)$-regular with respect to~$H_{i-1}$. 
We refer to the density of this~$i$-partite subgraph of~$G$ with respect to~$H_{i-1}$ as the \emph{relative density of~$G$ with respect to~$H_{i-1}$}.

Now let $s \geq k \geq 3$ and let $\mathcal{H}$ be an~$s$-partite~$k$-complex on vertex classes $V_1, \dots, V_s$. For any set $A \subseteq [s]$, we write~$V_A$ for $\bigcup_{i \in A} V_i$. Note that, if $e \in \mathcal{H}^{(i)}$ for some $2 \leq i \leq k$, then the vertices of~$e$ induce a copy of~$K_i^{i-1}$ in $\mathcal{H}^{(i-1)}$. Therefore, for any set $A \in \binom{[s]}{i}$, the density $d(\mathcal{H}^{(i)}[V_A] \mid \mathcal{H}^{(i-1)}[V_A])$ is the proportion of `possible edges' of~$\mathcal{H}^{(i)}[V_A]$, which are indeed edges. We say that~$\mathcal{H}$ is \emph{$(d_k, \dots, d_2, \eps_k, \eps, r)$-regular} if
\begin{enumerate}[label=(\alph*)]
    \item for any $2 \leq i \leq k-1$ and any $A \in \binom{[s]}{i}$, the induced subgraph $\mathcal{H}^{(i)}[V_A]$ is $(d_i, \eps, 1)$-regular with respect to~$\mathcal{H}^{(i-1)}[V_A]$ and 
    \item for any $A \in \binom{[s]}{k}$, the induced subgraph $\mathcal{H}^{(k)}[V_A]$ is $(d_k, \eps_k, r)$-regular with respect to~$\mathcal{H}^{(k-1)}[V_A]$.
\end{enumerate}
For $\mathbf{d} = (d_k, \dots, d_2)$, we write $(\mathbf{d}, \eps_k, \eps, r)$-regular to mean $(d_k, \dots, d_2, \eps_k, \eps, r)$-regular.
We say that a $(k-1)$-complex $\mathcal{J}$ is \emph{$(t_0, t_1, \eps)$-equitable} if it has the following properties.
\begin{enumerate}[label = (\alph*)]
    \item $\mathcal{J}$ is $\mathcal{P}$-partite for some $\mathcal{P}$ which partitions $V(\mathcal{J})$ into~$t$ parts, where $t_0 \leq t \leq t_1$, of equal size. We refer to $\mathcal{P}$ as the \emph{ground partition} of $\mathcal{J}$ and to the parts of $\mathcal{P}$ as the \emph{clusters} of~$\mathcal{J}$.
    \item There exists a \emph{density vector} $\mathbf{d} = (d_{k-1}, \dots, d_2)$ such that, for each $2 \leq i \leq k-1$, we have $d_i \geq 1/t_1$ and $1/d_i \in \mathbb{N}$ and $\mathcal{J}$ is $(\mathbf{d}, \eps, \eps, 1)$-regular.
\end{enumerate}
For any $k$-set~$X$ of clusters of $\mathcal{J}$, we denote by $\hat{\mathcal{J}}_X$ the $k$-partite $(k-1)$-graph $(\mathcal{J}_{X^<})^{(k-1)}$ and call $\hat{\mathcal{J}}_X$ a \emph{polyad}. Given a $(t_0, t_1, \eps)$-equitable $(k-1)$-complex $\mathcal{J}$ and a $k$-graph~$G$ on~$V(\mathcal{J})$, we say that~$G$ is \emph{$(\eps_k, r)$-regular with respect to a $k$-set~$X$ of clusters of $\mathcal{J}$} if there exists some~$d$ such that~$G$ is $(d, \eps_k, r)$-regular with respect to the polyad $\hat{\mathcal{J}}_X$.
Moreover, we write $d_{G, \mathcal{J}}^*(X)$ for the relative density of~$G$ with respect to $\hat{\mathcal{J}}_X$; we may drop either subscript if it is clear from context. 

We can now give the crucial definition of a regular slice.
\begin{definition}[Regular slice]
Given $\eps, \eps_k > 0, r, t_0, t_1 \in \mathbb{N}$, a $k$-graph~$G$, a $(k-1)$-complex~$\mathcal{J}$ on~$V(G)$, is a \emph{$(t_0, t_1, \eps, \eps_k,r)$-regular slice} for~$G$ if $\mathcal{J}$ is $(t_0, t_1, \eps)$-equitable and~$G$ is $(\eps_k, r)$-regular with respect to all but at most $\eps_k \binom{t}{k}$ of the $k$-sets of clusters of~$\mathcal{J}$, where~$t$ is the number of clusters of $\mathcal{J}$.
\end{definition}

If we specify the density vector~$\mathbf{d}$ and the number of clusters~$t$ of an equitable complex or a regular slice, then it is not necessary to specify~$t_0$ and~$t_1$ (since the only role of these is to bound~$\mathbf{d}$ and~$t$). In this situation we write that $\mathcal{J}$ is $(\cdot, \cdot, \eps)$-equitable, or is a $(\cdot, \cdot, \eps, \eps_k, r)$-regular slice for~$G$.

Given a regular slice $\mathcal{J}$ for a $k$-graph~$G$, we define the $d$-reduced $k$-graph $\mathcal{R}_d^{\mathcal{J}}(G)$ as follows.
\begin{definition}[The $d$-reduced $k$-graph]
Let $k \geq 3$.
Let~$G$ be a $k$-graph and let $\mathcal{J}$ be~a $(t_0, t_1, \eps, \eps_k, r)$-regular slice for~$G$. Then, for~$d >0$, we define the \emph{$d$-reduced $k$-graph $\mathcal{R}_d^{\mathcal{J}}(G)$} to be the $k$-graph whose vertices are the clusters of $\mathcal{J}$ and whose edges are all $k$-sets~$X$ of clusters of $\mathcal{J}$ such that~$G$ is $(\eps_k, r)$-regular with respect to~$X$ and $d^*(X) \geq d$.
\end{definition}

We now state the statement of the Regular Slice Lemma that we need, that is a straightforward consequence of~\cite[Lemma 10]{MR3620730}.

\begin{lemma}[Regular Slice Lemma {\cite[Lemma 10]{MR3620730}}]
\label{lem:regular slice}
Let $k \geq 3$. For all~$t_0, r, s \in \mathbb{N}$, ~$\eps_k > 0$ and all functions $r' \colon \mathbb{N} \rightarrow \mathbb{N}$ and $\eps \colon \mathbb{N} \rightarrow (0,1]$, there are integers~$t_1$ and~$n_0$ such that the following holds for all $n \geq n_0$ which are divisible by~$t_1!$. 
Let~$H$ be an $r$-edge-coloured~$K_n^{(k)}$ with colour set~$[r]$. For $i\in [r]$, let $H_i$ denote the monochromatic subgraph of $H$ in colour~$i$.
Then there exists a $(k-1)$-complex~$\mathcal{J}$ on~$V(H)$ which is a $(t_0, t_1, \eps(t_1), \eps_k, r'(t_1))$-regular slice for each~$H_i$.
\end{lemma}

The following lemma which is a direct generalisation of~\cite[Lemma 12]{MR4533706} shows that the union of the corresponding reduced graphs $\bigcup_{i \in [r]} \mathcal{R}_d^\mathcal{J}(H_i)$ is almost complete.

\begin{lemma}
\label{lem:reduced graph edge count}
Let $k \geq 3$, $\eps, \eps_k > 0, r\in \mathbb{N}$ and $r' : \mathbb{N} \to \mathbb{N}$.
Let~$H$ be an $r$-edge-coloured~$K_n^{(k)}$ and for $i \in [r]$ let~$\mathcal{J}$ be a $(\cdot,\cdot,\eps, \eps_k,r')$-regular slice for each~$H_i$. Let~$t$ be the number of clusters of~$\mathcal{J}$. Then, provided that $d \leq 1/r$, 
we have $\s{\bigcup_{i \in [r]} \mathcal{R}_d^{\mathcal{J}}(H_i)} \geq (1 - r\eps_k)\binom{t}{k}$.
\end{lemma}
\begin{proof}
Since $\mathcal{J}$ is a $(\cdot,\cdot,\eps, \eps_k,r)$-regular slice for each $H_i$ there are at least $(1-r \eps_k)\binom{t}{k}$ $k$-sets~$X$ of clusters of $\mathcal{J}$ such that each $H_i$ is $(\eps_k, r)$-regular with respect to~$X$. Let~$X$ be such a $k$-set. Since the $H_i$ are edge-disjoint and $\bigcup_{i \in [r]} H_i = H$, we have $\sum_{i \in [r]} d_{H_i}^*(X) = 1$. Hence for some $i \in [r]$, $d_{H_i}^*(X) \geq 1/r$ and thus, since $d \leq 1/r$, we have $X \in \bigcup_{i \in [r]} \mathcal{R}_d^{\mathcal{J}}(H_i)$.
\end{proof}

Let $ r \in \mathbb{N}, 0 < d \le 1/r, \eps >0$ and $0< \eps_k < 1/2r $. Let $H$ be an $r$-edge-coloured~$K_n^{(k)}$ on colour set~$[r]$. Let $\mathcal{J}$ be a $(\cdot,\cdot,\eps, \eps_k,r')$-regular slice for each~$H_i$ with partition~$\mathcal{P}$ of~$V(H)$. By Lemma~\ref{lem:reduced graph edge count}, we have $\s{ \mathcal{R}_d^{\mathcal{J}} (H)} \ge (1-r\eps_k)\binom{t}{k} \ge \binom{t}{k}/2$. Given the following natural edge-colouring, we will be working with an edge-coloured reduced graph. For an edge~$X$ in~$\mathcal{R}_d^{\mathcal{J}} (H)$, let $e_{H, i} (X)$ denote the number of edges in colour~$i$ among the vertex clusters in~$H$ that are induced by vertices in~$X$. Assign each such $k$-set $X$ the colour~$i$ if $e_{H, i} (X) \ge e_H(X)/r$ (if multiple colours satisfy this, then choose one of them arbitrarily). Note that  $\mathcal{R}_d^{\mathcal{J}} (H)$ is an $r$-edge-coloured $k$-graph with at least $\binom{t}{k}/2$ edges. We write $\mathcal{R}$ for $\mathcal{R}_d^{\mathcal{J}} (H)$. For each~$v \in V(\R)$, we denote $V_v$ to be the cluster of~$\mathcal{P}$ corresponding to~$v$.

Let~$H$ be a $k$-graph. 
A \emph{fractional matching} in~$H$ is a function $\omega : E(H) \rightarrow [0, 1]$ such that for all $v \in V(H)$, $\omega(v) \coloneqq \sum_{e \in H: v \in e} \omega(e) \le 1$. 
The \emph{weight} of the fractional matching is defined to be $\sum_{e \in H} \omega(e)$. 
A fractional matching is \emph{tightly connected} if the subgraph induced by the edges~$e$ with $\omega(e)>0$ is tightly connected in~$H$. For a fractional matching~$\omega$, we write the size of $\omega$ as $|\omega| = \sum_{e \in E(G)}\omega(e)$.

We would use the following lemma to lift a tightly connected fractional matching in the reduced graph to a monochromatic tight cycle covering almost all vertices in the corresponding clusters.

\begin{lemma}[{\cite[Lemma 18]{MR4533706}}] \label{Lem: Lo-Pfe frac matching}
Let $1/n \ll 1/r', \eps \ll \eps_k, d_{k-1}, \dots, d_2$ and $\eps_k \ll \eps' \ll \psi, d_k, \beta, 1/k \leq 1/3$ and $1/n \ll 1/t$ such that~$t$ divides~$n$ and $1/d_i \in \mathbb{N}$ for all $2 \leq i \leq k-1$.
Let~$G$ be a $k$-graph on~$n$ vertices and $\mathcal{J}$ be a $(\cdot, \cdot,\eps, \eps_k, r')$-regular slice for~$G$. Further, let $\mathcal{J}$ have~$t$ clusters $V_1, \dots, V_t$ all of size~$n/t$ and density vector $\mathbf{d} = (d_{k-1}, \dots, d_2)$.
Suppose that the reduced graph $\mathcal{R}_{d_k}^{\mathcal{J}}(G)$ contains a tightly connected fractional matching~$\varphi$ with weight~$\mu$. Assume that all edges with non-zero weight have weight at least~$\beta$.
For each~$i \in [t]$, let $W_i \subseteq V_i$ be such that $\s{W_i} \geq ((1-3\eps')\varphi(V_i) + \eps')n/t$.
Then $G\left[\bigcup_{i \in [t]}W_i\right]$ contains a tight cycle of length~$\ell$ for each $\ell \leq (1-\psi)k\mu n/t$ that is divisible by~$k$.
\end{lemma}

Another tool we use in the next section is the notion of a hypergraph being $(\mu, \alpha)$-dense. For constants $\mu,\alpha > 0$, we say that a $k$-graph~$H$ on~$n$ vertices is $(\mu,\alpha)$\emph{-dense} if, for each $i \in [k-1]$, we have $d_H(S) \geq \mu \binom{n}{k-i}$ for all but at most~$\alpha\binom{n}{i}$ sets $S \in \binom{V(H)}{i}$ and $d_H(S) = 0$ for all other $S \in \binom{V(H)}{i}$. We use the following result from~\cite{MR4533706}.
\begin{proposition}[{\cite[Proposition 5]{MR4533706}}]
\label{prop:dense}
Let $1/n \ll \alpha \ll 1/k \leq 1/2$.
Let~$H$ be a $k$-graph on~$n$ vertices with $\s{H} \geq (1-\alpha)\binom{n}{k}$. Then there exists a subgraph~$H'$ of~$H$ such that $V(H') = V(H)$ and~$H'$ is $(1-2\alpha^{1/4k^2}, 2\alpha^{1/4k})$-dense. 
\end{proposition}

\section{The reduction of Lemma~\ref{lma: reserved set absorbs everything}} \label{sec: triangle cycles, larger absorption, tight cycle}
We now generalise the triangle cycle from Erd\H{o}s, Gy\'arf\'as and Pyber~\cite{MR1088628}. 
Let $t, k, m$ be positive integers such that $m=(k-1)t$. The \emph{$k$-uniform triangle cycle $T_m^{(k)}$} consists of a $k$-uniform tight cycle $a_1\cdots a_m$ and a vertex set $B = \{b_1,\dots, b_t\}$ such that $$a_{(k-1)i-(k-2)}a_{(k-1)i-(k-3)}\dots a_{(k-1)i}b_i a_{(k-1)i+1}\dots a_{(k-1)i + (k-1)}$$ is a $k$-uniform tight path for all $i \in [t]$ (when subscripts are taken modulo $m$). Moreover, there is a tight cycle on $V\left({T}_m^{(k)}\right) \setminus B'$ for all $B' \subseteq B$. Note that $\Delta\left({T}_m^{(k)}\right) = 2k$.

For a $k$-graph $H$ and $r \in \N$, the \emph{multicolour Ramsey number} $R_r(H)$ is the minimum positive integer $N$ such that any $r$-edge-coloured $K_N^{(k)}$ contains a monochromatic copy of~$H$. We use the following result of Conlon, Fox and Sudakov  on the Ramsey number of hypergraphs with bounded maximum degree, to obtain a monochromatic triangle cycle~${T}_m^{(k)}$.

\begin{theorem}[Conlon, Fox and Sudakov {\cite[Theorem 5]{MR2532871}}]\label{thm: Ramsey result}
    Let $\Delta, k, r \in \mathbb{N}$. There exists a constant $c(\Delta, k, r)$ such that if $H$ is a $k$-graph on $n$ vertices with maximum degree~$\Delta$, then $R_r(H) \le c(\Delta, k, r) n$.
\end{theorem}
Note that $c(\Delta, k, r)$ is a tower-type function in~$r$ and~$k$. Since $T_m^{(k)}$ has bounded degree, we get the following corollary.

\begin{corollary}\label{Cor: Triangle cycle exists}
  Let $m, n, r, k \in \mathbb{N}$ and $1/n \ll \alpha \ll 1/r, 1/k$ be such that $m = (k-1) \alpha n$. Then a monochromatic triangle cycle ${T}_m^{(k)}$ exists in every $r$-edge-coloured~$K_n^{(k)}$.
\end{corollary}

The main aim of this section is to reduce Lemma~\ref{lma: reserved set absorbs everything} to the following lemma.

\begin{lemma}\label{Lma: how to absorb}
    Let $k, r \in \mathbb{N}$, $\delta_0 = \left(2r\right)^{-2^{k+1}}$ and $\alpha = 2^5 rk (2r)^{2^{k+1}}$. Let $H$ be an $r$-edge-coloured~$K_n^{(k)}$ and $X, Z$ be disjoint subsets of $V(H)$ such that $|X| \ge \alpha|Z|$. Then there exists a set~$\mathcal{C}$ of vertex-disjoint monochromatic tight cycles covering~$Z$ with $|\mathcal{C}|\le 2^{18}r\delta_0^{-5}$ and $|V(\mathcal{C})\cap X| = (k-1)|Z|$. 
\end{lemma}
Suppose that Lemma~\ref{Lma: how to absorb} holds. A naive approach of proving Lemma~\ref{lma: reserved set absorbs everything} is to first reserve a monochromatic $T_m^{(k)}$ using Corollary~\ref{Cor: Triangle cycle exists}. We then apply Lemma~\ref{Lma: how to absorb} with $B$ (from $T_m^{(k)}$) playing the role of $X$ to cover the remaining vertices. However, $T_m^{(k)}$ obtained is too short (unless the number of cycles we remove using Proposition~\ref{Prp: leftover} in the proof of Theorem~\ref{Thm: ourthm} is a super polynomial of~$r$). In order to achieve this, we use the following lemma, which is motivated by Gy\'{a}rf\'{a}s, Ruszink\'{o}, S\'{a}rk\"{o}zy and Szemer\'{e}di~\cite{MR2274080}.

\begin{lemma} \label{Lma: larger absorption}
   Let $1/n \ll \psi \ll 1/r, 1/k$. Let $\delta_0 = \left(2r\right)^{-2^{k+1}} \text{ and } \eps^* = (2r)^{-2^{k+7}}/k$. Let~$H$ be an $r$-edge-coloured $K_n^{(k)}$.
   Then there exists a vertex set $U^*$ with $|U^*| \le 3n/4$ such that, for any vertex subset $B^*$ of $V(H)\setminus U^*$ with $|B^*| = 4 \eps^* n$, $H[B^* \cup U^*]$ can be covered by at most $2^{18}r\delta_0^{-5}+1$ vertex-disjoint monochromatic tight cycles and at most $\psi n$ isolated vertices. 
\end{lemma}
This lemma will be proved in Section~\ref{subsec: matchings}.
Observe that this lemma significantly reduces the size of the leftover vertex set. However, we would require precisely $4\eps^*n$ leftover vertices, which may not be guaranteed by Proposition~\ref{Prp: leftover}. We deal with it using the following lemma.

\begin{lemma}\label{Lem: tight cycle reservoir}
Let $1/n \ll \psi \ll 1/r, 1/k$.  
Let $H$ be an $r$-edge-coloured~$K_n^{(k)}$. Then there exists~a vertex subset $W \subseteq V(H)$ such that $n/5r \le |W| \le n/4r$ and $H[W]$ contains a monochromatic tight cycle of length $\ell$ for all $\ell \le (1-\psi)|W|$ with $\ell \equiv 0 \Mod{k}$. 
\end{lemma}

\begin{proof}
    Let $d_k = 1/r$, $$1/n \ll1/t_1\ll 1/t_0,$$ $$1/n \ll 1/r', \tilde{\eps} \ll \eps_k, d_{k-1}, \dots, d_2,$$  $$\eps_k \ll \eps' \ll  \psi, \beta \ll 1/r, 1/k.$$ Let~$H'$ be an induced subgraph of~$H$ such that $|V(H')| \equiv 0 \Mod{t_1!}$ and $|V(H\setminus H')| < t_1!$. 
By Lemma~\ref{lem:regular slice} with~$H'$ playing the role of $H$, there exists a $(\cdot, \cdot, \tilde{\eps},\eps_k, r')$-regular slice~$\J$ with partition $\mathcal{P}$ of~$V(H')$. Let $m = |V|$ for $V \in \mathcal{P}$ and $t=|\mathcal{P}|$. Note that $mt \ge n - t_1!.$

    Let $\R = \R_{d_k}^{\J}(H')$. 
     Recall that for each $x \in V(\R)$, $V_x$ denotes the corresponding cluster in~$\mathcal{P}$. By Lemma~\ref{lem:reduced graph edge count} (with $d = d_k$),~$\mathcal{R}$ contains at least $(1-r\eps_k)\binom{t}{k} \ge \binom{t}{k}/2$ edges.
There exists a colour~$j \in [r]$ such that $\s{E(\mathcal{R}_{ j})} \ge \binom{t}{k}/2r$. By Theorem~\ref{Thm: almostcoverref} with $\delta= \alpha = 1/4r$, $\mathcal{R}_{ j}$ contains~a monochromatic tight cycle $C = v_1v_2\dots v_{\ell}$ with $t/5r \le \ell \le t/4r$ and ${\ell \equiv 0\Mod{k}}$. Set $W = \bigcup_{i \in [\ell]}V_{v_i}$. Let $\vphi$ be a fractional matching on $C$ such that for any edge $e \in E(C)$, we have $\vphi(e) = 1/k$. Let~$\mu$ be the total weight of~$\vphi$, so $\mu = \sum_{i \in [\ell]}\vphi(e) = \ell/k$. For $i \in [\ell]$, note that $\vphi(v_i) = 1$ and $|V_{v_i}| = m \ge ((1-3\eps')\vphi(v_i)+\eps')m.$ Therefore, by Lemma~\ref{Lem: Lo-Pfe frac matching} with~$H$ playing the role of~$G$, we have that $H[W]$ contains a tight cycle of each length upto $(1-\psi)k\mu m = (1-\psi)\ell m = (1-\psi)|W|$ that is divisible by~$k$, as required. 
\end{proof}

\subsection{Proof of Lemma~\ref{lma: reserved set absorbs everything} assuming Lemmas~\ref{Lma: how to absorb} and~\ref{Lma: larger absorption}} \label{subsec: proof of absorbing everything}

We now prove Lemma~\ref{lma: reserved set absorbs everything} assuming Lemmas~\ref{Lma: how to absorb} and~\ref{Lma: larger absorption}.

\begin{proof}[Proof of Lemma~\ref{lma: reserved set absorbs everything}]
Set \begin{equation}\label{eqn: triangle cycle constants}
     \delta_0 = \left(2r\right)^{-2^{k+1}}, \hspace{2mm} \eps^* = (2r)^{-2^{k+7}}/k, \hspace{2mm} \alpha = 2^5 rk (2r)^{2^{k+1}}, \hspace{1mm} \alpha_{\Delta}=2\alpha \psi  \text{ and } \tilde{m} = (k-1)\alpha_{\Delta}N.
    \end{equation}
By Corollary~\ref{Cor: Triangle cycle exists} with $\alpha_{\Delta}, \tilde{m}$ playing the roles of $\alpha, m$, respectively,~$H$ contains a monochromatic $k$-uniform triangle cycle~${T}_{\tilde{m}}^{(k)}$. Note that ${T}_{\tilde{m}}^{(k)}$ contains a vertex-set $\tilde{B}$ with  \begin{equation}\label{eqn: size of triangle set reservoir}
    |\tilde{B}| = \tilde{m}/(k-1) = \alpha_{\Delta}N
\end{equation} such that for any $\tilde{B}' \subseteq \tilde{B}$, there is a  monochromatic tight cycle on $V\left({T}_{\tilde{m}}^{(k)} \setminus \tilde{B}'\right)$. Note that $\left|V\left({T}_{\tilde{m}}^{(k)}\right)\right| = \tilde{m}+\alpha_{\Delta}N =  k\alpha_{\Delta}N$.  Let $H_1 = H \setminus {T}_{\tilde{m}}^{(k)}$ and
    \begin{equation} \label{eqn: relation between N and n}
        n=|V(H_1)|=N- k\alpha_{\Delta} N\ge N/2.
    \end{equation}

Lemma~\ref{Lma: larger absorption} with $H_1$ playing the role of~$H$ implies that there exists a vertex set $U^* \subseteq V(H_1)$ with $|U^*| \le 3n/4$ such that, for any $B^* \subseteq V(H_1) \setminus U^*$ with $|B^*|=4\eps^* n$, $H[B^* \cup U^*]$ can be covered by at most $2^{18}r\delta_0^{-5}+1$ vertex-disjoint monochromatic tight cycles and at most $\psi  n$ isolated vertices.  Let $H_2 = H_1\setminus U^*$ and $n'= |V(H_2)| = n - |U^*| \ge n-3n/4 = n/4$.
    
By Lemma~\ref{Lem: tight cycle reservoir} with $H_2, n'$ playing the role of~$H, n$, respectively, there exists a vertex subset $W \subseteq V(H_2)$ with $$n/20r \le n'/5r \le |W| \le n'/4r \le n/4r$$ such that $H[W]$ contains a monochromatic tight cycle of length~$\ell$ for all $\ell \le (1-\psi)|W|$ with $\ell \equiv 0 \Mod{k}$.

 Set $R = V\left({T}_{\tilde{m}}^{(k)}\right) \cup W \cup U^*$. We now show that $R$ has the desired properties. Let $B$ be a subset of $V(H)\setminus R$ with $$|B| \le \eps^*N\stackrel{\eqref{eqn: relation between N and n}}{\le} 2\eps^*n.$$
Let $\ell$ be the largest integer such that $\ell \equiv 0 \Mod{k}$ and $|B|+|W|-\ell \ge 2\eps^* n$. Since $\psi \ll \eps^*$, $\ell \le (1-\psi)|W|$ and therefore there exists a monochromatic tight cycle~$C_1$ in $H[W]$ of length~$\ell$. Let $B^+$ and $L_1$ be disjoint vertex sets such that \begin{equation} \label{eqn: upper bd on L_1}
    B^+ \cup L_1 = B \cup (W \setminus V(C_1)), \hspace{2mm} |B^+| = 4\eps^* n \text{ and } |L_1| < k.
\end{equation}
 Thus $H[U^* \cup B^+]$ can be covered by a set $\mathcal{C}$ of monochromatic tight cycles such that $|\mathcal{C}| \le 2^{18}r\delta_0^{-5}+1$ and 
\begin{equation} \label{eqn: upper bound on L_2}
    |U^*\setminus V(\mathcal{C})| \le \psi n\le \psi N.
\end{equation} 
Let $L = L_1 \cup (U^*\setminus V(\mathcal{C}))$. We deduce that $$|L| = |L_1|+|U^*\setminus V(\mathcal{C})| \stackrel{\eqref{eqn: upper bd on L_1},~\eqref{eqn: upper bound on L_2}}{\le} k+ \psi N  \stackrel{\eqref{eqn: triangle cycle constants}}{\le} \alpha_{\Delta}N/\alpha \stackrel{\eqref{eqn: size of triangle set reservoir}}{=} |\tilde{B}|/\alpha.$$
By Lemma~\ref{Lma: how to absorb} with $\tilde{B}, L$ playing the roles of $X, Z$, respectively, $H[\tilde{B} \cup L]$ contains a set~$\mathcal{C}'$ of vertex-disjoint monochromatic tight cycles such that $L \subseteq V(\mathcal{C}')$ and $|\mathcal{C}'| \le  2^{18}r\delta_0^{-5}$. By the property of~${T}_{\tilde{m}}^{(k)}$, there exists a monochromatic tight cycle $C_2$ with vertex set $V\left({T}_{\tilde{m}}^{(k)}\right) \setminus (\tilde{B} \cap V(\mathcal{C}'))$.
Thus, we have partitioned $R \cup B$ into monochromatic cycles, namely $\mathcal{C} \cup \mathcal{C}' \cup \{C_1, C_2\}$. The total number of cycles created is at most $2^{19}r\delta_0^{-5}+3 = 2^{19}r (2r)^{5\cdot 2^{k+1}} + 3$. This completes the proof of the lemma.
\end{proof}

\subsection{Proof of Lemma~\ref{Lma: larger absorption}}\label{subsec: matchings}
The proof of Lemma~\ref{Lma: larger absorption} is motivated by Gy\'{a}rf\'{a}s, Ruszink\'{o}, S\'{a}rk\"{o}zy and Szemer\'{e}di~\cite{MR2274080}. 
Let $t \in \N$. A \emph{$t$-half dense matching} in a $k$-partite $k$-graph $H$ with vertex classes $X_1,\dots, X_k$ is~a matching $M = \{x_{i, 1}x_{i, 2}\dots x_{i, k}: i \in [\ell]\}$ where for all $i \in [\ell]$, $x_{i, j} \in X_j$ and for any vertex $x_1 \in V(M) \cap X_1$, we have $$|\{j \in [\ell] : x_1 x_{j, 2}\dots x_{j, k} \in E(H)\}| \ge t.$$
We say a matching $M$ is \emph{$t$-semi-dense} if for all $i \in [\ell]$, we have $$|\{j \in [\ell]: x_{i, 1} x_{i, 2}\dots x_{i, k-1} x_{j, k} \in E(H)\}|\ge t.$$
Note that the order of $X_1,\dots X_k$ matters, but it will be clear from context. When $k=2$, they are both equivalent to every vertex $x \in X_1$ having at least~$t$ neighbours in $X_2$, leading to the following simple fact.
\begin{fact}\label{fct: semi dense iff half dense}
 A matching in a $2$-graph is $t$-semi-dense if and only if it is $t$-half-dense.   
\end{fact}

We need the following result by Gy\'{a}rf\'{a}s, Ruszink\'{o}, S\'{a}rk\"{o}zy and Szemer\'{e}di~\cite{MR2274080}.

\begin{lemma}[{\cite[Lemma 4]{MR2274080}}]\label{lma: average degree to half dense matching}
    Every graph of average degree $8k$ has a connected $k$-half dense matching.
\end{lemma}
Recall that a $k$-graph $H$ is locally $r$-edge-coloured if any set of $k-1$ vertices sees at most~$r$ colours in $H$. 
We use the following result by Gy\'{a}rf\'{a}s, Lehel, Ne\v{s}et\v{r}il,
              R\"{o}dl, Schelp and Tuza~\cite{MR0904401}.
\begin{lemma}[{\cite[Corollary 3]{MR0904401}}]\label{Cor: local colouring many monochromatic edges}
    Let $G$ be a graph with average degree $d$ that is locally-$r$-edge-coloured. Then there is a monochromatic subgraph $G'$ such that $|E(G')| \ge d^2/2r^2$.  
\end{lemma}
For a $k$-graph $H$, we denote by $\partial H$ the $(k-1)$-graph on $V(H)$ whose edges are all $(k-1)$-tuples of vertices contained in some edge in~$H$. Therefore, $|E(\partial H)| \le k |E(H)|$. Recall that a monochromatic tight component is a tight component that is monochromatic, i.e. there is a sequence of edges of the same colour joining any two edges in the monochromatic tight component, such that any two consecutive edges in the sequence share~$k-1$ vertices.

The following lemma shows that one can find a linear-sized semi-dense matching in a monochromatic tight component of a locally~$r$-edge-coloured almost complete $k$-graph.

\begin{lemma}\label{lma: find a semi dense matching}
   Let $0 < 1/t \ll \eps_k \ll \eps_{k-1}\ll \dots \ll \eps_2 \ll 1/r, 1/k$ and let $\delta(r, k) = \left( 2^{2^k+3k-5}r^{2^k - 2}\right)^{-1}$. Let $R$ be a locally $r$-edge-coloured $k$-graph on $t$ vertices with $|E(R)| \ge (1-\eps_k)\binom{t}{k}$. Then there exists a monochromatic tight component of $R$ containing~a $\delta(r, k)t$-semi-dense matching. 
\end{lemma}
\begin{proof}

    We prove the lemma by induction on $k$.  Note that 
    \begin{equation} \label{eqn: properties of delta}
        \delta(r, 2) = \left(2^5 r^2\right)^{-1} \text{ and  } \delta(r, k+1) = \delta(2r^2, k)/2^5r^2.
    \end{equation}
Suppose that $k=2$. The average degree of $R$ is $$ \frac{2|E(R)|}{t} \ge \frac{2(1-\eps_2)\binom{t}{2}}{t} \ge \frac{t}{2}.$$ 
By Lemma~\ref{Cor: local colouring many monochromatic edges}, there is a monochromatic subgraph $R'$ such that $|E(R')| \ge t^2/8r^2$. 
The average degree of $R'$ is $$ \frac{2|E(R')|}{t} \ge \frac{t}{4r^2} = 8\delta(r, 2)t.$$
By Lemma~\ref{lma: average degree to half dense matching}, $R'$ contains a connected $\delta(r, 2)t$-half-dense matching. By Fact~\ref{fct: semi dense iff half dense}, we are done.
    
    Thus, we may assume $k \ge 3$. By Proposition~\ref{prop:dense} with $\eps_k, t$ playing the role of $\alpha, n$, respectively, there exists a spanning subgraph $R'$ of $R$ that is $(1-2\eps_k^{1/4k^2}, 2\eps_k^{1/4k})$-dense. Let $\mathcal{T}$ be the set of monochromatic tight components of $R'$. For $T \in \mathcal{T}$, let $\phi_{R'}(T)$ be the colour of $T$. Let \begin{equation} \label{eqn: n is t/8r}
        n = \frac{t}{8r}.
    \end{equation} Partition $V(R')$ into $V, W$ with $|V| = n$. Let $G = \left(\partial R'\right)[V]$, so $G$ is a $(k-1)$-graph. Note that $$|E(G)| \ge \binom{n}{k-1} - 2\eps_k^{1/4k^2}\binom{t}{k-1} \ge (1-\eps_{k-1})\binom{n}{k-1}.$$
    For every edge $e \in E(G)$, we have $$d_{R'}(e, W)\ge |W| - 2\eps_k^{1/4k^2}t \ge |W|/2.$$ Define an edge-colouring $\phi_G$ of~$G$ with colour set $\mathcal{T}$ so that, for an edge $e \in E(G)$, we have $$\phi_G(e)=T \text{ if } d_{T, G}(e, W) \ge |W|/2r.$$ If multiple $T \in \mathcal{T}$ satisfy this, pick one such $T$ arbitrarily.

    We now show that $G$ is locally $2r^2$-edge-coloured. Suppose for a contradiction, there exists a set $S$ of $k-2$ vertices $x_1,\dots, x_{k-2}$ in $G$ such that $|\phi_G (S)| \ge 2r^2+1$. Since $R'$ is locally $r$-edge-coloured, there exist a colour $c \in \phi(R')$, vertices $y_1,\dots, y_{2r+1}$ and distinct tight components $T_1,\dots, T_{2r+1}$ in $R'$ for which $\phi_{R'}(T_i) = c$ and $\phi_G(S \cup y_i)=T_i$ for all~$i \in [2r+1]$. Then there exist distinct $i, j \in [2r+1]$ such that $N_{T_i}(S \cup y_i, W)\cap N_{T_j}(S \cup y_j, W) \neq \emptyset$. Let $w \in N_{T_i}(S \cup y_i, W)\cap N_{T_j}(S \cup y_j, W)$. Note that $y_i x_1\dots x_{k-2}w y_j$ is a tight path in~$R'$ where both edges are coloured $c$ implying $T_i = T_j$, a contradiction.

By our induction hypothesis, $G$ contains a monochromatic $(k-1)$-uniform $\delta(2r^2, k-1)n$-semi-dense matching $M^{(k-1)}$ of size $\ell$. Note that \begin{equation*}\label{eqn: initial bound on ell}
   \delta(2r^2, k-1)n \le \ell \le n/2 = t/16r.
\end{equation*} 
Let $T_0 \in \mathcal{T}$ be the colour of $M^{(k-1)}$ in~$G$.  Let $M^{(k-1)} = \{x_{i,1}\dots x_{i, k-1}: i \in [\ell]\}$ be such that for each $i \in [\ell]$, we have $$|\{j\in [\ell]: x_{i,1}\dots x_{i, k-2}x_{j, k-1} \in E(T_0)\}| \ge \delta(2r^2, k-1)n.$$ We now extend $M^{(k-1)}$ to a $k$-uniform  $\delta(r, k)t$-semi-dense matching $M^{(k)}$ in~$T\subseteq R' \subseteq R$ using vertices from~$W$.

Let $\eta = \delta(2r^2, k-1)$ and $m = |W|$. For each $i \in [\ell]$, let $X^i = \{x_{i,1}, x_{i, 2},\dots, x_{i, k-2}\}$ and let $X_{k-1} = \{x_{i, k-1}: i \in [\ell]\}$. Delete some edges in $T_0[V(M^{(k-1)})]$ if necessary so that for each $i \in [\ell]$, we have $d_{T_0 \left[V(M^{(k-1)})\right]}(X^i, X_{k-1}) = \eta n$. We now choose vertices $w_1,\dots w_{\ell}$ from $W$ in turns such that $d_{T_0}(X^i\cup w_i, X_{k-1})\ge \delta(r, k) t$. Fix $i \in [\ell]$. For each $x' \in N_{G[V(M^{(k-1)})]}(X^i, X_{k-1}) $, we have $d_{T_0}(X^i \cup x', W) \ge m/2r$. Therefore, $T_0$ contains at least $\eta n m/2r$ many $x_{i,1}\dots x_{i, k-2} X_{k-1} W$-edges. At most $\ell \eta n \le \eta n^2/2$ such edges contain a vertex in $w_1,\dots, w_{i-1}$. Thus there exists a vertex $w_i \in W\setminus \{w_j: j \in [i-1]\}$ such that \begin{align*} \label{eqn: induction verification}
    d_{T_0}(X^i \cup w_i, X_{k-1}) &\ge \frac{\eta n m/(2r) - \eta n^2/2}{m} \ge \left(\frac{\eta n}{2r} \right)\left(1 -  \frac{rn}{m}\right) \\  &\ge \frac{\eta n}{4r} = \frac{\delta(2r^2, k-1)n}{4r} \stackrel{\eqref{eqn: properties of delta},~\eqref{eqn: n is t/8r}}{\ge}\delta(r, k)t. \end{align*}
Let $M^{(k)} = \{w_i x_{i, 1}\dots x_{i, k-2}x_{i, k-1}: i \in [\ell]\}$. Note that $M^{(k)}$ is the desired $k$-uniform semi-dense matching.
\end{proof}

We now convert the semi-dense matching into a half-dense matching. 
The following lemma finds a half-dense matching in a bipartite graph where we control the vertex set with a large minimum degree. Its proof is based on~\cite[Lemma 3]{MR2274080}. For a matching~$M$ in a graph~$G$, an \emph{$M$-augmented path} is a path in~$G$ where every alternate edge is in~$M$.
\begin{lemma} \label{Lem: bpt matching}
    Let $1/n \ll \delta \le 1$ and $G$ be a bipartite graph with vertex classes $X$ and $Y$ such that $|X|, |Y| \le n$ and $|E(G)| \ge \delta n^2$. Then $G$ contains a matching $M$ such that for all $x \in X \cap V(M)$, we have $d_{G[V(M)]}(x)  \ge \delta^2 n/8$.
\end{lemma}
\begin{proof}
   Let $X' \subseteq X$ and $Y' \subseteq Y$ be such that $\delta(G[X' \cup Y']) \ge \delta n/2$. Let $m = \min\{|X'|, |Y'| \}$, so $m \ge \delta n/2$.
    \begin{claim}
        There exist subsets $X^* \subseteq X'$ and $ Y^* \subseteq Y'$ such that $|X^*|=|Y^*|=m$ and $\delta\left({G[X^* \cup Y^*]}\right) \ge \delta^2 n/8$.
    \end{claim}
    \begin{proofclaim}
         Suppose that $|X'| = m$ (and the case $|Y'|= m$ is proved similarly).
         Let $X^*=X'$. Pick a subset $Y^*$ of $Y'$ of size $m$ uniformly at random. Note that for any $y \in Y^*$, $d(y, X^*) \ge \delta n/2$. Let $x \in X^*$. Note that $d(x, Y^*)\sim \text{Hyp}(|Y'|, m, d(x, Y'))$ and
    \begin{align*}
         \mathbb{E}d(x, Y^*) 
        = \frac{|Y^*|d(x, Y')}{|Y'|} \ge \frac{m(\delta n/2)}{n} \ge \frac{\delta^2 n}{4}.
    \end{align*}
    Applying Lemma~\ref{Lma: Chernoff hyp}, we deduce that for $x \in X^*$ $$\mathbb{P}\left(d(x, Y^*) < \frac{\delta^2 n}{8}\right) \le 2e^{-\frac{\delta^2 n}{12}}.$$ Taking a union bound over all $x \in X^*$, we have that with high probability, $\delta(G[X^* \cup Y^*]) \ge \delta^2 n/8$. Fix such $X^*$ and $Y^*$.
    \end{proofclaim}
    Let $G^* = G[X^* \cup Y^*]$.
    Pick a largest matching $M^*$ in $G^*$. If $M^*$ is spanning on $V(G^*)$, then we are done by setting $M = M^*$. Hence we may assume $Y^* \setminus V(M^*) \neq \emptyset$. Let $X_1^*$ be the set of vertices in $V(M^*) \cap X^*$ that can be reached by an $M^*$-augmented path from $X^* \setminus V(M^*)$. Observe that $E(G^*[X_1^* \cup (Y^*\setminus V(M^*))]) = \emptyset$. Otherwise  for a vertex $x \in X_1^*$, let~$xy$ be an edge in $E(G^*[X_1^* \cup (Y^*\setminus V(M^*))])$. There exists an $M^*$-augmented path~$P$ between~$x$ and $X^*\setminus V(M^*)$. But $(M^* \Delta E(P)) \cup \{xy\}$ is a larger matching, contradicting that~$M^*$ is the largest matching. 
    
    Let $Y_1^* = N_{G^*}(X_1^*) \subseteq V(M^*) \cap Y^*$, so $E(G[X_1^* \cup (Y^* \cap V(M^*))\setminus Y_1^*])=\emptyset$. 
    Since $M^*$ is maximal, $E(G^*\setminus V(M^*)) = \emptyset$. On the other hand, $E(G[X^*\setminus V(M^*), Y^*]) \ge \delta(G^*) > 0$. Hence $E(G[X^*\setminus V(M^*), Y^* \cap V(M^*)]) \neq \emptyset$ implying that $X_1^* \neq \emptyset$.
    
    Let $M = M^*[X_1^* \cup Y_1^*]$. We have $\delta_{M}(x) \ge \delta^2 n/8$ for all $x \in X_1^*$.
\end{proof}

This now lets us find a large half-dense matching in the reduced $k$-graph.

\begin{proposition}\label{prop: semi to half}
    Let $ 1/t \ll \delta \ll 1/r, 1/k$ and $R$ be a $k$-graph on $t$ vertices with a $\delta t$-semi-dense matching. Then $R$ contains a $\delta^3t/2$-half-dense matching.
\end{proposition}
\begin{proof}
    Let $M = \{v_{i, 1}v_{i, 2}\dots v_{i, k}: i \in [\ell]\}$ be a $\delta t$-semi-dense matching and let $V_k = \{v_{i, k}: i \in [\ell]\}$. Note that \begin{equation} \label{eqn: bounds on ell}
        \delta t \le \ell \le t/2.
    \end{equation}
    Let $H$ be the auxiliary bipartite graph with vertex classes $[\ell]$ and $V_k$ such that, for each~$i \in [\ell]$ and $v_{j, k} \in V_k$, we have $iv_{j, k} \in E(H)$ if and only if $v_{i, 1}v_{i, 2}\dots v_{i, k-1}v_{j, k} \in E(R)$. We deduce that
\begin{equation*}
    e(H) \ge \delta t\ell \stackrel{\eqref{eqn: bounds on ell}}{\ge} 2\delta \ell^2. 
\end{equation*}
By Lemma~\ref{Lem: bpt matching} with $H, \ell, \left(\delta/2 \right), V_k, [\ell]$ playing the roles of $G, n, \delta, X, Y$, respectively, $H$ contains a matching $M'$ such that, for all $v \in V_k \cap V(M')$, we have \begin{equation}\label{eqn: half density of the matching}
    d_{H[V(M')]}(v) \ge \delta^2 \ell/2 \stackrel{\eqref{eqn: bounds on ell}}{\ge} \delta^3 t/2.
\end{equation}
Without loss of generality, $M' = \{i v_{i, k}\}_{i \in [\ell]}$. Let $M'' = \{v_{i, k}v_{i, k-1}\dots v_{i, 1}: i \in [\ell]\}$. By~\eqref{eqn: half density of the matching}, $M''$ is the required half-dense matching.
\end{proof}

\begin{corollary} \label{Cor: half-dense from k-graph}
    Let $1/t \ll \eps_k \ll 1/r, 1/k$.  Let $R$ be an $r$-edge-coloured $k$-graph on~$t$ vertices with $|E(R)| \ge (1-\eps_k)\binom{t}{k}$. Then there exists a monochromatic tight component of $R$ containing a $\left(2^{9k}(2r)^{3\cdot 2^k} \right)^{-1}t$-half-dense matching.
\end{corollary}
\begin{proof}
    Let $1/t \ll \eps_k \ll \eps_{k-1} \ll \dots \ll \eps_2 \ll 1/r, 1/k$ and let $\delta = \left( 2^{2^k+3k-5}r^{2^k - 2}\right)^{-1}$. By Lemma~\ref{lma: find a semi dense matching} there exists a monochromatic tight component of $R$ containing a $\delta t$-semi-dense matching. By Proposition~\ref{prop: semi to half}, there is a $\delta^3 t/2$-half-dense matching in~$R$. Note that $$\frac{\delta^3 t}{2} = \frac{t}{2^{3\cdot 2^k + 9k -14}r^{3\cdot 2^k -6}} \ge \frac{t}{2^{9k}(2r)^{3\cdot 2^k}}.$$
This completes the proof of the corollary.
\end{proof}

We need the following lemma on bipartite graphs (Lemma~\ref{Lem: nice fractional matching in bpt}) for the main result in this subsection. The lemma will enable us to find a suitable fractional matching. To prove it we need the following result by Gy\'arf\'as, Ruszink\'o, S\'ark\"ozy and Szemer\'{e}di~\cite{MR2274080}.

\begin{lemma}[{\cite[Lemma 5]{MR2274080}}]\label{Lma: X+ X-}
Let $1/n \ll c\le 0.001$ and $G$ be a directed graph on~$n$ vertices with minimum out-degree $\delta^+(G) \ge cn$. Then there are subsets $Y \subseteq X \subseteq V(G)$ such that $|Y| \ge cn/2$ and for all $x \in X$ and $y \in Y$, there are at least $c^6 n$ internally vertex-disjoint paths from~$x$ to~$y$ of length at most~$c^{-3}$.  
\end{lemma}

The following lemma lets us balance weights in an appropriate fractional matching, which will be required to prove Lemma~\ref{Lma: larger absorption}.

\begin{lemma}\label{Lem: nice fractional matching in bpt}
    Let $1/n \ll c, \mu < 1$ with $c+\mu \le 1/8$. Let $G$ be a bipartite graph with vertex classes $X = \{x_i : i \in [n]\}$, $Y = \{y_i: i \in [n]\}$ and a perfect matching $M = \{x_iy_i: i \in [n]\}$. Suppose that for all $x \in X$, $d(x, Y) \ge \delta n$. Then there exist disjoint subsets $I^+$ and $I^-$ of~$[n]$ with $|I^+| = |I^-| = \delta n/4$ such that the following hold.
    Let $\omega$ be~a vertex weighting of $G$ such that we have
    \begin{align} 
   \nonumber \omega(y_i) &= 1/2 & &\text{ for all } i \in [n],  \\ \nonumber
    \omega(x_i) &\in  [1/2-c, 1/2] & &\text{ for all } i \in [n]\setminus I^+, \\  
    \omega(x_i) &\in  [1/2, 1/2+c] & &\text{ for all } i \in [n]\setminus I^- \text{ and} \label{eqn: part 1} \\   
        \sum_{i \in I^+} \left(\omega(x_i)-\omega(y_i)\right) &= \sum_{j \in I^-} \left(\omega(y_j)-\omega(x_j)\right) < \frac{\delta^9 n}{8}. \label{eqn: part 2}
    \end{align}
    Then $G$ has a fractional matching~$\omega^*$ of weight at least $\sum_{i \in [n]}\omega(x_i) -  \mu n$ such that, for each~$v \in V(G)$, we have $1/2 -c \le \omega^*(v) \le \omega(v)$ and each non-zero weighted edge has weight at least~$1/8$. 
\end{lemma}

\begin{proof}
    We start with the following claim identifying $I^+$ and $I^-$.
\begin{claim} \label{clm: I+ I-}
    There exists a subset $I$ of $[n]$ with $|I| = \delta n/2$ such that for all distinct $i, i' \in I$, there are at least $\delta^6 n$ internally vertex-disjoint $M$-augmented paths from $x_i$ to $x_{i'}$ of the form $x_i y_{i_1} x_{i_1}\dots y_{i'} x_{i'}$.
\end{claim}
\begin{proofclaim}
    Define an auxiliary digraph $H$ on $X$ such that $x_ix_j \in E(H)$ is directed from~$x_i$ to~$x_j$ if $x_i y_j \in E(G)$. Note that $\delta^+(H) \ge \delta n$. By Lemma~\ref{Lma: X+ X-}, there exists $X_2 \subseteq X_1 \subseteq X$ such that $|X_2| \ge \delta n/2$ and every $x_1 \in X_1$ has at least $\delta^6 n$ internally vertex-disjoint paths of length at most $\delta^{-3}$ to every $x_2 \in X_2$. Let $I \subseteq [n]$ with $|I|=\delta n/2$ be such that $\{x_i: i \in I\} \subseteq X_2$. Note that a path $x_{i_1}\dots x_{i_{\ell}}$ in~$H$ corresponds to an $M$-augmented path in~$G$ from~$x_{i_1}$ to~$x_{i_{\ell}}$, namely $x_{i_1}y_{i_2}x_{i_2}\dots y_{i_{\ell}}x_{i_{\ell}}$.
 \end{proofclaim}
 Let $I$ be given by Claim~\ref{clm: I+ I-}. Partition~$I$ into~$I^+$ and~$I^-$ so that $|I^+|=|I^-|= \delta n/4$.
We write $X^+ = \{x_i: i \in I^+\}$ and write $X^-, Y^+, Y^-, M^+$ and $M^-$ similarly. Let $\omega$ be a vertex weighting satisfying the assumptions of the lemma. 

Define a fractional matching $\omega_0^*$ on $G$ such that, for all $xy \in E(G)$ with $x \in X$ and~$y \in Y$,

\begin{align*}
    \omega_0^*(xy) = \begin{cases}
        \min\{\omega(x), \omega(y) \} = 1/2 & \text{ if } xy \in M\setminus M^-, \\
        \min\{\omega(x), \omega(y) \} = \omega(x) & \text{ if } xy \in M^-, \\
        0 & \text{ otherwise.}
    \end{cases}
\end{align*}
For a fractional matching $ \omega^*$ on $G$, let $$||\omega^* - \omega_0^* || = \sum_{e \in E(G)}|\omega^*(e) - \omega_0^*(e)|,$$ which denotes the sum of the edge-weight differences between $\omega^*$ and $\omega_0^*$. Let $\omega^*$ be a fractional matching on $G$ such that 
\begin{enumerate}[label = (\alph*)]
    \item \label{itm: (a)} $1/2-c \le \omega^*(v) \le \omega(v)$ for all $v \in V(G)$;
    \item \label{itm: (b)} $\omega^*(x_i y_i) \ge 1/8$ for all $i \in [\ell]$;
    \item \label{itm: (c)} $||\omega^* - \omega_0^* || \le 2\delta^{-3}(|\omega^*| - |\omega_0^*|)$;
    \item \label{itm: (d)} $\omega^*(v) = \omega(v)$ for all $v \in V(G) \setminus \left(X^+ \cup Y^-\right)$.
    
\end{enumerate}
    Such an $\omega^*$ exists by taking $\omega^* = \omega_0^*$.
    We further assume that $|\omega^*|$ is maximal. We may also assume that $|\omega^*| < \sum_{i \in [n]} \omega(x_i) - \mu n$ or else we are done. 
By~\ref{itm: (d)}, there exist $x^+ \in X^+$ and~$y^- \in Y^-$ such that 
\begin{equation} \label{eqn: deficient pair}
    \omega(x^+)-\omega^*(x^+), \hspace{2mm} \omega(y^-)-\omega^*(y^-) \ge \mu.
\end{equation}
 We aim to add weight to $x^+$ and $y^-$. Let $x^- = N_M(y^-)$. Let $E_0 = \{e \in E(G): |\omega^*(e) - \omega_0^*(e)| \ge 1/4\}$. Note that
  $$|E_0| \le \frac{||\omega^* - \omega_0^* ||}{1/4} \stackrel{\ref{itm: (c)}}{\le} 8\delta^{-3}\left(|\omega^*| - |\omega_0^*| \right) \stackrel{\ref{itm: (d)}}{\le} 8\delta^{-3}\left(\sum_{i \in I^+}\omega(x_i) - \omega(y_i)\right) \stackrel{\eqref{eqn: part 2}}{<} \delta^6 n.$$
Thus, by the definitions of $X^+$ and $X^-$, there exists an $M$-augmented path $P$ in $G$ from~$x^+$ to~$x^-$ of length at most $\delta^{-3}$ with $E(P) \cap E_0 = \emptyset$.
Without loss of generality,  let $P = x_1y_2x_2y_3\dots x_{\ell -1}y_{\ell}x_{\ell}$ where $x_1 = x^+$, $x_{\ell} = x^-$, $y_{\ell} = y^-$ and $\ell \le \delta^{-3}$. 
Define the fractional matching $\omega_1^*: E(G) \to [0, 1]$ such that, for any edge $e \in E(G)$, we have
\begin{align}\label{eqn: adjusted weights} 
    \omega_1^*(e) = \begin{cases} 
        \omega^*(e) - \mu & \text{ if } e=y_ix_i \text{ with } i \in [\ell-1]\setminus\{1\}, \\
        \omega^*(e) + \mu & \text{ if } e = x_i y_{i+1} \text{ with } i \in [\ell -1], \\
        \omega^*(e) &\text{ otherwise.}
    \end{cases}
\end{align}
Note that 
\begin{equation} \label{eqn: only mu increase}
    |\omega_1^*| = |\omega^*| +\mu
\end{equation}
and for $v \in V(G)$,
\begin{equation*}
    \omega_1^*(v) = \begin{cases}
        \omega^*(v) + \mu &\text{ if } v \in \{x^+, y^-\}, \\
        \omega^*(v) &\text{ otherwise.}
    \end{cases}
\end{equation*}
Therefore by~\eqref{eqn: deficient pair}, we have that $\omega_1^*$ satisfies~\ref{itm: (a)} and~\ref{itm: (d)}. 

To see $\omega_1^*$ satisfies~\ref{itm: (b)}, if $e \in E(G)\setminus \left(E(P)\setminus \{x_{\ell}y_{\ell}\}\right)$, then $\omega_1^*(e) = \omega^*(e)$. If $e \in E(P)\setminus \{x_{\ell}y_{\ell} \} \subseteq E(G)\setminus E_0$, then $|\omega_1^*(e) - \omega_0^*(e)| = \mu$ and so \begin{align*}
    \omega_1^*(e) \begin{cases}
        \ge \omega_0^*(e) - 1/4 - \mu \ge 1/8 &\text{ if } e \in E(M)\\
        \le \omega_0^*(e) + 1/4 + \mu \le 1/2 &\text{ if } e\notin E(M).
    \end{cases}
\end{align*}
Thus, $\omega_1^*$ indeed satisfies~\ref{itm: (b)} and moreover is a fractional matching on~$G$.

To verify that $\omega_1^*$ satisfies~\ref{itm: (c)}, recall that~$P$ has $2\ell -1 \le 2\delta^{-3}$ edges. Therefore, we deduce that
\begin{align*}
    ||\omega_1^* - \omega_0^*|| &\le ||\omega_1^* - \omega^*||+||\omega^*-\omega_0^*|| \stackrel{\mathclap{\eqref{eqn: adjusted weights},~\ref{itm: (c)}}}{\le} \mu(2\ell -1) + 2\delta^{-3}(|\omega^*| - |\omega_0^*|) \\
    &\le 2\delta^{-3}\mu + 2\delta^{-3}(|\omega^*| - |\omega_0^*|) \stackrel{\eqref{eqn: only mu increase}}{=} 2\delta^{-3}(|\omega_1^*|- |\omega_0^*|).
\end{align*}

Thus, $\omega_1^*$ satisfies~\ref{itm: (a)} to~\ref{itm: (d)} and has larger weight than $\omega^*$, a contradiction. 
\end{proof}

We now prove Lemma~\ref{Lma: larger absorption} assuming Lemma~\ref{Lma: how to absorb}. First, we find a monochromatic tightly connected half-dense matching in the reduced $k$-graph and reserve some unbalanced vertex clusters corresponding to the vertices of the matching, as the vertex set $U^*$. We then use Lemma~\ref{Lma: how to absorb} to cover $B^*$ with a set $\mathcal{C}$ of few monochromatic tight cycles using some vertices from $U^*$. Finally, we use the half-dense property of the matching and Lemma~\ref{Lem: nice fractional matching in bpt} to find one long tight cycle covering almost all the vertices in $U^* \setminus V(\mathcal{C})$.

\begin{proof}[Proof of Lemma~\ref{Lma: larger absorption}]
Let  $d_k = 1/r$, $\delta = \left(2^{9k}(2r)^{3\cdot 2^k}\right)^{-1}$, $$1/n \ll 1/t_1 \ll 1/t_0 \ll \eps_k \ll \eps', \mu \ll \psi \ll 1/r, 1/k, $$ $$1/n \ll 1/r', \Tilde{\eps} \ll \eps_k, d_2,\dots,d_{k-1}.$$

\noindent \textbf{Step 1}: Defining $U^*$.
    
Let $H'$ be an induced subgraph of $H$ such that $|V(H')| \equiv 0 \Mod{t_1!}$ and $|V(H\setminus H')| < t_1!$. 
By Lemma~\ref{lem:regular slice} with $H'$ playing the role of $H$, there exists a $(\cdot, \cdot, \tilde{\eps},\eps_k, r')$-regular slice~$\J$ with partition $\mathcal{P}$ of $V(H')$. Let $m = |V|$ for $V \in \mathcal{P}$,  $t=|\mathcal{P}|$ and $t_0 \le t \le t_1$.
    Note that \begin{equation}\label{eqn: relation between m, t, n}
     n \ge  mt \ge n - t_1! \ge n/2.
    \end{equation}
Let $\R = \R_{d_k}^{\J}(H')$. Note that $|E(\R)| \ge (1-r\eps_k)\binom{t}{k}$ by Lemma~\ref{lem:reduced graph edge count}. By Corollary~\ref{Cor: half-dense from k-graph} with~$\R, r\eps_k$ playing the roles of~$R, \eps$, respectively, there exists~a monochromatic tight component~$T$ of~$\R$ and~a $\delta t$-half dense matching~$M$ of size $\ell$, where \begin{equation} \label{eqn: upper bound on ell}
      \delta t \le  \ell \le t/k \le t/2.
    \end{equation} Let $M = \{x_iy_{i, 2}\dots y_{i, k}: i \in [\ell]\}$ and $X=\{x_i: i\in [\ell]\}$ be such that, for all $x \in X$, we have $|\{i \in [\ell]: xy_{i, 2}\dots y_{i, k} \in E(T) \}| \ge \delta t$.

    Define an auxiliary bipartite graph $G$ on $X$ and $Y = \{y_1,\dots, y_{\ell}\}$ such that $x_iy_j \in E(G)$ if and only if $x_iy_{j, 2}\dots y_{j, k} \in E(T)$. Note that $M$ corresponds to a perfect half-dense matching~$M' = \{x_iy_i: i \in [\ell]\}$ in~$G$ and for each $x\in X$, $$d_G(x, Y) \ge \delta t \stackrel{\eqref{eqn: upper bound on ell}}{\ge} 2\delta \ell.$$
    By Lemma~\ref{Lem: nice fractional matching in bpt} with $\ell, 2\delta$ playing the roles of $n, \delta$, respectively, there exist disjoint $I^+, I^- \subseteq [\ell]$ such that \begin{equation*}\label{eqn: I+ I-}
        |I^+| = |I^-|= \delta \ell/2 \stackrel{\eqref{eqn: upper bound on ell}}{\ge} \delta^2 t/2 \stackrel{\eqref{eqn: relation between m, t, n}}{\ge} \delta^2(n/2m) - \delta^2 t_1!/(2m) \ge \delta^2 n/4m.
    \end{equation*} Define $X^+ = \{x_i \in X: i \in I^+\}$ and define $X^-, Y^+, Y^-$ similarly.
Let
\begin{align}\label{eqn: fixing constants}
 \alpha = 2^5 rk(2r)^{2^{k+1}}, \, \gamma = \frac{4\eps^* \alpha n}{m|I^+|} \text{ and } \gamma_a = \frac{(k-1)\gamma}\alpha = \frac{4\eps^*(k-1) n}{m|I^+|}.  
\end{align}
Note that Lemma~\ref{Lem: nice fractional matching in bpt} also implies that if $\omega$ is a vertex-weighting satisfying~\eqref{eqn: part 1} and~\eqref{eqn: part 2} with $\gamma - \gamma_a, \ell$ playing the roles of $c, n$, respectively, then $G$ has a fractional matching $\omega^*$ with 
 $   |\omega^*| \ge \sum_{i \in [\ell]}\omega(x_i)-\mu\ell$,   
where each vertex $v$ has weight between $1/2 - \gamma+\gamma_a$ and $\omega(v)$ and furthermore each non-zero weighted edge has weight at least~$1/8$. 

For each~$v \in V(\R)$, recall that $V_v$ is the cluster of~$\mathcal{P}$ corresponding to~$v$. We pick disjoint $U_v, A_v \subseteq V_v$ such that 
\begin{align*}
    |A_v| &= \begin{cases}
         \gamma m & \text{ if } v \in X^+, \\
        0 & \text{ otherwise,}
    \end{cases}
\\
    |U_v| &= \begin{cases}
        \left(\frac{1}{2}-\gamma + \gamma_a \right)m &\text{ if } v \in X^-, \\
        \frac{m}{2} &\text{ if } v \in V(M')\setminus X^-, \\
        \eps' m &\text{ otherwise.}
    \end{cases}
\end{align*}
Let $A = \bigcup_{v \in V(\R)}A_v$, $U= \bigcup_{v \in V(\R)}U_v$ and $U^* = U \cup A$. Therefore, we have $|A| = |I^+|\gamma m = 4\alpha \eps^*n$ by~\eqref{eqn: fixing constants} and  \begin{align*}
    |U^*| &= |U|+|A| = km\ell/2  - (\gamma - \gamma_a)m|X^-| + \eps'm(t-k\ell) + \gamma m|X^+| \\
    &= km\ell/2 + \gamma_a m |X^-| + \eps'm(t-k\ell) =km\ell/2 + 4(k-1)\eps^*n + \eps' m(t-k\ell)\\ & \stackrel{\mathclap{\eqref{eqn: upper bound on ell}}}{\le} mt/2 + 4(k-1)\eps^* n +\eps'm(t-k\ell) \stackrel{\eqref{eqn: relation between m, t, n}}{\le} 3n/4. 
\end{align*}

\noindent \textbf{Step 2:} Verifying properties of $U^*$.\\
Let $B^* \subseteq V(H)\setminus U^*$ be such that $|B^*| = 4\eps^* n$.\\
\textbf{Step 2A:} Covering $B^*$.
\\
 By Lemma~\ref{Lma: how to absorb} with $A, B^*$ playing the roles of $X, Z$, respectively, there exists a set~$\mathcal{C}$ of vertex-disjoint monochromatic tight cycles covering $B^*$ with $$|\mathcal{C}| \le 2^{18}r\delta_0^{-5} \text{  and  }|V(\mathcal{C})\cap A| = 4(k-1)\eps^*n = \gamma_a m |I^+|.$$
\textbf{Step 2B:} Covering $U^* \setminus V(\mathcal{C})$.
\\
Let $A' = A\setminus V(\mathcal{C})$, so $|A'|= (\gamma - \gamma_a)|I^+|m$.
For each $x \in V(\R)$, let $A_x' = A' \cap V_x$, so $\bigcup_{x \in X^+} A_x' = A'$.
Note that \begin{equation}\label{eqn: bound on leftover in half dense matching}
    km\ell/2 \le |U \cup A'| \le km\ell/2 + \eps' mt.
\end{equation}
We define a vertex-weighting $\omega: V(M') \to [0, 1]$ to be such that for $v \in V(M')$, 
\begin{align*}
    \omega(v) = \frac{|V_v \cap (U \cup A')|}{m} = \frac{|U_v|+|A_v'|}{m} \begin{cases}
         \in [\frac{1}{2}, \frac{1}{2}+\gamma] &\text{ if } v \in X^+, \\
        = \frac{1}{2} - \gamma+\gamma_a &\text{ if } v \in X^-, \\
        = \frac{1}{2} &\text{ if } v \in V(M')\setminus (X^+ \cup X^-).
    \end{cases}
\end{align*}
Firstly, note that  
\begin{align*}
    \sum_{i \in I^+} (\omega(x_i) - \omega(y_i)) &= \sum_{i \in I^+}\left( \left( \frac{1}{2}+ \frac{|A_{x_i}'|}{m} \right)-\frac{1}{2}\right) = \frac{|A'|}{m}  \\
    &=(\gamma - \gamma_a)|I^+| = \sum_{i \in I^-}(\omega(y_i) - \omega(x_i)).
\end{align*}
Also 
\begin{align*}
    \sum_{i \in I^+} (\omega(x_i) - \omega(y_i)) &= (\gamma - \gamma_a)|I^+| \le \gamma |I^+| \\ &\stackrel{\mathclap{\eqref{eqn: fixing constants}}}{=}\ 4\eps^* \alpha n/m \stackrel{\eqref{eqn: relation between m, t, n}}{\le} 8\eps^* \alpha t \stackrel{\eqref{eqn: upper bound on ell}}{\le} 8\eps^* \alpha \ell / \delta < (2\delta)^9\ell/8.
\end{align*}
Therefore, all the conditions for Lemma~\ref{Lem: nice fractional matching in bpt} (with $2\delta$ playing the role of $\delta$) are satisfied. Furthermore, $\sum_{i \in [\ell]}\omega(x_i) = \sum_{i \in [\ell]}\omega(y_i)$. We deduce that~$G$ has a fractional matching~$\omega^*$ with \begin{equation}\label{eqn: weight of omega^*}
    |\omega^*| \ge \sum_{i \in [\ell]}\omega(x_i)-\mu\ell = \sum_{i \in [\ell]}\omega(y_i) - \mu \ell = \left(\frac{1}{2}-\mu\right)\ell, 
\end{equation} where each vertex $v$ has weight at least $1/2 - \gamma+\gamma_a$ and at most $\omega(v)$ and each non-zero weighted edge has weight at least~$1/8$.
Define a corresponding fractional matching $\omega_H^*$ on~$H$ such that for an edge~$e\in E(H)$, $\omega_H^*(e) = \omega^*(x_i y_j)$ if $e=x_i y_{j, 2}\dots y_{j, k}$ and $\omega_H^*(e) = 0$ otherwise. Note that $|\omega_H^*| = |\omega^*|$.
For any vertex $u \notin V(M)$ we have $\omega_H^*(u)= 0$, so $|V_u \cap (U \cup A')| = \eps' m = ((1-3\eps')\omega_H^*(u) +\eps')m$. For each $u \in V(M)$, we have $\omega_H^*(u) \ge 1/3$ and so $$|V_u \cap (U \cup A')| = |U_u \cup A_u'| = \omega(u)m \ge \omega_H^*(u) m \ge((1-3\eps')\omega_H^*(u)+\eps')m.$$
By Lemma~\ref{Lem: Lo-Pfe frac matching}, $H[A' \cup U]$ contains a tight cycle~$C$ of length $\ell_0$ with 
\begin{align*}
    \ell_0 &\ge (1-\psi)k|\omega_H^*|m-k \stackrel{\mathclap{\eqref{eqn: weight of omega^*}}}{\ge} (1-\psi)(1/2 - \mu)k\ell m - k \ge (1-\psi)km\ell/2 - \mu km\ell.
\end{align*}
Recall that $\mu, \eps', 1/t \ll \psi$ and $1/m \ll 1/k$. Hence, 
\begin{align*}
    |(A' \cup U)\setminus V(C)| &\stackrel{\eqref{eqn: bound on leftover in half dense matching}}{\le} \psi km\ell/2 + \mu km\ell + \eps' mt \le \psi km\ell \stackrel{\eqref{eqn: upper bound on ell}}{\le}\psi m t \stackrel{\eqref{eqn: relation between m, t, n}}{\le} \psi n
\end{align*} as required.
\end{proof}

\section{Edge-coloured Multigraphs} \label{section: Multigraph properties}
Let $H$ be a $k$-partite $k$-graph with vertex classes $X_1, \dots, X_k$ with $X_i = \{x_{i, j}: j \in [n]\}$. We denote the $2$-blowup of a $k$-edge by $K_k^{(k)} (2)$. Note that $K_k^{(k)} (2)$ is the complete $k$-partite $k$-graph with each vertex class of order $2$. We say that an edge-coloured multigraph $G$ \emph{respects} $H$ if the following hold:
\begin{itemize}
    \item the colour set of $G$ is a subset of $X_k$;
    \item $V(G) \subseteq \{v_i: i \in [n]\}$;
    \item if $\phi(v_iv_{i'}) = x_{k, t}$ for some $t \in [n]$, then $\{x_{i, j}x_{i',j}: j \in [k-1]\}$ forms a $K_{k-1}^{(k-1)}(2)$ in~$H(x_{k, t})$.
\end{itemize}
There is a natural correspondence between a tight cycle in $H$ and a rainbow cycle in $G$.
\begin{fact} \label{fct: equivalence with multigraph covering problem}
    Let $H$ be a $k$-partite $k$-graph with vertex sets $X_1,\dots, X_k$ such that for $i \in [k-1]$, $X_i = \{x_{i, j} : j \in [n] \}$. Let $G$ be an edge-coloured multigraph respecting $H$. Suppose that~$G$ contains a rainbow cycle $v_1v_2\cdots v_{\ell}$ with $\phi(v_i v_{i+1})= x_{k, i}$ for $i \in [\ell-1]$ and $\phi(v_{\ell} v_1) = x_{1, \ell}$. Then~$H$ contains a tight cycle with vertex set $\{x_{i, j}: i \in [k], j \in [\ell] \}$.
\end{fact}
Let $G$ be an edge-coloured multigraph. Recall that $\phi(G)$ is the set of colours in~$G$ and for a colour $c \in \phi(G)$, $G_c$ is the subgraph with edges only of the colour~$c$ and no isolated vertices. The monochromatic colour degree of $G$ to be $\delta_{\text{mon}}(G) = \min_{c \in \phi(G)} \delta(G_c)$. A \emph{path system} is a collection of vertex-disjoint paths. A rainbow path system is a path system such that its union (viewed as a graph) is rainbow. A (rainbow) cycle system is defined similarly. The main goal of this section is to reduce
Lemma~\ref{Lma: how to absorb} into the following lemma, which will be proved in the next section.

\begin{lemma} \label{lma: Lemma Rainbow Cycle System}
    Let  $1/n \ll \delta_0 \le 1/8$. Suppose $G$ is an edge-coloured multigraph on $n$ vertices with $\delta_{\text{mon}}(G) \ge \delta_0 n$ and $|\phi(G)| \le \delta_0 n/16$. Then there exists a rainbow cycle system $\mathcal{C}$ with $\phi(\mathcal{C}) = \phi(G)$ and $|\mathcal{C}| \le 2^{18}\delta_0^{-5}$.
\end{lemma}

To reduce Lemma~\ref{Lma: how to absorb} to Lemma~\ref{lma: Lemma Rainbow Cycle System}, we need the help of the following lemma, which will be proved in the next subsection.
\begin{lemma}\label{Lma: Revised Lemma}
  Let $H$ be an $r$-edge-coloured $K_n^{(k)}$ with disjoint vertex sets~$X$ and~$Z$ with $|X| \ge (k-1)|Z|$. Then for each $i \in [r]$, there exist an edge-coloured multigraph $G^i$ and a subgraph $H^i$ of $H$ such that
\begin{enumerate}[label=\rm{(\textbf{A\arabic*})}, ref=(\textbf{A\arabic*})]
    \item \label{itm: (A1)} $H^i$ is a monochromatic $k$-partite $k$-graph on vertex classes $X_1^i,\dots, X_k^i$ of colour $i$;
    \item \label{itm: (A2)} $\phi(G^i) = X_k^i$, $Z = \bigcup_{i \in [r]} X_k^i$ and $\bigcup_{j \in [k-1]}X_j^i \subseteq X$;
    \item \label{itm: (A3)} if $H^i$ is not empty, then $|V(G^i)| = \left(1-\frac{1}{4kr^k} \right)\frac{|X|}{r(k-1)}$, $\delta_{\text{mon}}(G^i) \ge \left( 2r\right)^{-2^k}|V(G^i)|/8$ and $G^i$ respects~$H^i$;
    \item \label{itm: (A4)} for distinct $i, j \in [r]$, we have $V(H^i) \cap V(H^j) = \emptyset$.
    
\end{enumerate}
\end{lemma}
We now prove Lemma~\ref{Lma: how to absorb} assuming Lemmas~\ref{lma: Lemma Rainbow Cycle System} and~\ref{Lma: Revised Lemma}. 

\begin{proof}[Proof of Lemma~\ref{Lma: how to absorb}]
By Lemma~\ref{Lma: Revised Lemma}, for each $i \in [r]$, there exist an edge-coloured multigraph~$G^i$ and a $k$-partite $k$-graph $H^i$ such that $G^i$ and $H^i$ satisfy $\ref{itm: (A1)}-\ref{itm: (A4)}$ and hence $G^i$ respects $H^i$. Fix $i \in [r]$. Observe that if $H^i$ is not empty, then $\delta_{\text{mon}}(G^i) \ge \delta_0 |V(G^i)|$ by~\ref{itm: (A3)}. Moreover $$\left|\phi(G^i)\right| \stackrel{\ref{itm: (A2)}}{\le} |Z| \le \frac{|X|}{\alpha} \le \frac{\delta_0 |X|}{32r(k-1)} \le \frac{\delta_0}{16}\left(1-\frac{1}{4kr^k}\right)\frac{|X|}{r(k-1)} \stackrel{\ref{itm: (A3)}}{=} \frac{\delta_0 |V(G^i)|}{16}.$$
By Lemma~\ref{lma: Lemma Rainbow Cycle System} with $G^i$ playing the role of~$G$, we deduce that $G^i$ contains a rainbow cycle system~$\mathcal{C}^i$ with $\phi(\mathcal{C}^i) = \phi(G^i) = X_k^i$ and $|\mathcal{C}^i| \le 2^{18}\delta_0^{-5}$. By Fact~\ref{fct: equivalence with multigraph covering problem} and~\ref{itm: (A1)},~$\mathcal{C}^i$ corresponds to a set~$\Hat{\mathcal{C}}^i$ of vertex-disjoint monochromatic tight cycles in~$H^i$ covering~$X_k^i$. Therefore by~\ref{itm: (A4)}, $\bigcup_{i \in [r]} \Hat{\mathcal{C}}^i$ is a set of vertex-disjoint monochromatic tight cycles covering $\bigcup_{i \in [r]}X_k^i = Z$. Note that $\left|\bigcup_{i \in [r]} \Hat{\mathcal{C}}^i \right| \le  2^{18}r\delta_0^{-5}.$ This completes the proof of the lemma.   \end{proof}

\subsection{Proof of Lemma~\ref{Lma: Revised Lemma}}
We start with the following result, which allows us to partition $X \cup Z$ into the vertex-disjoint $k$-partite $k$-graphs $H^i$ for  $i \in [r]$. Let $H$ be an $r$-edge-coloured $k$-graph. For vertex sets $X_1,\dots, X_k$, let $e_j(X_1,\dots, X_k)$ be the number of $X_1\dots X_k$-edges of colour~$j$ in~$H$. 

\begin{lemma} \label{Lma: Slicing tripartite lemma}
Let $1/N \ll 1/k, 1/r \le 1/2$. Let $H$ be an $r$-edge-coloured complete $k$-partite $k$-graph with vertex classes $X_1, X_2, \dots, X_k$ where for all $i \in [k-1]$, ${|X_i|=N}$ and ${|X_k| \le N}$. Then there exists a partition of $X_1,\dots, X_k$ into $\{X_1^j: j\in [r] \cup \{0\}  \}$, $\dots,$ $\{X_{k-1}^j: j\in [r] \cup \{0\}  \}$, $\{X_k^j: j\in [r]  \}$ such that, for each $ i \in [k-1]$, $j \in [r]$ and $x \in X_k^j$, we have $$|X_i^j|=\left(1-\frac{1}{4kr^k}\right)\left(\frac{N}{r}\right) \text{ and } e_j\left(X_1^j,\dots, X_{k-1}^j, x\right) \ge \frac{\left|X_1^j\right|\cdots\left|X_{k-1}^j\right|}{2r}.$$
\end{lemma}
    \begin{proof}
Partition $X_k$ into $X_k^1, \dots, X_k^r$ such that for all $j \in [r]$ and $x_k \in X_k^j$, $$e_j(X_1, \dots, X_{k-1}, x_k) \ge \frac{|X_1|\dots|X_{k-1}|}{r}.$$ Let $\varepsilon = (4kr^k)^{-1}$.

\begin{claim}\label{Clm: Partition Y}
    There exists a partition of $X_i$ for each $i\in [k-1]$ into $X_i^1, \dots, X_i^r$ such that for all $j\in [r]$, $\left| \left|X_i^j\right| - \frac{N}{r} \right| \le \varepsilon \frac{N}{r}$ and $$\left| e_j\left(X_1^j,\dots, X_i^j, X_{i+1},\dots, X_{k-1}, x_k\right) - \frac{e_j\left(X_1^j,\dots, X_{i-1}^j, X_i,\dots, X_k\right)}{r}  \right| < \varepsilon N^{k-1}.$$
\end{claim}
\begin{proofclaim}
We proceed by induction on $i$. We only prove the base case $i=1$ as the rest can be proved analogously.

Let $X_1 = X_1^1 \cup \dots \cup X_1^r$ be a random partition of $X_1$, where for all $x_1 \in X_1$ and $i \in [r]$,~$x_1$ is assigned to $X_1^j$ independently with probability $1/r$.  
Consider $x_i \in X_i$ for $i \in [2, k]$ and $j \in [r]$. Note that $\mathbb{E}\left(e_j(X_1^j,x_2,\dots, x_k)\right) = \frac{e_j(X_1, x_2,\dots, x_k)}{r}$. If $e_j(X_1, x_2, \dots, x_k) \le \varepsilon N$, then clearly $e_j(X_1^i,x_2,\dots, x_k) \le \varepsilon N$. If $e_j(X_1^j,x_2, \dots, x_k) \ge \varepsilon N$, then by Lemma~\ref{Lma: Chernoff}, we have $$\mathbb{P}\left(\left|e_j(X_1^j,x_2,\dots, x_k)-\frac{e_j(X_1, x_2, \dots, x_k)}{r}\right| > \varepsilon\frac{e_j(X_1, x_2,\dots, x_k)}{r} \right) \le 2e^{-\frac{\varepsilon^3 N}{3r}}.$$
Note that $\mathbb{E}|X_1^j| = N/r$. By Lemma~\ref{Lma: Chernoff}, we have $$\mathbb{P}\left(\left| |X_1^j| - \frac{N}{r}\right| > \varepsilon \frac{N}{r} \right) < 2e^{- \frac{\varepsilon^2 N}{3r}}.$$ 
Hence by the union bound, with positive probability for all $j \in [r]$, $i\in [2, k]$ and ${x_i \in X_i}$, we have $\left| |X_1^j| - \frac{N}{r}\right| < \varepsilon \frac{N}{r}$ and $\left| e_j(X_1^j, x_2,\dots,x_k) - \frac{e_j(X_1, x_2, \dots x_k)}{r}\right| \le \varepsilon N$. 
Fix such a partition. We have
\begin{align*}
\left|e_j(X_1^j, X_2,\dots, X_{k-1}, x_k) - \frac{e_j(X_1, X_2,\dots, X_{k-1}, x_k)}{r}\right| & \hspace{39mm}
\end{align*} 
\begin{align*}
= \left|\sum_{\substack{x_i \in X_i, \\ i \in [2, k-1]}}\left( e_j(X_1^j, x_2,\dots,  x_k) - \frac{e_j(X_1, x_2,\dots, x_{k-1}, x_k)}{r}\right) \right| \le \sum_{\substack{x_i \in X_i,\\ i \in [2, k-1]}} \varepsilon N  \le \varepsilon N^{k-1}. 
\end{align*}
\end{proofclaim}

Fix a partition of $X_1,\dots, X_{k-1}$ given by Claim~\ref{Clm: Partition Y}. For each $i \in [k-1]$ and $ j \in [r]$, remove additional vertices from $X_i^j$ if necessary so that, for all $i \in [k-1]$ and $ j \in [r]$, $\left|X_i^j\right| = \left(1-\varepsilon \right)\frac{N}{r}$. Let $X_i^0 = X_i \setminus \bigcup_{j \in [r]} X_i^j$ for each $i \in [k-1]$. Each $X_i^j$ has at most $2\varepsilon N/r$ vertices removed. Thus for all $j \in [r]$ and $ x_k \in X_k^j$, we have 
\begin{equation*}
    \begin{split}
      e_j(X_1^j,\dots, X_{k-1}^j, x_k) & \ge \frac{e_j(X_1, \dots, X_{k-1}, x_k)}{r^{k-1}} - k\varepsilon N^{k-1} - N^{k-2}\left(\frac{2\varepsilon N}{r}\right) \\ & \ge \frac{N^{k-1}}{r^k} - 2k\varepsilon N^{k-1}\ge \frac{N^{k-1}}{2r^k} \ge \frac{|X_1^j|...|X_{k-1}^j|}{2r}. \qedhere 
    \end{split}
\end{equation*} 
\end{proof}

We need a lower bound on the number of $K_k^{(k)}(2)$ in a dense $k$-partite $k$-graph~$H$.

\begin{lemma} \label{lma: lots of K22}
Let $1/n \ll 1/k, d\le 1$. Let $H$ be a $k$-partite $k$-graph with $n$ vertices in each vertex class and  $e(H) \ge dn^k$. Then there are at least $d^{2^k} n^{2k}/2$ many $K_k^{(k)}(2)$ in $H$. 
\end{lemma}

\begin{proof}
Let the vertex classes of $H$ be $X_1,\dots, X_k$. For vertices $x_1,\dots,x_k \in V(H)$, we write $\1(x_1\dots x_k)$ for the indicator function $\1(x_1\dots x_k\in E(H))$. For $j \in [k]\cup\{0\}$, let $$f(j)= \sum_{\substack{x_i \in X_i, \\ i \in [k-j]}} \sum_{\substack{x_s, y_s \in X_s, \\ s\in [k-j+1, k]}} \prod_{\substack{v_t \in \{x_t, y_t\}, \\ t \in [k-j+1, k]}} \1(x_1\dots x_{k-j} v_{k-j+1}\dots v_k).$$ Note that
\begin{equation}
 dn^k\le |E(H)| = \sum_{\substack{x_i \in X_i, i \in [k]}} \1(x_1\dots x_k) = f(0).   \label{eqn: base relation}
\end{equation}

\begin{claim} \label{claim: f(t) to f(0)}
    For $t \in [k]$, we have $n^{2^t k -k-t} f(t) \ge  f(0)^{2^t}$.
\end{claim}
\begin{proofclaim}
For any $j \in [k-1]\cup \{0\}$, we apply the Cauchy-Schwarz inequality to obtain
\begin{align}
        f(j)^2 & = \left( \sum_{\substack{x_i \in X_i, \\ i \in [k-j-1]}} \sum_{\substack{x_s, y_s \in X_s, \\ s\in [k-j+1, k]}} \left( \sum_{x_{k-j} \in X_{k-j}} \prod_{\substack{v_t \in \{x_t, y_t\}, \\ t \in [k-j+1, k]}} \1(x_1\dots x_{k-j} v_{k-j+1}\dots v_k)\right) \right)^2 \nonumber \\
        &\le n^{k+j-1} \sum_{\substack{x_i \in X_i, \\ i \in [k-j-1]}} \sum_{\substack{x_s, y_s \in X_s,\\ s \in [k-j+1, k]}}  \left(\sum_{x_{k-j} \in X_{k-j}} \prod_{\substack{v_t \in \{x_t, y_t \}, \\ t \in [k-j+1, k]}} \1(x_1\dots x_{k-j} v_{k-j+1}\dots v_k) \right)^2 \nonumber\\
        & = n^{k+j-1} \sum_{\substack{x_i \in X_i, \\ i \in [k-j-1]}} \sum_{\substack{x_s, y_s \in X_s, \\ s\in [k-j, k]}} \prod_{\substack{v_t \in \{x_t, y_t\}, \\ t \in [k-j, k]}} \1(x_1\dots x_{k-j-1} v_{k-j}\dots v_k) \nonumber\\
        & = n^{k+j-1} f(j+1). 
    \label{eqn: relation between f(j) and f(j+1)}
\end{align}
We now prove the claim by induction on $t$.
Note that $f(0)^2 \le n^{k-1}f(1)$ by~\eqref{eqn: relation between f(j) and f(j+1)}. Suppose that $t > 1$. Thus
\begin{equation}
    n^{2^t k -k-t} f(t) \stackrel{\eqref{eqn: relation between f(j) and f(j+1)}}{\ge} n^{2^t k -k-t} n^{-(k+t-2)} f(t-1)^2 = \left(n^{2^{t-1}k-k-(t+1)} f(t-1) \right)^2 \ge f(0)^{2^t},
\end{equation}
where the last inequality is due to our induction hypotheses.
\end{proofclaim}
The number of $K_k^{(k)}(2)$ in $H$ is 
\begin{equation*}
    \begin{split}
          \sum_{\substack{x_i, y_i \in X_i, \\ x_i \neq y_i, \\ i \in [k] }} \prod_{\substack{v_t \in \{x_t, y_t\}, \\ t \in [k]}} \1(v_1\dots v_k) 
         &\ge f(k) - kn^{2k-1} \stackrel{\rm{Claim}~\ref{claim: f(t) to f(0)}}{\ge} n^{2k-k2^k}f(0)^{2^k} - kn^{2k-1} \\
         & \stackrel{\eqref{eqn: base relation}}{\ge} n^{2k-k 2^k}(dn^k)^{2^k} - kn^{2k-1} \ge d^{2^k}n^{2k}/2. \qedhere
    \end{split}
\end{equation*} \end{proof}

We now show that there exists a (non-empty) subgraph $H_0$ of $H$ such that each edge is in many $K_k^{(k)}(2)$. 
\begin{lemma} \label{lma: every edge in enough K22}
    Let $1/n \ll d \le 1$, $k \ge 2$ and $\gamma = d^{2^k}/2$. Let $H$ be a $k$-partite $k$-graph with $n$ vertices in each vertex class and $e(H) \ge dn^k$. Then there exists a subgraph $H_0$ of $H$ such that $\left|E(H_0)\right| \ge  2^{k-1}\gamma n^k$ and every edge of $H_0$ is contained in at least $\gamma n^{k}/2$ many $K_k^{(k)}(2)$ in $H_0$.
\end{lemma}
\begin{proof}

Define the auxillary $2k$-graph $\mathcal{A}$ to be such that $V(\mathcal{A}) = V(H)$ and each $K_k^{(k)}(2)$ of $H$ is an edge of $\mathcal{A}$. By Lemma~\ref{lma: lots of K22}, $|E(\mathcal{A})| \ge \gamma n^{2k}$.  
If there exists $S \in \binom{V(\mathcal{A})}{k}$ such that $|S \cap V_i| = 1$ for $i\in [k]$ and $d_{\mathcal{A}}(S) < (\gamma/2)n^k$, then we delete all edges in $\mathcal{A}$ containing~$S$. Repeat this process and call the resulting $2k$-graph $\mathcal{A}'$. Note that $S$ can be chosen in at most~$n^k$ ways, so $$|E(\mathcal{A}')| \ge \gamma n^{2k} - n^k\left(\frac{\gamma n^k}{2}\right) = \frac{\gamma n^{2k}}{2}.$$
Let $H_0$ be the subgraph of $H$ induced by $e \in E(H)$ with $d_{\mathcal{A}'}(e) \ge 1$. For each $e \in E(H_0)$, we have $d_{\mathcal{A}'}(e) \ge \gamma n^k/2$ and therefore $e$ is contained in at least $\gamma n^{k}/2$ many $K_k^{(k)}(2)$ in~$H_0$. Since each edge can be contained in at most $n^k$ many $K_k^{(k)}(2)$ in $H_0$, 
$$|E(H_0)| \ge \frac{2^k |E(\mathcal{A}')|}{n^k} = 2^{k-1}\gamma n^k$$ as required.
\end{proof}

Let $S_n$ denote the set of permutations on~$[n]$. The following lemma shows that if the vertices within each class of a $k$-partite $k$-graph are permuted randomly, then the number of `horizontal' edges is concentrated around its expectation.  

\begin{lemma} \label{lma: permutation}
Let $1/n \ll 1/k \le 1/2$ and $H$ be a $k$-partite $k$-graph with vertex classes $X_1,\dots, X_k$ each of size $n$. Suppose $\sigma_1,\dots, \sigma_{k-1} \in S_n$ are chosen independently and uniformly at random. Let $H_{\sigma_1,\dots, \sigma_{k-1}} =  \{x_{1, \sigma_1(i)}\cdots x_{k-1, \sigma_{k-1} (i)} x_{k, i} : i \in [n]  \} \cap E(H).$ Then      
$$\mathbb{P}\left( \left| \mathbb{E}|H_{\sigma_1, \dots,\sigma_{k-1}}| - \frac{|E(H)|}{n^{k-1}} \right| > \varepsilon n\right) \le 2 e^{-\frac{\varepsilon^2 n}{32}}.$$
\end{lemma}
    \begin{proof}
       We consider the martingale obtained by exposing each position of $\sigma_1,\dots,\sigma_{k-1}$ in turns as follows.  For $i \in [k-1]$ and $j \in [n]$, let $Z_{i, j}$ be the $j^{\text{th}}$ exposed position of $\sigma_i$. 
       Let
       \begin{align*}
       B_{0} &= \mathbb{E}\left|H_{\sigma_1,\dots,\sigma_{k-1}}\right|  = \sum_{ x_{1,i_1}  \dots x_{k-1,i_{k-1}}x_{k, i} \in E(H)} \mathbb{P}(i_1=\sigma_1(i),\dots, i_{k-1}=\sigma_{k-1}(i)) \\
       &=\sum_{x_{1,i_1} \dots x_{k-1,i_{k-1}} x_{k, i} \in E(H)}\ \prod_{j \in [k-1]}\mathbb{P}(i_j = \sigma_j(i)) =\frac{|E(H)|}{n^{k-1}}.      
       \end{align*}        
Define \begin{equation*}
    B_{(i-1)n+j} =\mathbb{E}\left(\left|H_{\sigma_1,\dots,\sigma_{k-1}} \right||Z_{1, 1},\dots,Z_{1, n}, Z_{2, 1},\dots,Z_{2, n},\dots, Z_{i,1},\dots, Z_{i, j}\right).
\end{equation*} Note that $B_0,\dots, B_{(k-1)n}$ forms a martingale. 

\begin{claim}
For $t \in [(k-1)n -1]$, we have $|B_t - B_{t+1}|\le 4$.
\end{claim}
\begin{proofclaim}
For simplicity, we only consider the case when $t \le n-1$ (and the other cases are proved similarly). Thus $$B_{t} = \mathbb{E}\left(\left|H_{\sigma_1,\dots,\sigma_{k-1}} \right| | Z_{1,1},\dots,Z_{1, t}\right) \text{ and } B_{t+1} = \mathbb{E}(\left|H_{\sigma_1,\dots,\sigma_{k-1}} \right| | Z_{1,1},\dots, Z_{1, t+1}).$$ If $t = n-1$, then $B_t = B_{t+1}$. Thus we may assume $t\le n-2$. For $t_1, t_2 \in [n]$, define $\pi_{t_1,t_2} : S_n \to S_n$ to be such that for all $\sigma' \in S_n$ and $s \in [n]$, $$\pi_{t_1,t_2} (\sigma') (s) = \begin{cases}
\sigma'(t_2) & \text{if}\ s=t_1, \\
\sigma'(t_1) & \text{if}\ s=t_2, \\
\sigma'(s) & \text{otherwise.}
\end{cases}$$
Equivalently $\pi_{t_1,t_2}(\sigma')$ swaps the $t_1^{\text{th}}$ and $t_2^{\text{th}}$ position of $\sigma'$. 
For all $\sigma_1', \dots, \sigma_{k-1}' \in S_n$,

\begin{align}
    &\left|H_{\sigma_1', \dots, \sigma_{k-1}' }\Delta H_{\pi_{t_1, t_2}(\sigma_1') \sigma_2', \dots, \sigma_{k-1}' }\right| & \nonumber\\
    &\le \left|\{x_{1, \sigma_1'(i)}\dots x_{k-1, \sigma_{k-1}'(i)}x_{k, i}, \ x_{1, \pi_{t_1, t_2}(\sigma_1'(i))}\dots x_{k-1, \pi_{t_1, t_2}(\sigma_{k-1}'(i) )}  x_{k, i}: i\in \{t_1, t_2\}\}\right| \le 4. \label{eq: verification for Azuma}
\end{align}
 Let $A= \{\sigma \in S_n: \sigma(j)=Z_{1, j}\ \text{for}\ j \in [t] \}$. Given $Z_{1,1},\dots, Z_{1, t}$, the probability space is now reduced to $A \times \left(S_n\right)^{k-2}$. Pick $j_0 \in [n] \setminus \{Z_{1,1}\dots Z_{1,t}\}$. Let $A'= \{\sigma \in A: \sigma(t+1)=j_0 \}$. Note that $A$ can be partitioned into $\{\pi_{t, j}(A') : j \in [t+1, n] \}$. Hence the probability space can be partitioned into $\{\pi_{t, j}(A')\times \left(S_n\right)^{k-2}: j \in [t+1, n] \}$. Hence the claim follows by~\eqref{eq: verification for Azuma}.\end{proofclaim} 

By Azuma's inequality (Lemma~\ref{Lma: azuma}), we have
$$\mathbb{P}\left( \left| \mathbb{E}\left|H_{\sigma_1, \dots,\sigma_{k-1}}\right| - \frac{|E(H)|}{n^{k-1}} \right| > \varepsilon n\right) \le 2 e^{-\frac{\varepsilon^2 n}{32}}$$ as required.
\end{proof}

We are now ready to prove Lemma~\ref{Lma: Revised Lemma}.

\begin{proof}[Proof of Lemma~\ref{Lma: Revised Lemma}]
Let $$n = \left(1-\frac{1}{4kr^k}\right)\left(\frac{|X|}{r(k-1)}\right).$$
Partition $X$ into equally sized subsets $X_1,\dots, X_{k-1}$ so that for $j \in [k-1]$, we have $\left|X_j\right| = |X|/(k-1) \ge |Z|$. By Lemma~\ref{Lma: Slicing tripartite lemma}, there exists a partition of $X_1,\dots, X_{k-1}, Z$ into $\{X_1^j: j\in [r] \cup \{0\}  \},\dots, \{X_{k-1}^j: j\in [r] \cup \{0\}  \},  \{Z^j: j\in [r]  \}$ such that for each $ i \in [k-1]$, $j \in [r]$ and $z \in Z^j$, we have $$ \left|X_i^j\right| = n 
\text{ and } e_j\left(X_1^j,\dots, X_{k-1}^j, z\right) \ge \frac{\left|X_1^j\right|\cdots\left|X_{k-1}^j\right|}{2r} = \frac{n^{k-1}}{2r}.$$ 
For each $j \in [r]$, let $H^j$ be the $k$-partite $k$-graph with vertex classes $X_1^j, \dots, X_{k-1}^j, Z^j$ and $E(H^j) = E\left(H_j\left[X_1^j,\dots, X_{k-1}^j, Z^j\right]\right)$.

    Let $d = (2r)^{-1}$ and $\gamma= d^{2^{k}}/2$. Fix $j \in [r]$. Consider $z \in Z^j$ and let $H^j(z)$ be the link graph of~$z$ on vertex set $X_1^j,\dots,X_{k-1}^j$. Note that $|E(H^j(z))| \ge dn^{k-1}$. By Lemma~\ref{lma: every edge in enough K22} with~$H^j(z)$ playing the role of~$H$, there exists a $(k-1)$-partite $(k-1)$-uniform subhypergraph~$J^{z}$ of~$H^j(z)$ such that $|E(J^{z})| \ge  \gamma2^{k-2}n^{k-1}$ and every edge in~$J^{z}$ is contained in at least $\gamma n^{k-1}/2$ many~$K_{k-1}^{(k-1)}(2)$ in~$J^{z}$. For each edge~$e$ in~$J^{z}$, let $J^{z, e}$ be the $(k-1)$-partite $(k-1)$-graph on vertex classes $X_1^j,\dots, X_{k-1}^j$ such that~$f$ is an edge in~$J^{z, e}$ if and only if $e \cup f$ forms a~$K_{k-1}^{(k-1)}(2)$ in~$J^{z}$. 
For all $z \in Z^j$ and $e \in E(J^{z})$, we have $$\left|E(J^{z})\right| \ge  2^{k-2}\gamma n^{k-1} \ \text{and}\ \left|E(J^{z, e})\right| \ge \gamma n^{k-1}/2.$$
Let $X_i^j = \{x_{i, 1},\dots, x_{i, n}\}$. Choose $\sigma_1,\dots, \sigma_{k-2} \in S_n$ independently and uniformly at random. For a $(k-1)$-partite $(k-1)$-graph $J$ with vertex classes $X_1^j, \dots X_{k-1}^j$, let $$J_{\sigma_1,\dots \sigma_{k-2}} = \{x_{1,\sigma_1(i)} \dots x_{k-2, \sigma_{k-2}(i)}x_{k-1, i} : i\in [n]\} \cap E(J). $$
By Lemma~\ref{lma: permutation} we have $$\mathbb{P}\left(\left|J_{\sigma_1, \dots, \sigma_{k-2}}^{z}\right| < \gamma 2^{k-3}n \right) \le 2 e^{-\gamma^2 2^{2k-11}n} \ \text{and}\ \mathbb{P}\left( \left|J_{\sigma_1,\dots, \sigma_{k-2}}^{z, e} \right| < \frac{\gamma n}{4} \right) \le 2 e^{-\frac{\gamma^2 n}{512}}.$$
By the union bound, there exist some $\sigma_1,\dots, \sigma_{k-2}\in S_n$ such that for all $z \in Z^j$ and all~$e \in E(J^{z})$,
\begin{equation}
\left|J_{\sigma_1, \dots, \sigma_{k-2}}^{z}\right| \ge \gamma 2^{k-3}n = d^{2^k}2^{k-4}n \ \text{and}\  \left|J_{\sigma_1,\dots, \sigma_{k-2}}^{z, e}\right|\ge \frac{\gamma n}{4} =  \frac{d^{2^k} n}{8} = \frac{(2r)^{-2^k}n}{8}. \label{eq: bound on J} 
\end{equation}
 Without loss of generality (by relabelling vertices of $X_1^j,\dots, X_{k-2}^j$) we may assume $\sigma_1 =\dots=\sigma_{k-2} = \rm{id}$. For $z \in Z^j$, define the edge-coloured graph $G_{z}$ on $V= \{v_i: i \in [n]\}$ with colour~$z$ to be such that $v_iv_j \in E(G)$ if and only if $x_{1,i}\dots x_{k-1,i} x_{1,j}\dots x_{k-1, j}$ forms a $K_{k-1}^{(k-1)}(2)$ in $J^{z}$. Hence after removing isolated vertices from $G_z$, $\delta(G_{z}) \ge  (2r)^{-2^k}n/8$ by~\eqref{eq: bound on J}. Let $G^j= \bigcup_{z \in Z^j}G_{z}$. We now show that $\ref{itm: (A1)}-\ref{itm: (A4)}$ are satisfied. 

Note that by the construction of $H^j$ for each $j \in [r]$, $\ref{itm: (A1)}$ holds, $Z = \bigcup_{j \in [r]}Z^j$ and~$\bigcup_{i \in [k-1]}X_i^j \subseteq X$. We also have $$\phi(G^j) = \phi\left(\bigcup_{z \in Z^j} G_z \right) =\bigcup_{z \in Z^j} \phi(G_z)  = Z^j.$$ Thus $\ref{itm: (A2)}$ holds. 

Let $H^j$ be non-empty for some $j \in [r]$. Note that $$|V(G^j)| = \left| \bigcup_{z \in Z^j} V(G_z) \right| = n = \left(1-\frac{1}{4kr^k}\right)\left(\frac{|X|}{r(k-1)}\right).$$ Since $\delta(G_z) \ge (2r)^{-2^k}n/8$, we have $\delta_{\text{mon}}(G^j) \ge (2r)^{-2^k} |V(G^j)|/8$ and $G^j$ respects $H^j$ by construction. So $\ref{itm: (A3)}$ holds.

To see $\ref{itm: (A4)}$, recall that Lemma~\ref{Lma: Slicing tripartite lemma} partitions $X_1, \dots, X_{k-1}, Z$. Hence, for distinct $i, j \in [r]$,~$H^i$ and~$H^j$ are vertex-disjoint. This completes the proof of the lemma.
\end{proof}

\section{Rainbow Path Systems} \label{sec: rainbow path systems}

From now on, we will work with edge-coloured multigraph $G$. Note that the colours of $G$ are different to those in $K_n^{(k)}$. We often assume that $\delta_{\text{mon}}(G)$ is linear in $|V(G)|$. Recall that a \emph{rainbow path system} is a path system (collection of vertex-disjoint paths) such that its union (viewed as a graph) is rainbow. Our first goal is to prove the following result, which shows that there is a rainbow path system in $G$ with few paths.
\begin{lemma} \label{Lma: rainbow path system}
    Let $n, d \in \mathbb{N}$ and $G$ be an edge-coloured multigraph on $n$ vertices with $\delta_{\text{mon}}(G) \ge d \ge 4|\phi(G)|$. Then there exists a rainbow path system $\mathcal{P}$ such that $\phi(\mathcal{P}) = \phi(G)$ and $|\mathcal{P}| \le 2n/d$. 
\end{lemma}
To show this, we need the help of the following proposition which lets us start building~$\mathcal{P}$ with a rainbow matching.
\begin{proposition} \label{Prp: rm}
    Let $G$ be an edge-coloured multigraph with $\delta_{\text{mon}}(G)  \ge 2|\phi(G)|-1$. Then there exists a rainbow matching $M$ such that $\phi(M)=\phi(G)$.
\end{proposition}
    \begin{proof}
    Without loss of generality let $\phi(G) = [d]$. Suppose for some $i \in [d]$, we have already found a rainbow matching ${M}_{i-1}$ such that $\phi({M}_{i-1}) = [i-1]$. We now construct~${M}_{i}$ as follows. Since   $$\delta_{\text{mon}}(G) \ge 2|\phi(G)|-1 > 2|\phi(M_{i-1})| = |V({M}_{i-1})|,$$ there exists a vertex $v \in V(G)\setminus V({M}_{i-1})$ such that $N_i(v, \overline{V(M_{i-1})}) \neq \emptyset$. Pick a vertex $v' \in N_i(v, \overline{ V({M}_{i-1})})$ and let ${M}_{i}= {M}_{i-1} \cup \{vv'\}$.    
    Let $M={M}_d$.  
    \end{proof}

We now present the proof of Lemma~\ref{Lma: rainbow path system}.

    \begin{proof}[Proof of Lemma~\ref{Lma: rainbow path system}]
       Let $\mathcal{P}$ be a rainbow path system of $G$ with $\phi(\mathcal{P})=\phi(G)$, which exists by Proposition~\ref{Prp: rm}. Suppose that $|\mathcal{P}|$ is minimal. If $|\mathcal{P}| \le 2n/d$, then we are done. So we may assume $|\mathcal{P}| > 2n/d$. Let $\mathcal{P} = \{P_1, P_2,\dots, P_{\ell} \}$ be such that $P_i = v_1^i...v_{q_i}^i$ and without loss of generality $\phi(v_1^iv_2^i) = i$.
       Note that
$$d_i(v_2^i , \overline{V(\mathcal{P})}) \ge d-|V(\mathcal{P})|  \ge d - 2|\phi(G)| \ge d/2.$$
Hence there exist distinct $i, j \in [\ell]$ such that $(N_i (v_2^i) \cap N_j(v_2^j)) \setminus V(\mathcal{P}) \neq \emptyset$. Let $w \in \left(N_i (v_2^i) \cap N_j(v_2^j)\right) \setminus V(\mathcal{P})$. Let $\mathcal{P}' = (\mathcal{P} \setminus \{P_i, P_j\}) \cup \{v_q^j...v_3^jv_2^j w v_2^i v_3^i ... v_q^i\}$. Note that $\mathcal{P}'$ is also a rainbow path system with $\phi(\mathcal{P}') = \phi(G)$ and $|\mathcal{P}'|< |\mathcal{P}|$. This contradicts the minimality of $|\mathcal{P}|$.
    \end{proof}
The following lemma is our absorption result, the proof of which will be presented in the next section.

\begin{lemma} \label{Lma: absorption lemma}
    Let $1/n \ll \delta$ and~$G$ be an edge-coloured multigraph satisfying~$|V(G)|=n$ and $\delta_{\text{mon}}^*(G) \ge  \delta n$. Then there exists a subgraph $G^*$ of $G$ such that  
		\begin{align*}
		|V(G)\setminus V(G^*)|, |\phi(G)\setminus \phi(G^*)| < 2000\delta^{-4}.
		\end{align*}
		Moreover, for all rainbow paths $P^*$ in $G^*$ with $|V(P^*)| \ge 3$ and vertex sets $S \subseteq V(G^*)\setminus V(P^*)$ such that $|V(P^*) \cup S| \le \delta n/4$, $G \setminus S$ contains a rainbow cycle $C^*$ with $\phi(P^*) \subseteq \phi(C^*) \subseteq \phi(P^*) \cup \left(\phi(G)\setminus \phi(G^*)\right)$. 
\end{lemma}

We now prove Lemma~\ref{lma: Lemma Rainbow Cycle System} using Lemma~\ref{Lma: absorption lemma}. We iteratively reserve a set of vertices and colours, which shall be used to close the rainbow path system constructed at the end, one path at a time. 

\begin{proof}[Proof of Lemma~\ref{lma: Lemma Rainbow Cycle System}]
If $|\phi(G)| \le 2^{18}\delta_0^{-5}$ then we have $\delta_{\text{mon}}(G) \ge 2|\phi(G)|$. By Proposition~\ref{Prp: rm},~$G$ contains a rainbow matching~$M$ such that $\phi(M) = \phi(G)$. Since each edge is a degenerate cycle, we are done setting $\mathcal{C} = M$. Therefore, we may assume $|\phi(G)| > 2^{18}\delta_0^{-5}$.

Let $G^1 = G$. For each $j \in \left[4\delta_0^{-1}\right]$ in turns, we apply Lemma~\ref{Lma: absorption lemma} with $G^j, \delta_0/2$ playing the roles of $G, \delta$, respectively, to obtain a subgraph $G^{j+1}$ of $G$ such that  
$$|V(G^j)\setminus V(G^{j+1})|, \ |\phi(G^j) \setminus \phi(G^{j+1})| < 2000 (\delta_0/2)^{-4} \le 2^{15}\delta_0^{-4}.$$
Moreover, for all rainbow paths $P$ in $G^{j+1}$ with $|V(P)| \ge 3$ and a vertex set $S \subseteq V(G^{j+1}) \setminus V(P)$ such that $|V(P) \cup S| \le \delta_0 n/8$,  $G^{j+1}\setminus S$ contains a rainbow cycle $C$ with $$\phi(P) \subseteq \phi(C) \subseteq \phi(P) \cup (\phi(G^j)\setminus \phi(G^{j+1})).$$
Note that 
\begin{align}\label{eqn: min degree of G(j+1)}
    \delta_{\text{mon}}(G^{j+1}) &\ge \delta_{\text{mon}}(G^j) - |V(G^j)\setminus V(G^{j+1})|  
\ge \delta_{\text{mon}}(G^j) - 2^{15}\delta_0^{-4} \nonumber \\ 
    &\ge \delta_{\text{mon}}(G) - 2^{15}\delta_0^{-4}j 
\ge \delta_0n/2. 
\end{align}
Let $G^* = G^{4\delta_0^{-1}}$. Note that
\begin{equation*}
\delta_{\text{mon}}(G^*) = \delta_{\text{mon}}(G^{4\delta_0^{-1}}) \stackrel{}{\ge} \delta_0 n/2 \ge 8\phi(G^*)     
\end{equation*}
 and
 \begin{equation}
 |\phi(G)\setminus \phi(G^*)|,|V(G)\setminus V(G^*)| \le 2^{15}\delta_0^{-4}(4\delta_0^{-1}) = 2^{17}\delta_0^{-5}. \label{eqn: Number of reserved colours and vertices}    
 \end{equation}
  By Lemma~\ref{Lma: rainbow path system} with $G^*, \delta_0n/2$ playing the roles of $G, d$, respectively, there exists a rainbow path system~$\mathcal{P}$ such that $\phi(\mathcal{P}) = \phi(G^*)$ and \begin{equation}
      |\mathcal{P}| \le \frac{2|V(G^*)|}{\delta_0 n/2} \le 4\delta_0^{-1}. \label{eq: upper bound on number of paths}
  \end{equation} 
Note that 
\begin{equation} \label{eqn: upper bound on V(P)}
       |V(\mathcal{P})|  =  |\phi(G^*)| + |\mathcal{P}|  \le \frac{\delta_0 n}{16} + 4\delta_0^{-1} \le \frac{\delta_0 n}{8}. 
\end{equation} 
Moreover $$\delta_{\text{mon}}(G \setminus V(\mathcal{P})) \stackrel{\eqref{eqn: upper bound on V(P)}}{\ge} \delta_0 n - \frac{\delta_0 n}{8} \ge  \frac{\delta_0n}{2} \ge 2\phi(G)$$ and
 $\phi(G \setminus V(\mathcal{P})) \supseteq \phi(G)\setminus \phi(G^*)$. By Proposition~\ref{Prp: rm}, we deduce that there exists a rainbow matching~$M$ in $(G\setminus V(\mathcal{P}))_{\phi(G) \setminus \phi(G^*)}$ with $\phi(M) = \phi(G) \setminus \phi(G^*)$. Let $Q_0 = M \cup \mathcal{P}$, so $\phi(Q_0) = \phi(G)$ and $$|Q_0| \le |\mathcal{P}|+ |\phi(G) \setminus \phi(G^*)| \stackrel{\eqref{eq: upper bound on number of paths},~\eqref{eqn: Number of reserved colours and vertices}}{\le} 4\delta_0^{-1} + 2^{17}\delta_0^{-5}\le 2^{18}\delta_0^{-5}.$$
Let $P_1,\dots, P_{\ell}$ be paths in $Q_0$ such that for $j \in [\ell]$, $|V(P_j)| \ge 3$. Suppose for some $j \in [\ell-1]\cup\{0\}$, we have constructed $Q_j$ such that
\begin{enumerate} 
    \item[(i)] $Q_j$ is a union of vertex-disjoint paths and cycles with $|Q_j| \le |Q_0|$;
    \item[(ii)] $Q_j$ is rainbow with $\phi(Q_j) = \phi(G)$;
    \item[(iii)] the paths in $Q_j$ of order at least $3$ are precisely $P_{j+1},\dots, P_{\ell}$;
    \item[(iv)] all edges in $Q_j$ of colours in $\phi(G^{j+1})\setminus \phi(G^*)$ form a rainbow matching.
\end{enumerate}
We now construct $Q_{j+1}$ as follows. Note that $P_{j+1} \subseteq G^* \subseteq G^{j+1}$. Let $S = V(Q_{j} \setminus P_{j+1}) \cap V(G^{j+1})$. Note that \begin{align*}
    |S \cup V(P_{j+1})| &\le |V(Q_{j})| \le |\phi(Q_{j})|+|Q_{j}| \le |\phi(G)|+|Q_0| \\ &\le \frac{\delta_0 n}{16} +2^{18}\delta_0^{-5} \le \frac{\delta_0 n}{8} \stackrel{\eqref{eqn: min degree of G(j+1)}}{\le} \frac{\delta_{\text{mon}}(G^{j+1})}{4}.
\end{align*}
By the property of $G^{j+1}$, there is a rainbow cycle $C^{j+1}$ in  $G^{j+1} \setminus S$ such that $$\phi(P_{j+1}) \subseteq \phi(C^{j+1}) \subseteq \phi(P_j) \cup \left(\phi(G^{j})\setminus \phi(G^{j+1})\right).$$ Since $\mathcal{P} \subseteq G^*$, the subgraph $G_{\phi(C^{j+1})}\cap Q_j$ of $Q_j$ induced by the colours in $\phi(C^{j+1})$ is precisely the path~$P_{j+1}$ and single edges of colour in $\phi(C^{j+1})\setminus \phi(P_{j+1}) \subseteq \phi(G^j) \setminus \phi(G^{j+1})$. 
Let $Q_{j+1} = \left(Q_j - G_{\phi\left(C^{j+1}\right)}\right) \cup C^{j+1}$. Thus $Q_{j+1}$ consists of the cycles $C^1, \dots, C^{j+1}$, a rainbow matching with colours in $\phi(G^{j+2})\setminus \phi(G^*)$ and all paths of length at least~$3$ from~$Q_j$ except~$P_{j+1}$. Therefore $|Q_{j+1}| \le |Q_j| \le |Q_0|$. Furthermore $\phi(Q_{j+1}) = \phi(Q_j)\cup \phi(C^{j+1})=  \phi(G)$ and by construction, $Q_{j+1}$ is rainbow. Therefore,~$Q_{j+1}$ satisfies (i)--(iv).

Finally, note that $Q_{\ell}$ consists of cycles and single edges. The proof of the lemma is complete.
\end{proof}

\section{Closing rainbow paths} \label{Bowties}
The aim of this section is to prove Lemma~\ref{Lma: absorption lemma}. To prove Lemma~\ref{Lma: absorption lemma}, we would need to reserve some vertices and colours so that we can close any rainbow path. Ideally our aim is to find a vertex $v$, small disjoint vertex sets $W_1, W_2$ and small disjoint colour sets~$C_1, C_2$ such that for every $x \in V(G)$ and $i \in [2]$, there exists a rainbow path $P_{i, x}$ from $x$ to $v$ such that $V(P_{i,x}) \subseteq \{x, v\} \cup W_i$ and $\phi(P_{i, x})\subseteq C_i$. We now reserve the vertex set $v \cup W_1 \cup W_2$ and colour set $C_1 \cup C_2$. Given any rainbow path $P$ with end vertices $x$ and~$y$, $P \cup P_{1, x} \cup P_{1, y}$ forms a rainbow cycle. This will motivate our definition of a bowtie (see later). First, we need the following definition which describes how to reach a vertex~$v$ through a vertex set~$W$ using a colour set~$C$.

Recall that $V^*(G) = \{v \in V(G): |\phi_G(v)| \ge 2\}$ and for a path $P = v_1\dots v_{\ell}$, we denote $\text{int}(P) = \{ v_i : 2\le i \le \ell-1\}$. Let $G$ be an edge-coloured multigraph with vertex set~$V$.
For $v\in V$, $C \subseteq \phi(G)$ and $W \subseteq V$, define $U_G(v, C', W)$ to be the subset $U \subseteq V$ such that for all vertices $u \in U$, there is a rainbow path $P$ from $v$ to $u$ with $\phi(P) \subseteq C$ and $\mbox{int}(P) \subseteq W$. We will always assume that $\{v\} \cup W \subseteq V^*(G)\cap U_G(v, C, W)$. Let $U= U_G(v, C, W)$ and $g \in \mathbb{N}$. We say that $U$ is \emph{$g$-maximal in~$G$} if for all $u \in U, w_1, w_2 \in V$ and distinct $c, c_1, c_2 \notin C$ with $\phi(uw_1) = c_1$ and $\phi(w_1w_2)=c_2$, $$d_c(u, \overline{U}), d_{c_1}(w_1, \overline{U}), d_{c_2}(w_2, \overline{U}) < g.$$
If $U$ is not $g$-maximal, then by adding at most two vertices to $W$ and three colours to $C$, we can enlarge $U$ by at least $g$. This leads to the following proposition. 
 
\begin{proposition} \label{Prp: skeleton colour max size}
    Let $G$ be an edge-coloured multigraph. Let $v \in V^*(G)$ and $c \in \phi_G(v)$. Then there exists a colour set $C \subseteq \phi(G)$ containing $c$ and a vertex-set $W \subseteq V^*(G) \setminus v$ such that $U_G(v, C, W)$ is $g$-maximal, $|U_G(v, C, W)| \ge d_c(v)$ and $|C|, |W| \le 3 |U_G(v, C, W)|/g$. 
    \end{proposition}
    \begin{proof}

Initially, we set  $W = \emptyset$ and $C = \{c\}$. Note that $U_G(v, C, W) = N_c (v)$, so $|U_G(v, C, W)| \ge d_c(v)$ and $|W|, |C| \le  3| U_G(v,C,W) |/g$.

If $U_G(v, C, W)$ is $g$-maximal, then we are done. Suppose~$U_G(v,C,W)$ is not $g$-maximal. If there exists~$u \in U_G(v, C, W)$ and a colour~$c' \notin C$  such~that~$d_{c'}(u, \overline{U_G(v, C, W)})\ge g$ then add~$c'$ to~$C$ and~$u$ to~$W$. If there exists $u \in U_G(v, C, W)$, $w\in V^*(G)$ and a colour~$c' \notin C$ such that $d_{c'} (w, \overline{ U_G(v, C, W)}) \ge g$ and $\phi(uw) \notin \{c'\} \cup C$, then add~$c', \phi(uw)$ to~$C$ and $u, w$ to $W$. If there exists $u \in U_G(v, C, W)$, $ w, w' \in V^*(G)$ and a colour $c'$ such that $d_{c'} (w', \overline{ U_G(v, C, W)}) \ge g$ and $\phi(uw), \phi(ww'), c'$ are distinct colours not in $C$, then add $c', \phi(uw), \phi(ww')$ to $C$ and $u, w, w'$ to~$W$.

Note that we still have $|W|, |C| \le 3| U_G(v,C,W)|/g$ and $|U_G(v, C, W)| \ge d_c(v)$. Note that $W \subseteq V^*(G)$. We repeat this process until $U_G(v,C,W)$ is $g$-maximal.
    \end{proof}

  Recall that $\delta_{\text{mon}}^* (G) = \min_{v \in V^*(G)}\ \min_{c \in \phi_G(v)} d_c(v)$. In other words, $\delta_{\text{mon}}^* (G)$ is the (non-zero) minimum number of edges seen by a vertex~$v$ in a particular colour, minimised over all the vertices in~$G$ that see at least two colours. We are mainly concerned about vertices in $V^*(G)$ because if a vertex $v \notin V^*(G)$, then it cannot be contained in a non-degenerate rainbow cycle. Also recall that the \emph{closed neighborhood} of a vertex~$w$ in a graph $G$ is defined as $\{w\} \cup N_G(w)$ and denoted by $N_G[w]$. 
  
  Let $U=U(v, C, W)$ be $g$-maximal in $G$. The following crucial lemma says that if  a rainbow path~$P$ satisfies $\text{int}(P) \cap N_{G-G_C}[U] \neq \phi$, then $\text{int}(P) \subseteq N_{G-G_C}[U]$.

\begin{lemma} \label{Lma: path}
    Let $g \in \mathbb{N}$ and let~$G$ be an edge-coloured multigraph on~$n$ vertices with ${\delta_{\text{mon}}^*(G)>g}$. Let $v \in V(G)$, $W \subseteq V(G)$ and $C \subseteq \phi(G)$. Suppose that $U=U_G(v, C, W)$ is~$g$-maximal. Furthermore let $G'=G-G_C$ and $U^* = N_{G'} [U]$. Then the following hold:
    \begin{enumerate}[label={\rm(\roman*)}]
        \item  for all $x \in U^* \cap V^*(G')$ and $c \in \phi(G')$, we have $d_c (x, \overline{U}) <g $; \label{itm: lma: path: i}
        \item if $P$ is a rainbow path in $G'$ such that $\rm{int}(P) \cap U^* \neq \emptyset$, then $\rm{int}(P) \subseteq U^*$. \label{itm: lma: path: ii}
    \end{enumerate}  
\end{lemma}
    \begin{proof}
Consider $x \in U^* \cap V^*(G')$ and $c \in \phi(G')$.
If $x \in U$, then since $c \notin C$, by the $g$-maximality of $U$, we have $d_c(x, \overline{U}) < g$. If $x \in U^* \setminus U$, then there exists a vertex $u \in U$ and a colour $c' \notin C$ such that $\phi(ux)=c'$. If $c' \neq c$, then $d_c(x, \overline{U}) < g$ by the $g$-maximality of $U$. If $c=c'$, then there exists a colour $c'' \in \phi_{G'}(v) \setminus \{c\}$ as $v \in V^*(G')$. By the argument above, we have $d_{c''}(x, \overline{U}) < g$. Since $\delta_{\text{mon}}^*(G) > g$, we deduce that $N_{c''}(x, U) \neq \emptyset$. Pick $u' \in N_{c''}(x, U)$. Now, $\phi(u' x) = c''\neq c$. Again by the previous argument we have $d_c(x, \overline{U}) < g$. Hence~\ref{itm: lma: path: i} holds.

We now prove~\ref{itm: lma: path: ii}. Let $P=x_1\dots x_{\ell}$ be a rainbow path in $G'$ with $\text{int}(P) \cap U^* \neq \emptyset$. Suppose that $\phi(x_{j-1}x_j) = j$ for $j \in [\ell]\setminus \{1\}$. Furthermore assume that $x_2 \in U^*$ and $\ell \ge 4$. (Indeed, if $x_i \in U^*$ with $i \in [3, \ell-2]$, then consider the two rainbow paths $x_{i-1}x_i...x_{\ell}$ and $x_{i+1}x_i...x_1$ separately.) Thus it is enough to show that $x_3 \in U^*$, as we can then consider the rainbow path~$x_2x_3...x_\ell$. If $x_2 \in U$, then $x_3 \in U^*$. If $x_2 \in U^*\setminus U$, then~\ref{itm: lma: path: i} implies $d_2(x_2, \overline{U}) < g$ and so $N_2(x_2, U) \neq \emptyset$. Then by $g$-maximality of~$U$, $d_4(x_3, \overline{U})< g$ and so $N_4(x_3, U) \neq \emptyset$. Thus $x_3 \in U^*$ as required.
\end{proof}

Let $G$ be an edge-coloured multigraph. We say that $B=(v, C_1, W_1, C_2, W_2)$ is a \emph{bowtie in $G$} if~$v \in V(G), \  W_1, W_2 \subseteq V(G) \setminus \{v\}$ are disjoint, $C_1$ and $C_2$ are disjoint nonempty colour sets. Denote $\phi(B)=C_1 \cup C_2$, $W(B)=\{v\} \cup W_1 \cup W_2$ and for $i \in [2]$, $U_i(B|G) = U_{G \setminus W_{3-i} - G_{C_{3-i}}}(v, C_i, W_i)$. A bowtie~$B$ is \emph{$g$-maximal} in~$G$ if for $i \in [2]$, $U_i(B|G)$ is $g$-maximal in $G\setminus W_{3-i} - G_{C_{3-i}}$.

The following corollary shows that a $g$-maximal bowtie exists, which follows from Proposition~\ref{Prp: skeleton colour max size}. 

\begin{corollary} \label{Cor: Corollary for size restriction}
Let $g \in \mathbb{N}$ and $G$ be an edge-coloured multigraph on $n$ vertices. Let $v$ be a vertex such that $d_{c_1}(v), d_{c_2}(v) \ge g + 3n/g$ for distinct colours $c_1, c_2$. Then there exists a $g$-maximal bowtie $B(v, C_1, W_1, C_2, W_2)$ such that for $i \in [2]$, $c_i \in C_i$, $W(B) \subseteq V^*(G)$, $|U_i(B|G)| \ge d_{c_i}(v) - 3n/g$ and $|\phi(B)|, |W(B)| \le 6n/g$.
\end{corollary}
\begin{proof}
    Apply Proposition~\ref{Prp: skeleton colour max size} with $v, c_1, G- G_{c_2}$ playing the roles of $v, c, G,$ respectively and obtain a colour set $C_1$ and a vertex set $W_1$ such that $c_1 \in C_1$ and $U_{G-G_{c_2}}(v, C_1, W_1)$ is $g$-maximal in $G-G_{c_2}$ and $|U_{G-G_{c_2}}(v, C_1, W_1)|\ge d_{c_1}(v)$. Note that $$|C_1|, |W_1| \le 3|U_{G-G_{c_2}}(v, C_1, W_1)|/g \le 3n/g.$$ 
    Let $G' = G\setminus W_1 - G_{C_1}$. Apply Proposition~\ref{Prp: skeleton colour max size} again with $v, c_2, G'$, playing the roles of $v, c, G$, respectively and obtain a colour set $C_2$ and a vertex set $W_2$ such that $c_2 \in C_2$ and $U_{G'}(v, C_2, W_2)$ is $g$-maximal in~$G'$ and $$|U_{G'}(v, C_2, W_2)|\ge d_{c_2, G'}(v) \ge d_{c_2, G}(v) - |W_1| \ge d_{c_2, G}(v) - 3n/g.$$ Set $B=(v, C_1, W_1, C_2, W_2).$ Note that $B$ is $g$-maximal and for $i \in [2]$, $|U_i(B|G)| \ge d_{c_i}(v) - 3n/g$. Furthermore, $$|\phi(B)|, |W(B)| \le \frac{3}{g}\left( \sum_{i \in [2]} |U_i(B|G)| \right) \le \frac{6n}{g}.$$ \end{proof} 
We now continue our motivation for the proof of Lemma~\ref{Lma: absorption lemma}. Recall that our aim is to use a bowtie to close a rainbow path. First, we show that there exist a small set of bowties $B_1, \dots, B_t$ such that $\{U_2(B_i|G): i \in [t]\}$ partitions $V^*(G)$. This will ensure that any rainbow path can be extended via one of these bowties (but we may not be able to close it).
 
Let $d, g \in \mathbb{N}$ and $G$ be an edge-coloured multigraph. Let $\mathcal{B}$ be a family of bowties $B_1, \dots, B_{t}$ in $G$. Let $\phi(\mathcal{B}) = \bigcup_{i \in [t]} \phi(B_i)$ and $W(\mathcal{B}) = \bigcup_{i \in [t]} W(B_i)$. We write $G - \mathcal{B}$ for  $ G \setminus W(\mathcal{B}) - G_{\phi(\mathcal{B})}$. For a bowtie $B \in \mathcal{B}$ and $i\in [2]$, we  denote $U_i(B|G, \mathcal{B})= U_i(B | G - (\mathcal{B}\setminus B))$ and $U_i^*(B|G, \mathcal{B}) = N_{G-G_{\phi(\mathcal{B})}}[U_i(B|G, \mathcal{B})] \cap V^*(G - \mathcal{B})$. We say $\mathcal{B}$ is a \emph{$(d, g)$-partition of $G$} if the following hold

\begin{enumerate}[label=(\textbf{P\arabic*}), ref=(\textbf{P\arabic*})]
        \item\label{itm: P1} for all $i \in [2]$ and $ j \in [t]$, we have $ |U_i(B_j| G, \mathcal{B})| \ge d$; 
        \item\label{itm: P2} $W(B_1), \dots, W(B_{t})$ are all disjoint and $W(\mathcal{B}) \subseteq V^*(G)$; 
        \item\label{itm: P3} $\phi(B_1), \dots, \phi(B_{t})$ are all disjoint; 
        \item\label{itm: P4} for all $j \in [t]$, $B_j$ is $g$-maximal in $ G - (\mathcal{B}\setminus B_j)$;
        \item\label{itm: P5} for all distinct $j, j' \in [t]$, $U_2^*(B_j| G, \mathcal{B}) \cap U_2^*(B_{j'}| G, \mathcal{B}) = \emptyset$; 
        \item\label{itm: P6} $U_2^*(B_1| G, \mathcal{B}), U_2^*(B_2| G, \mathcal{B}), \dots ,
        U_{2}^*(B_t| G, \mathcal{B})$ partition $V^*(G-\mathcal{B})$.  
\end{enumerate}
We say that \emph{$\mathcal{B}$ is a weak $(d, g)$-partition of~$G$} if only~\ref{itm: P1} to~\ref{itm: P5} hold.

\begin{fact} \label{Fct: size of family}
  Let $d, g, n \in \mathbb{N}$ and $G$ be a graph on $n$ vertices. Let $\mathcal{B}$ be a weak $(d, g)$-partition of~$G$. Then $|\mathcal{B}| \le n/d$. 
  \end{fact}
  \begin{proof}
Consider $B \in \mathcal{B}$. By~\ref{itm: P1}, $|U_2^*(B|G, \mathcal{B}) \cup W(B)| \ge |U_2(B|G, \mathcal{B})| \ge d$. Note that $U_2^*(B|G, \mathcal{B}) \subseteq V^*(G-\mathcal{B}) \subseteq V(G) \setminus W(\mathcal{B})$. By~\ref{itm: P2} and~\ref{itm: P3}, for any distinct $B, B' \in \mathcal{B}$, $$(U_2^*(B| G, \mathcal{B}) \cup W(B)) \cap (U_2^*(B'|G, \mathcal{B}) \cup W(B')) = \emptyset.$$ Therefore, $|\mathcal{B}| \le n/d$.
  \end{proof}
The next lemma shows that one can extend a weak $(d, g)$-partition into a $(d, g)$-partition.

\begin{lemma}\label{Lma: find partition}
    Let $d, g, n \in \mathbb{N}$ with $d \ge 4g$ and $g \ge \max\{3n/g, 12n^2/gd\}$. Let $G$ be an edge-coloured multigraph on $n$ vertices such that $\delta_{\text{mon}}^*(G) \ge d$. Suppose $\mathcal{B}_0$ is a weak $(d/2, g)$-partition of $G$ with $|\phi(\mathcal{B}_0)|, |W(\mathcal{B}_0)| \le 6n|\mathcal{B}_0|/g$. Then there exists a $(d/2, g)$-partition~$\mathcal{B}^*$ so that~$\mathcal{B}_0 \subseteq \mathcal{B}^*$ and $|\phi(\mathcal{B}^*)|, |W(\mathcal{B}^*)| \le 6n|\mathcal{B}^*|/g \le 12n^2/gd$.
\end{lemma}
    \begin{proof}

Suppose we have already constructed a weak $(d/2, g)$-partition $\mathcal{B} = \{B_1,\dots, B_{t}\}$ with $\mathcal{B}_0 \subseteq \mathcal{B}$ and $|\phi(\mathcal{B})|, |W(\mathcal{B})| \le t(6n/g)$. By Fact~\ref{Fct: size of family}, 
\begin{equation}
t \le 2n/d. \label{eq: bound on t}     
\end{equation}
Hence $$|\phi(\mathcal{B}^*)|, |W(\mathcal{B}^*)| \le  \left(\frac{2n}{d}\right)\left(\frac{6n}{g}\right) = \frac{12n^2}{gd}.$$
We further assume that $t$ is maximal. If $\mathcal{B}$ satisfies~\ref{itm: P6} then we are done by setting $\mathcal{B}^* = \mathcal{B}$. 

Thus we may assume that~\ref{itm: P6} does not hold. We now construct a bowtie $B_{t+1}$ as follows. Let $G' = G - \mathcal{B}$. Pick $v^{t+1} \in V^*(G')\setminus \bigcup_{i \in [t]} U_2^*(B_i | G, \mathcal{B})$ and $c_1, c_2 \in \phi_{G'}(v^{t+1})$, which exist as $\ref{itm: P6}$ does not hold for $\mathcal{B}$. Note that $$\delta_{\text{mon}}^*(G')\ge \delta_{\text{mon}}^*(G)- |W(\mathcal{B})| \ge d - (12n^2/gd) \ge d-g \ge d/2 \ge 2g \ge g+3n/g.$$  By Corollary~\ref{Cor: Corollary for size restriction}, $G'$ contains a $g$-maximal bowtie $B_{t+1}$ such that $|\phi(B_{t+1})|, |W(B_{t+1})| \le 6n/g$ and $W(B_{t+1}) \subseteq V^*(G')$. 

Let $\mathcal{B}' = \mathcal{B} \cup B_{t+1}$. Clearly, $|\phi(\mathcal{B}')|, |W(\mathcal{B}')| \le (t+1)(6n/g)$. We now show that $\mathcal{B}'$ is a weak $(d/2, g)$-partition, contradicting the maximality of $t$.
Fix $i \in [2]$ and $j \in [t+1]$. Note that $$\delta_{\text{mon}}^*(G-(\mathcal{B}'\setminus B_j)) \ge d - |W(\mathcal{B}'\setminus B_j)| \ge d - (12n^2/gd) \ge d/2.$$ Thus, $|U_i (B_j|G, \mathcal{B}')| \ge \delta_{\text{mon}}^*(G-(\mathcal{B}'\setminus B_j)) \ge d/2$ and so~\ref{itm: P1} holds. Note that~\ref{itm: P2} and~\ref{itm: P3} hold by our construction. 
For $j \in [t]$, $B_j$ is $g$-maximal in $G-(\mathcal{B}\setminus B_j)$ so it is $g$-maximal in $G-(\mathcal{B}'\setminus B_j)$. Recall $B_{t+1}$ is $g$-maximal in $G' = G - \mathcal{B} = G - (\mathcal{B}'\setminus B_{t+1})$. Thus~\ref{itm: P4} holds for~$\mathcal{B}'$.

 It remains to show~\ref{itm: P5} holds for $\mathcal{B}'$. For $i \in [t]$, $U_2^*(B_i| G, \mathcal{B}') \subseteq U_2^*(B_i| G, \mathcal{B})$. Since $\mathcal{B}$ is a weak $(d/2, g)$-partition of $G$,~\ref{itm: P5} of $\mathcal{B}$ implies that $U_2^*(B_1|G, \mathcal{B}'), \dots, U_2^*(B_t|G, \mathcal{B}')$ are disjoint. 
Suppose that $U_2^*(B_{t+1}|G, \mathcal{B}') \cap U_2^*(B_j| G, \mathcal{B}') \neq \emptyset$ for some $j \in [t]$. Let $u \in U_2^*(B_{t+1}|G, \mathcal{B}') \cap U_2^*(B_j|G, \mathcal{B}') \subseteq V^*(G - \mathcal{B})$ and $c \in \phi_{G - \mathcal{B}'}(u)$. Note that $$d_{c, G-\mathcal{B}'}(u) \ge \delta_{\text{mon}}^*(G - \mathcal{B}') \ge d - (12n^2/gd) > g.$$ Lemma~\ref{Lma: path}\ref{itm: lma: path: i} and~\ref{itm: P4} imply that $d_{c, G-\mathcal{B}'}(u, \overline{U_2(B_j | G, \mathcal{B}')}) < g$. Therefore, 

\begin{align*}
d_{c, G}(u, U_2(B_j | G, \mathcal{B}')) &\ge d_{c, G}(u) - |W(\mathcal{B}')| - d_{c, G- \mathcal{B}'}(u, \overline{U_2(B_j| G,\mathcal{B}')})\\& > d_{c, G}(u) - \frac{12n^2}{gd} - g \ge \frac{d_{c, G}(u)}{2}.    
\end{align*}
Similarly, we have $d_{c, G}(u, U_2(B_{t+1}|G, \mathcal{B}')) > d_{c, G}(u)/2$. Hence,
\begin{align*}
\begin{split}
    &|U_2(B_j | G, \mathcal{B}') \cap U_2(B_{t+1}|G, \mathcal{B}')|
		  \ge |U_2(B_j | G, \mathcal{B}') \cap U_2(B_{t+1}|G, \mathcal{B}') \cap N_{c, G}(u)|\\ 
		\ge& d_{c, G}(u, U_2(B_j|G, \mathcal{B}')) + d_{c, G}(u, U_2(B_{t+1}|G, \mathcal{B}'))- d_{c, G}(u)		 > 0 .   
\end{split}
\end{align*}
 Let $w \in U_2(B_j|G, \mathcal{B}') \cap  U_2(B_{t+1}| G, \mathcal{B}')$. Let $B_{t+1} = (v^{t+1}, C_1^{t+1}, W_1^{t+1}, C_2^{t+1}, W_2^{t+1})$. Since $w \in U_2(B_{t+1}|G, \mathcal{B}')$, there exists a rainbow path $P=wv_1...v_{\ell}v^{t+1}$ with $\phi(P) \subseteq C_2^{t+1}$ and $\rm{int}(P) \subseteq W_2^{t+1}$. By~\ref{itm: P3}, $\phi(P) \cap \phi(B_j) = \emptyset$. We deduce that $v_1 \in U_2^*(B_j | G, \mathcal{B}')$. Let $c_1 \in C_1^{t+1} \cap \phi_G(v^{t+1})$. Note that $$d_{c_1, G'}(v^{t+1})\ge \delta_{\text{mon}}^*(G') \ge d/2 > |C_2^{t+1}| \ge |V(P)|.$$ Pick $v' \in N_{c_1, G'}(v^{t+1})\setminus V(P)$. Then the path $P' = wv_1\dots v_{\ell} v^{t+1} v'$ is a rainbow path with $\phi(P') \cap \phi(B_j) = \emptyset$. By Lemma~\ref{Lma: path}\ref{itm: lma: path: ii}, $v^{t+1} \in \mbox{int}(P') \subseteq U_2^*(B_j| G, \mathcal{B})$, contradicting the fact that~$v^{t+1}$ was chosen from $V^*(G')\setminus \bigcup_{i \in [t]} U_2^*(B_i|G, \mathcal{B})$.  
    \end{proof}

\subsection{Proof of Lemma~\ref{Lma: absorption lemma}}

We now present the ideas in the proof of Lemma~\ref{Lma: absorption lemma}. Let $d=\delta n$.

We apply Lemma~\ref{Lma: find partition} to obtain a $(d/2, g)$-partition $\mathcal{B}$ of $G$. Let $P$ be a rainbow path in~$G-\mathcal{B}$. Since $\rm{int}(P) \subseteq V^*(G-\mathcal{B})$, we deduce (by~\ref{itm: P4},~\ref{itm: P6} and Lemma~\ref{Lma: path}\ref{itm: lma: path: ii}) that there exists $B \in \mathcal{B}$ such that $\rm{int}(P) \subseteq U_2^*(B|G, \mathcal{B})$. We now discuss how to augment~$P$ into a rainbow cycle using~$B$. 

Let $P= x_1x_2\dots x_{\ell}$ with $\phi(x_1x_2)=2$ and $\phi(x_{\ell -1}x_{\ell})=\ell$.

\noindent\textbf{Case 1}: $|U_2(B|G, \mathcal{B})| \le 3d/4$. By Lemma~\ref{Lma: path}\ref{itm: lma: path: i}, we can find $x \in U_2(B|G, \mathcal{B})$ to replace both $x_1$ and $x_{\ell}$ in $P$. This transforms~$P$ into a rainbow cycle.

\noindent\textbf{Case 2}: $U_2^*(B|G, \mathcal{B}) = U_1^*(B|G, \mathcal{B})$. By Lemma~\ref{Lma: path}\ref{itm: lma: path: i}, we may assume that $x_1 \in U_1(B|G, \mathcal{B})$ and $x_{\ell} \in U_2(B|G, \mathcal{B})$. Let $B=(v, C_1, W_1, C_2, W_2)$. There exists a rainbow path $P'$ from~$x_1$ to~$x_{\ell}$ through $v$ such that $\phi(P') \subseteq \phi(B)$ and $\rm{int}(P') \subseteq W(B)$. Hence $PP'$ is a rainbow cycle.

Therefore we would like to ensure Case 1 or 2 hold. Our aim is to refine~$\mathcal{B}$ by replacing bowties with smaller ones so that Case 1 will hold eventually. In particular, we will increase the number of bowties in each step. For simplicity, suppose~$V^*(G-\mathcal{B}) = V(G)$. Note that $B \in \mathcal{B}$ is $g$-maximal in $G-(\mathcal{B}\setminus B)$. Lemma~\ref{Lma: path}\ref{itm: lma: path: i} implies that ``the subgraph $H_B$ induced by $U_2^*(B|G, \mathcal{B})$ satisfies $\delta_{\text{mon}}^*(H_B) \ge d-g\ge d/2$". Hence we can apply Lemma~\ref{Lma: find partition} to obtain a $(d/4, g)$-partition~$\mathcal{B}'$ of~$H_B$. Thus, we refine $\mathcal{B}$ by replacing~$B$ with $\mathcal{B}'$.

Suppose that we are unable to refine $\mathcal{B}$ further and for simplicity $\mathcal{B}$ consists of only one bowtie $B=(v, C_1, W_1, C_2, W_2)$. If Case 2 fails, then $U_1^*(B|G, \mathcal{B})$ is smaller than $U_2^*(B|G, \mathcal{B})$. We consider the `swapped' bowtie $B'=(v, C_2, W_2, C_1, W_1)$ instead. We extend this weak partition~$\{B'\}$ into a $(d/4, g)$-partition of~$G$ (using Lemma~\ref{Lma: find partition}). Note that we increase the number of bowties in the partition.

\begin{proof}[Proof of Lemma~\ref{Lma: absorption lemma}]
     Set $d=\delta n$, $\gamma = \delta^2/16$ and $g = \gamma n$. Note that $$d \ge 4g+2 \text{ and } g \ge \max\{6n/g, 12n^2/gd\}.$$
   Let $G^0=G$, $J^0=G$ and $H^0=G$. By Lemma~\ref{Lma: find partition} with $H^0, d, g,\emptyset$ playing the roles of $G, d, g, \mathcal{B}_0$, respectively, we obtain a $(d/2, g)$-partition $\mathcal{B}^1$ of $H^0$ and $|\phi(\mathcal{B}^1)|, |W(\mathcal{B}^1)| \le 2n|\mathcal{B}^1|/g$ and $W(\mathcal{B}^1) \subseteq V^*(H^0)$.

Suppose for some $i \in \mathbb{N}$ we have already constructed families $\mathcal{B}^1, \mathcal{B}^2, \dots, \mathcal{B}^i$ of bowties, edge-coloured multigraphs $J^0, G^0, H^0, \dots, J^{i-1}, G^{i-1}, H^{i-1}$ whose properties will be specified later. Let
\begin{align*}
J^i &= G^{i-1} - \mathcal{B}^i, \\
G^i &= J^i - \bigcup_{\substack{B, B' \in \mathcal{B}^i \\ B \neq B'}} J^i[U_2^*(B|H^{i-1}, \mathcal{B}^i), U_2^*(B'|H^{i-1}, \mathcal{B}^i)].
\end{align*}
In other words, $G^i$ is obtained from $G^{i-1} - \mathcal{B}^i$ by removing all the edges that are between $U_2^*(B|H^{i-1}, \mathcal{B}^i)$ and $U_2^*(B'|H^{i-1}, \mathcal{B}^i)$ for distinct $B, B' \in \mathcal{B}^i$.

Let $\mathcal{B}_1^i = \{B \in \mathcal{B}^i : |U_2(B|H^{i-1}, \mathcal{B}^i)| \le 3d/4 \}$. (These will consist of bowties that satisfy Case 1.) Consider $i \in \mathbb{N}$ and $B \in \mathcal{B}^i$. For a bowtie family $ \mathcal{B}'$ in $H^i$, define $$\partial_{\mathcal{B}'}(B) = \{B' \in \mathcal{B}' : W(B') \subseteq U_2^*(B|H^{i-1}, \mathcal{B}^i)\}.$$
We say $B$ is \emph{covered by $\mathcal{B}'$} if $\partial_{\mathcal{B}'}(B) = \{B'\}$ and $U_2^*(B'|H^i, \mathcal{B}') = U_1^*(B'|H^i, \mathcal{B}')$. We write $\partial B$ for~$\partial_{\mathcal{B}^{i+1}}B$. Let $$\mathcal{B}_2^i =
 \begin{cases}
   \bigcup\{\partial B\colon B\in \mathcal{B}^{i-1} \ \text{is covered by}\ \mathcal{B}^i \} & \text{if}\ i \ge 2, \\
   \emptyset & \text{if}\ i=1. 
 \end{cases}
 $$ (These bowties in $\mathcal{B}_2^i$ will satisfy Case 2.)
  Let $\mathcal{B}_3^i = \mathcal{B}^i \setminus (\mathcal{B}_1^i \cup \mathcal{B}_2^i)$ and
  \begin{align*}
      H^i &=G^i \setminus \left(V^*(G^i) \setminus \bigcup_{B \in \mathcal{B}_3^i}U_2^* (B| H^{i-1}, \mathcal{B}^i)\right) \\&= G^i\left[(V(G^i)\setminus V^*(G^i)) \cup \bigcup_{B \in \mathcal{B}_3^i} U_2^*(B|H^{i-1}, \mathcal{B}^i)\right].
  \end{align*}
In other words,  $H^i$ is induced by $U_2^*(B|H^{i-1}, \mathcal{B}^i)$ for all $B \in \mathcal{B}_3^i$ and the vertices in $G^i$ that see at most one colour. From now on, for $B \in \mathcal{B}^j$, we write $U_i^*(B)$ for $U_i^*(B|H^{j-1}, \mathcal{B}^j)$. We now list the desired properties of $G^i$ and $H^i$.
Suppose that for all $j \in [i]$, 

\begin{enumerate}[label = {\rm(\roman*)}]
    \item \label{itm: (i)} $\mathcal{B}^j$ is a $(d/4, g)$-partition of $H^{j-1}$ with $|\phi(\mathcal{B}^j)|, |W(\mathcal{B}^j)| < 384 \delta^{-3}$ and $W(\mathcal{B}^j) \subseteq V^*(H^{j-1})$;
    \item \label{itm: (ii)} if $B' \in \mathcal{B}^{j-1}$ and $|\partial B'| = 1$, then $B'$ is covered by $\mathcal{B}^j$;
    \item \label{itm: (iii)} if $j \ge 2$, then $\mathcal{B}^j = \bigcup_{B \in \mathcal{B}_3^{j-1}} \partial B$;
    \item \label{itm: (iv)} for all $B \in \mathcal{B}_3^j$, $|U_2^*(B)| \le n - j(d/4)$;
    \item \label{itm: (v)}  $\delta_{\text{mon}}(G^j) \ge d-j(g+12n^2/gd) \ge d-2gj$;
    \item \label{itm: (vi)} if $V^*(H^j)$ is not empty, then $\delta_{\text{mon}}^* (H^j) \ge d - j(g+12n^2/gd) \ge d-2gj$;
    \item \label{itm: (vii)} $V^*(H^j) \subseteq \bigcup_{B \in \mathcal{B}_3^j} U_2^*(B)$;
    \item \label{itm: (viii)}  $\phi(G^j) = \phi(G^{j-1}) \setminus \phi(\mathcal{B}^j)$;
    \item \label{itm: (ix)} $G^j$ has no edge between $U_2^*(B)$ and $V^*(G^j)\setminus U_2^*(B)$ for all $B \in \mathcal{B}^j$;
    \item \label{itm: (x)} for each $B \in \mathcal{B}^j$, $B$ is $g$-maximal in $G^{j-1} - (\mathcal{B}^j \setminus B)$.
\end{enumerate}
It should be noted that~\ref{itm: (i)} and~\ref{itm: (ii)} imply~\ref{itm: (iii)} to~\ref{itm: (x)}. Hence we technically only require $\mathcal{B}^1, \dots, \mathcal{B}^j$ to satisfy~\ref{itm: (i)} and~\ref{itm: (ii)}.
Note that~\ref{itm: (iv)} implies that
\begin{equation}
i \le 4n/d = 4/\delta. \label{eq: bound on i}    
\end{equation}

\noindent\textbf{Case A: $V^*(H^i) \neq \emptyset$}. We now construct $\mathcal{B}^{i+1}$ as follows.  We have $$\delta_{\text{mon}}^*(H^i) \stackrel{\mathclap{\text{(\ref{eq: bound on i}),~\ref{itm: (vi)}}}}{\ge} d- 2g(4\delta^{-1})\ge d/2.$$  By Lemma~\ref{Lma: find partition} with $H^i, d/2, g, \emptyset$ playing the roles of $G, d, g, \mathcal{B}_0$, respectively, there exists a $(d/4, g)$-partition $\mathcal{B}$ of $H^i$ with $|\phi(\mathcal{B})|, |W(\mathcal{B})| \le 6n|\mathcal{B}|/g = 192 \delta^{-3}$.

\begin{claim} \label{Clm: inside parent}
Let $\mathcal{B}$ be a $(d/4, g)$-partition of $H^i$ with $|\phi(\mathcal{B})|, |W(\mathcal{B})| \le 12n^2/gd$. Let $B^+ \in \mathcal{B}^i$ and $B \in \partial_{\mathcal{B}}B^+$. Then for $t \in [2]$, $U_t^*(B|H^i, \mathcal{B}) \subseteq U_2^*(B^+)$. Moreover, if $\partial_{\mathcal{B}}B^+ = \{B\}$, then $U_2^*(B|H^i, \mathcal{B}) = U_2^*(B^+) \cap V^*(H^i - \mathcal{B})$.
\end{claim}
    \begin{proofclaim}
        We prove the case for $t=2$ (and the case for $t=1$ is proven analogously). Let $B = (v, C_1, W_1, C_2, W_2)$, so $v \in U_2^*(B^+)$. Suppose there exists a vertex
        \begin{equation}
        u \in U_2^*(B|H^i, \mathcal{B}) \setminus U_2^*(B^+). \label{eq: contradicting vertex}    
        \end{equation}
Then there exists a rainbow path $P = vw_1 \dots w_{\ell}u$ such that $w_1,\dots, w_{\ell-1} \in W_2$, $\phi(P) \subseteq C_2 \cup \phi(w_{\ell} u) \subseteq \phi(H^i)$ and $C_1 \cap \phi(P) = \emptyset$. Note that \begin{equation} 
        |P| \le |W(\mathcal{B})| + 2 \le 192\delta^{-3}+2 < d/2  \label{eq: bound on path}    
        \end{equation}
        and $u \in V^*(H^i - \mathcal{B})$. Thus, there exists a colour $c' \in \phi_{H^i - \mathcal{B}}(u) \setminus \phi(w_{\ell}v)$. By~\ref{itm: (vi)}, $$d_{c', H^i}(u) \ge d - 2ig \ge d - \frac{8g}{\delta}\  \stackrel{\mathclap{\text{\eqref{eq: bound on path}}}}{>}  |P|.$$
Pick $u' \in N_{c', H^i}(u)\setminus V(P)$. Recall that $\phi(H^{i})\supseteq \phi(\mathcal{B}) \supseteq C_1 \neq \emptyset$. Pick $c_1 \in C_1 \cap \phi_{H^{i}}(v) \subseteq \phi(H^i)$. By~\ref{itm: (vi)}, $$d_{c_1, H^{i-1}}(v) \ge d - 2g(i-1) \ge d - \frac{8g}{\delta}\ \stackrel{\mathclap{\text{\eqref{eq: bound on path}}}}{>} |P|+1.$$ Pick $v' \in N_{c_1, H^{i-1}}(u) \setminus (V(P) \cup \{u'\})$. Then the path $P' = v'Pu' = v'vw_1\dots w_{\ell}uu'$ is rainbow in~$H^{i-1}$ with $$\phi(P') \cap \phi(\mathcal{B}^i) \subseteq \phi(H^i) \cap \phi(\mathcal{B}^i) \subseteq \phi(G^i) \cap \phi(\mathcal{B}^i) = \emptyset$$ by~\ref{itm: (viii)}. Recall that $U_2^*(B^+)$ is $g$-maximal in $H^{i-1} - (\mathcal{B}^i \setminus B^+) $. Lemma~\ref{Lma: path}\ref{itm: lma: path: ii} implies that $u \in \rm{int}(P') \subseteq U_2^*(B^+)$ contradicting~\eqref{eq: contradicting vertex}. The moreover statement follows.
    \end{proofclaim}

For each $B^+ \in \mathcal{B}^i$ with $\partial_{\mathcal{B}} B^+ = \{B\}$ that is not covered by $\mathcal{B}$, we replace the bowtie $B=(v, C_1, W_1, C_2, W_2)$ with $B' = (v, C_2, W_2, C_1, W_1)$. We call the resulting family $\mathcal{B}'$. We now show that $\mathcal{B}'$ is a weak $(d/4, g)$-partition of $H^i$. Note that~\ref{itm: P1} to~\ref{itm: P4} hold. To show~\ref{itm: P5}, note that if $B' \in \mathcal{B}'\setminus \mathcal{B}$, then there exist unique $B \in \mathcal{B}\setminus \mathcal{B}'$ and $B^+ \in \mathcal{B}^i$ with $$B=(v, C_1, W_1, C_2, W_2), \ B' = (v, C_2, W_2, C_1, W_1)\ \text{and}\  \partial_{\mathcal{B}} B^+ = \{B\}.$$ By Claim~\ref{Clm: inside parent} and the fact $\partial_{\mathcal{B}} B^+ = \{B\}$, we deduce that $$U_2^*(B'|H^i, \mathcal{B}') = U_1^*(B|H^i, \mathcal{B}) \subseteq U_2^*(B^+) \cap V^*(G^{i+1}) = U_2^*(B|H^i, \mathcal{B}).$$
Recall that $\mathcal{B}$ is a $(d/4, g)$-partition of $H^i$, in particular this implies that $\{U_2^*(B|H^i, \mathcal{B}): B \in \mathcal{B}\}$ are pairwise disjoint. Thus, $\mathcal{B}'$ satisfies~\ref{itm: P5} and is a weak $(d/4, g)$-partition. Apply Lemma~\ref{Lma: find partition} with $H^i, d/2, g, \mathcal{B}'$ playing the roles of $G, d, g, \mathcal{B}_0$, respectively and obtain $(d/4, g)$-partition $\mathcal{B}^{i+1}$ containing~$\mathcal{B}'$ and satisfying $W(\mathcal{B}^{i+1}) \subseteq V^*(H^i)$ and $$|\phi(\mathcal{B}^{i+1})|, |W(\mathcal{B}^{i+1})| \le 2(6n|\mathcal{B}^{i+1}|/g) \le 24n^2/gd \le 384\delta^{-3}.$$ 

\begin{claim}
    $\mathcal{B}^{i+1}, G^{i+1}, H^{i+1}$ satisfy~\ref{itm: (i)} to \ref{itm: (x)}.
\end{claim}
\begin{proofclaim}
   
 Note that~\ref{itm: (i)} holds by construction.

Consider $B^+ \in \mathcal{B}^i$. If $|\partial_{\mathcal{B}}B^+| = 1$ and $B^+$ is not covered by $\mathcal{B}$, then we have $|\partial_{\mathcal{B}^{i+1}} B^+| >  1$. Otherwise $\partial_{\mathcal{B}}B^+ = \partial_{\mathcal{B}^{i+1}}B^+$. Suppose that $\partial B^+ = \{B\}$. By~\ref{itm: P6} and Claim~\ref{Clm: inside parent}, $U_2^*(B|H^i, \mathcal{B}^{i+1}) = U_2^*(B^+|H^{i-1}, \mathcal{B}^i) \cap V^*(H^i - \mathcal{B}^{i+1})$. Note that $B^+$ is  covered by $\mathcal{B}$, therefore covered by $\mathcal{B}^{i+1}$. Thus~\ref{itm: (ii)} holds.

Consider $B = (v, C_1, W_1, C_2, W_2) \in \mathcal{B}^{i+1}$. For each $w \in W(B)\setminus \{v\}$, there is a path from~$v$ to~$w$ in~$H^i$. Thus $W(B)$ is contained in a component of $H^i[V^*(H^i)]$. Since $H^i \subseteq G^i$,~\ref{itm: (ix)} for $G^i$ implies that $W(B) \subseteq U_2^*(B^+)$ for some $B^+\in \mathcal{B}_3^i$. Therefore~\ref{itm: (iii)} holds for $G^{i+1}$.

We now show that~\ref{itm: (iv)} holds. Let $B \in \mathcal{B}_3^{i+1}$. Then there exists $B^+ \in \mathcal{B}^i$ and $B' \in \mathcal{B}_3^{i+1}$ such that~$B, B' \in \partial B^+$. By~\ref{itm: P1},~\ref{itm: P5} and Claim~\ref{Clm: inside parent}, we deduce that $U_2^*(B) \cup W(B)$ and $U_2^*(B')\cup W(B')$ each has size at least $d/4$, are disjoint and contained in $U_2^*(B^+)$. Hence,  
\begin{equation*}
    \begin{split}
    |U_2^*(B)| & \le |U_2^*(B^+) - |U_2^*(B') \cup W(B')| \\
    & \stackrel{\mathclap{\text{\ref{itm: (iv)}}}}{\le} n-i(d/4) - (d/4) = n-(i+1)(d/4).
    \end{split}
\end{equation*}

To show~\ref{itm: (v)}, consider a vertex $u \in V(G^{i+1})$ and a colour $c \in \phi_{G^{i+1}}(u)$. It is enough to show that $d_{c, G^{i+1}}(u) \ge d - (i+1)(g+12n^2/gd)$. By~\ref{itm: (v)} for $G^i$, we have $d_{c, G^i}(u) \ge d - i(g+12n^2/gd)$. Then $$d_{c, J^{i+1}}(u) \ge d_{c, G^i}(u) - | W(\mathcal{B}^{i+1})| \ge d - ig - (i+1)(12n^2/gd).$$ If $u \in V(G^{i+1})\setminus V^*(G^{i+1})$, then $d_{c, G^{i+1}}(u)= d_{c, J^{i+1}}(u) $. If $u \in V^*(G^{i+1})$, then by~\ref{itm: P6},  $u \in U_2^*(B)$ for some $B \in \mathcal{B}^{i+1}$. Then by $g$-maximality of $B$ in $H^i - (\mathcal{B}^{i+1}\setminus B)$, Lemma~\ref{Lma: path}\ref{itm: lma: path: i} implies that $d_{c, G^{i+1}} (u, \overline{U_2^*(B)}) < g$. We deduce that $$d_{c, G^{i+1}}(u) \ge d_{c, J^{i+1}}(u) - d_{c, G^{i+1}}(u, \overline{U_2^*(B))}) \ge d - (i+1)(g+12n^2/gd).$$ This argument also implies~\ref{itm: (vi)},~\ref{itm: (vii)} and $\phi(J^{i+1}) = \phi(G^{i+1})$. Since $\delta_{\text{mon}}(G^i) \ge d/4 >  |W(\mathcal{B}^{i+1})|$, we have $\phi(J^{i+1}) = \phi(G^i) \setminus \phi(\mathcal{B}^{i+1})$ implying~\ref{itm: (viii)}.

To show~\ref{itm: (ix)}, consider $B \in \mathcal{B}^{i+1}$. Let $B^+ \in \mathcal{B}^i$ with $B \in \partial B^+$. By Claim~\ref{Clm: inside parent}, $U_2(B) \subseteq U_2^*(B^+)$. Hence $G^{i+1}$ contains no edge between $U_2^*(B)$ and $U_2^*(B^+)\setminus U_2^*(B)$. By~\ref{itm: (ix)} for~$G^i$, there are no edges in~$G^i$ (and so in $G^{i+1}$) between $U_2^*(B^+)$ and $V^*(G^i)\setminus U_2^*(B^+)$ in $G^i$ and so in $G^{i+1}$. Thus~\ref{itm: (ix)} holds for $G^{i+1}$.

Suppose for contradiction~\ref{itm: (x)} does not hold. That is, some bowtie $B \in \mathcal{B}^{i+1}$ is not $g$-maximal in $G^{i} - (\mathcal{B}^{i+1} \setminus B)$. Note that $B$ is $g$-maximal in $H^{i} - (\mathcal{B}^{i+1} \setminus B)$. Let $B^+ \in \mathcal{B}^{i}$ with $B \in \partial B^+$. By~\ref{itm: (ix)}, $G^{i}$ has no edge between $U_2^*(B^+|H^{i-1}, \mathcal{B}^i)$ and $V^*(G^{i}) \setminus U_2^*(B^+|H^{i-1}, \mathcal{B}^{i})$. Thus for any~$v \in V^*(H^{i})$, the neighbourhoods of $v$ are the same in~$H^{i}$ and in $G^{i}$. Hence any vertices $u, w, w'$ that could prevent $g$-maximality must also be contained in $V^*(H^{i})$. Therefore $B$ is $g$-maximal in $G^{i} - (\mathcal{B}^{i+1} \setminus B)$.\end{proofclaim}
 
\noindent\textbf{Case B: $V^*(H^i) = \emptyset$}. Set $C^* = \bigcup_{j \in [i]} \phi(\mathcal{B}^j)$ and $W^* = \bigcup_{j \in [i]} W(\mathcal{B}^j)$.  Let $G^* = G^i$. Clearly, 
\begin{equation}
|W^*| \le \sum_{j \in [i]} \left|W(\mathcal{B}^j)\right| \stackrel{\ref{itm: (i)},~\eqref{eq: bound on i}}{\le} (384\delta^{-3})(4\delta^{-1}) < 2000 \delta^{-4}. \label{eq: bound on W^*}    
\end{equation}
There is a similar upper bound on $|C^*|$. Note that $\phi(G^*) = \phi(G) \setminus C^*$ by~\ref{itm: (viii)}. Moreover, by~\ref{itm: (v)}, $$\delta_{\text{mon}}(G^*) \ge d - |W^*| \stackrel{\eqref{eq: bound on W^*}}{\ge} d - 2000 \delta^{-4}.$$ 

Consider a rainbow path $P$ in $G^*$ with $|V(P)| \ge 3$ and $S \subseteq V^*(G) \setminus V(P)$ such that $|V(P) \cup S| \le d/4$. We shall find a rainbow cycle $C^*$ in $G\setminus S$, which closes~$P$ with $\phi(P) \subseteq \phi(C^*) \subseteq \phi(P) \cup \phi(G)\setminus \phi(G^*)$. 

\begin{claim}
    There exists $j \in [i]$ and $B \in \mathcal{B}_1^j \cup \mathcal{B}_2^j$ such that $\rm{int}(P) \subseteq U_2^*(B)$.
\end{claim}
    \begin{proofclaim}
Let $v \in \rm{int}(P)\subseteq V^*(G^*) \subseteq V^*(G)$. Note that $V^*(G)=V^*(H^0) \supseteq \dots \supseteq V^*(H^i) = \emptyset$. Let $j \in [i]$ be such that $v \in V^*(H^j)$ but $v \notin V^*(H^{j+1})$. By~\ref{itm: (vii)}, there exists $B \in \mathcal{B}_1^j \cup \mathcal{B}_2^j$ such that $v \in U_2^*(B)$.   By~\ref{itm: (x)}, $B$ is $g$-maximal in $G^{j-1} - (\mathcal{B}^j \setminus B)$. Therefore Lemma~\ref{Lma: path}\ref{itm: lma: path: ii} implies that $\rm{int}(P) \subseteq U_2^*(B)$.   
\end{proofclaim}

Let $B$ and $j$ be as given by the claim. Let $P=x_1\dots x_{\ell}$ with $\phi(x_1x_2)=2$ and $\phi(x_{\ell -1}x_{\ell})=\ell$. Let $H' = H^{j-1} - (\mathcal{B}^j \setminus B)$.

\noindent\textbf{Case B(i):} $B \in \mathcal{B}_1^j$. Note that $\text{int}(P) \subseteq U_2^*(B) \subseteq V^*(H^{j-1})$. Moreover, $|\mbox{int}(P) \cup S| < d/4$. Let $U_2 = U_2(B)$, so $|U_2| \le 3d/4$ as $B \in \mathcal{B}_1^j$. By Lemma~\ref{Lma: path}\ref{itm: lma: path: i}, $d_{2, H'}(x_2, \overline{U_2}) < g.$  Then by~\ref{itm: (vi)}, $$d_{2, H'}(x_2, U_2) > \delta_{\text{mon}}^*\left(H^{j-1}\right) - g - |W^*| \stackrel{\eqref{eq: bound on W^*}}{\ge} d - 2g(j-1)-g - 2000\delta^{-4} \ge d/2.$$ Similarly, we have $d_{\ell, H'}(x_{\ell -1}, U_2)\ge d/2$. We deduce that
\begin{align*}
       &|N_{2, H'}(x_2, U_2) \cap N_{\ell, H'} (x_{\ell-1}, U_2) \setminus (S \cup V(P))| 
			\\ 
			\ge& d_{2, H'}(x_2, U_2)+ d_{\ell, H'} (x_{\ell-1}, U_2) - |U_2| - |S \cup V(P)| \\
        >&  d/2+d/2 -3d/4 -d/4 = 0. 
\end{align*}
 Pick $x \in (N_{2, H'}(x_2, U_2) \cap N_{\ell, H'} (x_{\ell-1}, U_2)) \setminus (S \cup V(P))$. Then $C^* = xx_2\cdots x_{\ell -1}x$ is a rainbow cycle with  $\phi(C^*)=\phi(P)$.

\noindent\textbf{Case B(ii):} $B \in \mathcal{B}_2^j$. There exists a bowtie $B' \in \mathcal{B}^{j-1}$ with $\partial B' = \{B\}$. Let $B=(v, C_1, W_1, C_2, W_2)$ and furthermore, $V^*(G^j) \cap U_2^*(B') = U_1^*(B) = U_2^*(B)$. Since $P$ is in~$G^*$, we have $ \rm{int}(P) \subseteq V^*(G^*)$. Note that $B$ is $g$-maximal in $H'$.
We deduce that
\begin{align*}
d_{2, H'}(x_2, U_1(B|H^{j-1}, \mathcal{B}^j)) &\ge \delta_{\text{mon}}^*(H^j) - d_{2, H'}(x_2, U_1(B))\\ & \stackrel{\mathclap{\ref{itm: (vi)},\ \rm{ Lemma  }~\ref{Lma: path}\ref{itm: lma: path: i}}}{\ge} d-2gj-g \ge d-(2i+1)g \\ &\stackrel{\mathclap{\eqref{eq: bound on i}, \eqref{eq: bound on W^*}}}{\ge} |S \cup V(P)|+|V(G)\setminus V(G^*)|. 
\end{align*}
Pick $y_1 \in U_1(B) \cap \left(V(G^*)\setminus (S \cup V(P)) \right)$ such that $\phi(y_1 x_2) = 2$. Similarly, we can find a distinct vertex $y_2 \in U_2(B) \cap \left(V(G^*)\setminus (S \cup V(P)) \right) $ such that $\phi(y_2 x_{\ell -1}) = \ell$.
Let $P_{{y_1}, v}$ be the rainbow path from $y_1$ to $v$ with $\text{int}(P_{{y_1}, v}) \subseteq W_1$ and $\phi(P_{y_1, v}) \subseteq C_1$; and $P_{v, {y_2}}$ be a rainbow path from~$y_2$ to~$v$ with $\text{int}(P_{ v, {y_2}}) \subseteq W_2$ and $\phi(P_{v, y_2}) \subseteq C_2$. Then $C^* = y_1P_{y_1, v}P_{v, y_2}y_2x_{\ell-1}x_{\ell-2}...x_2y_1$ is a rainbow cycle in $G\setminus S$. Note that $$\phi(P) \subseteq \phi(C^*) \subseteq \phi(P) \cup \phi(G)\setminus \phi(G^*).$$ This concludes the proof of the lemma. 
\end{proof}     

\section{Concluding Remarks}\label{Sec: Conclusion}
In this work, we showed that there exists a monochromatic tight cycle partition of any large~$r$-edge-coloured $K_n^{(k)}$ with a number of cycles that is a polynomial of~$r$. For $k \ge 3$, it is easy to construct an~$r$-edge-coloured hypergraph which needs at least~$r$ monochromatic tight cycles to partition all the vertices. We show such a construction below for completeness. 

\begin{proposition}
    Let $k, r \in \N$ such that $k \ge 2$. There exists an $r$-edge-coloured complete~$k$-graph such that at least~$r$ monochromatic cycles are required for a monochromatic tight cycle partition of its vertex set. 
\end{proposition}
\begin{proof}
    Let $H$ be a complete $k$-graph with vertex set partitioned as $V = \bigcup_{i \in [r]} V_i$, where $V_1 \neq \emptyset$, for distinct $i, j \in [r]$, $V_i \cap V_j = \emptyset$ and for each $i \in [r]\setminus \{1\}$, \begin{equation} \label{eqn: size distribution}
        |V_i| > (k-1) \sum_{j \in [i-1]} |V_j|.
    \end{equation} We colour each $e \in E(G)$ with colour $ \phi (e) =  \min \{i \in [r]: V_i \cap e \neq \emptyset\}$. Let~$\mathcal{C}$ be a set of vertex-disjoint monochromatic tight cycles that partition~$V(H)$. It is enough to show that for each $i \in [r]$, there is a tight cycle in~$\mathcal{C}$ coloured~$i$. Clearly, a monochromatic tight cycle $C$ that contains at least one vertex from $V_i$ satisfies $\phi(C) \le i$. Let
    $$\mathcal{C}^{<i} = \{C \in \mathcal{C}: V(C) \cap V_i \neq \emptyset \text{ and } \phi(C) < i\}.$$
It suffices to show that $V_i \nsubseteq V(\mathcal{C}^{<i})$. Note that each $C \in \mathcal{C}^{<i}$ satisfies $\left| V(C) \cap \bigcup_{j < i} V_j \right| \ge |V(C)|/k$ and therefore $|V(C) \cap V_i| \le (k-1)\left|V(C) \cap \bigcup_{j < i} V_j \right|$.
     We deduce that $$|V\left(\mathcal{C}^{<i}\right) \cap V_i| \le \sum_{C \in \mathcal{C}^{<i}} |V(C) \cap V_i| \le (k-1)\sum_{j \in [i-1]}|V_j| \stackrel{\eqref{eqn: size distribution}}{<} |V_i|.$$ This completes the proof of the proposition.
\end{proof}
When $k =2$ and $r \ge 3$, Pokrovskiy~\cite{MR3194196} showed that there exist infinitely many $r$-edge-coloured~$K_n$ which require at least $r+1$  monochromatic cycles. Lo and Pfenninger~\cite{MR4533706} showed that for any $k \ge 3$, there are $2$-edge-coloured $K_n^{(k)}$ that cannot be partitioned into two monochromatic tight cycles of distinct colours. It is interesting to consider whether the same holds for~$r$-edge-coloured complete $k$-graphs when $r, k \ge 3$. 

We believe that an~$r$-edge-coloured $K_n^{(k)}$ can be partitioned into a linear number of monochromatic tight cycles.

\begin{conjecture}
    Let $r, k \in \N$ with $r, k \ge 2$. Then there exists a constant $C = C(k)$ such that every $r$-edge-coloured $K_n^{(k)}$ can be partitioned into at most $Cr$ monochromatic tight cycles.
\end{conjecture}

If one can strengthen Lemma~\ref{Lma: how to absorb} to show that $f(r, k)$ monochromatic tight cycles suffice, then our proof method would imply $O(r\log r) + 2f(r, k) +3$ cycles partition any $r$-edge-coloured~$K_n^{(k)}$. The $O(r \log r)$ term corresponds to the number of monochromatic cycles required to cover almost all the vertices. 

\begin{conjecture}
    For all $r, k \ge 2$, there exists $C = C(k)$ such that any $r$-edge-coloured $K_n^{(k)}$ contains at most $Cr$ cycles covering $(1-o(1))n$ vertices. 
\end{conjecture}
Note that the case $k=2$ and general $r$ is still open. 

\subsection*{Acknowledgements}
The authors would like to thank the anonymous referee, Vincent Pfenninger, Andrea Freschi and Simona Boyadzhiyska for their comments on the draft and helpful discussions. The research leading to these results was supported by EPSRC, grant no. EP/V002279/1 (A.~Lo) and EP/V048287/1 (A.~Lo).
There are no additional data beyond that contained within the main manuscript.

\printbibliography   
\end{document}